\documentclass[oneside,final,11pt]{amsart}
\usepackage[T1]{fontenc}
\usepackage[latin9]{inputenc}
\usepackage{geometry}
\usepackage{graphicx}
\usepackage{soul}
\usepackage{xcolor}
\usepackage{amssymb}

\newtheorem{theorem}{Theorem}[section]
\newtheorem{lemma}[theorem]{Lemma}
\newtheorem{proposition}[theorem]{Proposition}

{ \theoremstyle{definition}
\newtheorem{definition}[theorem]{Definition}}
{ \theoremstyle{definition}
}
{ \theoremstyle{remark}
\newtheorem{remark}[theorem]{Remark}}
\usepackage{placeins}
\usepackage{cite}
\usepackage{caption}
\usepackage{enumerate}
\usepackage{afterpage}
\usepackage{enumitem}
\usepackage{bmpsize}

\title{Six-vertex models and the GUE-corners process}
\date{\today}
\author{Evgeni Dimitrov}

\begin{document}

\maketitle 

\begin{abstract}
In this paper we consider a class of probability distributions on the six-vertex model from statistical mechanics, which originate from the higher spin vertex models of \cite{BP}. We define operators, inspired by the Macdonald difference operators, which extract various correlation functions, measuring the probability of observing different arrow configurations. 
The development of our operators is largely based on the properties of a remarkable family of symmetric rational functions, which were previously studied in \cite{Bor14}. \\

For the class of models we consider, the correlation functions can be expressed in terms of multiple contour integrals, which are suitable for asymptotic analysis. For a particular choice of parameters we analyze the limit of the correlation functions through a steepest descent method. Combining this asymptotic statement with some new results about Gibbs measures on Gelfand-Tsetlin cones and patterns, we show that the asymptotic behavior of our six-vertex model near the boundary is described by the GUE-corners process.
\end{abstract}

\tableofcontents

\section{Introduction and main results}\label{Section1}

The exact formulation of our model and the main results of this paper are given in Section \ref{Section1.2}. The section below provides some literary context for our work.

%
\subsection{Preface}\label{Section1.1} The six-vertex model is a well-known exactly solvable lattice model of equilibrium statistical mechanics. The study of its properties is a rich subject, which has enjoyed many exciting developments during the last half-century (see, e.g., \cite{Bax}, \cite{Res10}, and the references therein). Fixing particular boundary conditions and weights, connects the six-vertex model to a number of combinatorial objects like alternating sign matrices and domino tilings \cite{EKLP}. The six-vertex model and certain higher spin generalizations of it have been linked to a large class of integrable probabilistic models that belong to the KPZ universality class in $1$ + $1$ dimensions - this was first observed in \cite{Gwa} and studied more recently in  \cite{BCG14,BP,CP15}. These recent advances have spurred new interest in vertex models and the development of tools to analyze them. 

The main subject of the paper is the (vertically inhomogeneous) six-vertex model in a half-infinite strip. We will work with a particular weight parametrization, introduced in \cite{Bor14}, whose origin lies in the Yang-Baxter equation, and which corresponds to the so-called {\em ferroelectric regime} \cite{Bax}. The partition function of this model is described by a remarkable family of symmetric rational functions $\mathsf{F}_\lambda$, parametrized by non-negative signatures $\lambda = \lambda_1 \geq \lambda_2 \geq ...\geq \lambda_N \geq 0$. These functions form a one-parameter generalization of the classical Hall-Littlewood polynomials \cite{Mac} and enjoy many of the same structural properties \cite{Bor14}. In a recent paper \cite{BP}, the authors derive many useful features of the functions $\mathsf{F}_\lambda$, which allow them to obtain integral representations for certain multi-point $q$-moments of the inhomogeneous higher spin six vertex model in infinite volume. Such formulas are well-known to be a fruitful source of asymptotic results and were recently utilized to study the asymptotics of various stochastic six-vertex models \cite{Amol1, Amol2, Bor16}.

In this paper we will develop a different approach to study the vertically inhomogeneous six-vertex model, which is based on a new class of operators $D^k_N$. These operators act diagonally on the functions $\mathsf{F}_\lambda$, whenever $\lambda$ has distinct parts and can be used to derive formulas for the probability of observing certain arrow configurations in different locations of the model. These observables were very recently investigated for the six-vertex model with domain wall boundary condition (DWBC) in \cite{CPS} under the name of {\em generalized emptiness formation probability} (GEFP). The derivation of the formulas in \cite{CPS} is based on the quantum inverse scattering method and heavily depends on the chosen boundary condition. On the other hand, our operators approach is more generic and can be applied to a much larger class of models. As discussed in \cite{CPS} the GEFP can be used to understand macroscopic frozen regions in the six-vertex model with DWBC and it is our hope that the operators we develop can be used to address similar questions for more general six-vertex models.

It is believed that the asymptotic behavior of the six-vertex model is similar to the {\em dimer models}, i.e. random lozenge tilings, plane partitions and domino tilings cf. \cite{Ke} (see also \cite{CLP, EKLP,J02,KOS}). One of the (conjectural) features of a large six-vertex model is the formation of the so-called limit shape or arctic curve - a well-investigated phenomenon in dimer models. The properties of the arctic curve for the DWBC six-vertex model were investigated in \cite{CP09}, \cite{CP10} and \cite{CPZJ}; for more general boundary conditions see \cite{Res10} and \cite{Z1}. To the author's knowledge, the exact form of the arctic curve is still conjectural (except in very special cases); however, it matches the numerical simulations in \cite{AR} and \cite{SZ}. The (conjectural) analogy between the six-vertex and dimer models further suggests that one should observe the GUE-corners process near the point separating two frozen regions in the limit shape. This is known to be the case for certain tiling problems on the plane: see, e.g., \cite{JN06}, \cite{Nor09} and \cite{OR}; and was also proved for the six-vertex model with DWBC under the uniform measure in \cite{Gor14}.

The main goal of this paper is to use the correlation functions obtained from our operators to analyze a particular class of homogeneous six-vertex models as the system size becomes large. There are two natural ways to understand the probability distributions that we analyze. On the one hand, one can view them as stochastic six-vertex models on the half-infinite strip with a particular choice of boundary data, which is related to a special class of symmetric functions, considered in \cite{BP}. Alternatively, these probabilities distributions describe the marginal law of a discrete time Markov process on vertex models, which is started from the stochastic six-vertex model of \cite{BCG14}, and whose dynamics is described by certain sequential update rules.
For the models we consider we show that certain configurations of holes (absence of arrows or empty edges) weakly converge to the GUE-corners process as the size of the system tends to infinity. We view the latter as the main result of this paper and the exact statement is given in Theorem \ref{theorem2}. The proof is based on the formulas obtained from our operators as well as a classification result, which identifies the GUE-corners process as the unique probability measure that satisfies the continuous Gibbs property (see Definition \ref{definitionGT}) and has the correct marginal distribution on the right edge.

One of the advantages of the model we consider is that it can be exactly sampled efficiently - the algorithm is described in Section \ref{Section8}. Based on simulations from this algorithm we believe that the six-vertex model we consider has a  limit-shape phenomenon - see Figure \ref{S1_5}. At this time, our methods do not seem to be strong enough to prove the existence of macroscopic frozen regions. The essential difficulty is that the relevant formulas we obtain to access the limit shape involve an increasing number of contour integrals, which makes the asymptotic analysis fairly hard. On the other hand, we can analyze our formulas when only finitely many contours are present using steepest descent arguments. This allows us to consider the special point separating two frozen regions and it is there that we identify the GUE-corners process. It would be very interesting to rigorously verify the arctic curve for the six-vertex model and prove the conjectural analogy with dimer models.

We now turn to describing our model and main results.

%
\subsection{Problem statement and main results}\label{Section1.2} For $N \in \mathbb{N}$ we let $\mathcal{P}_N$ denote the collection of $N$ up-right paths drawn in the sector $D_N :=\mathbb{Z}_{\geq 0} \times \{1,...,N\}$ of the square lattice, with all paths starting from a left-to-right arrow entering each of the points $\{(0,m): 1 \leq m \leq N\}$ on the left boundary and all paths exiting from the top boundary. We assume that no two paths are allowed to share a horizontal or vertical piece. For $\omega \in \mathcal{P}_N$ and $k = 1,...,N$ we let $\lambda^k(\omega) = \lambda^k_1 \geq \lambda^k_2 \geq \cdots \geq \lambda^k_k$ be the ordered $x$-coordinates of the intersection  points of $\omega$ with the horizontal line $y = k + 1/2$. Let $\mathsf{Sign}_k^+$ denote the set of signatures $\lambda$ of length $k$ with $\lambda_k \geq 0$, then $\lambda^k(\omega) \in \mathsf{Sign}_k^+$ for all $\omega \in \mathcal{P}_N$ and $k = 1,...,N$. The condition that no two paths share a horizontal piece, implies that $\lambda^k$ satisfy the interlacing property 
$$\lambda^{k+1}_1 \geq \lambda^{k}_1 \geq \lambda^{k+1}_{2} \geq \cdots \geq \lambda^k_k \geq \lambda^{k+1}_{k+1} \mbox{ for $k = 1,...,N-1$},$$
while the condition that no vertical pieces are shared implies $\lambda^k_1 > \lambda^k_2 > \cdots > \lambda^k_k$. See Figure \ref{S1_1}.\\

\begin{figure}[h]
\centering
\begin{minipage}{.5\textwidth}
  \centering
  \includegraphics[width=0.9\linewidth]{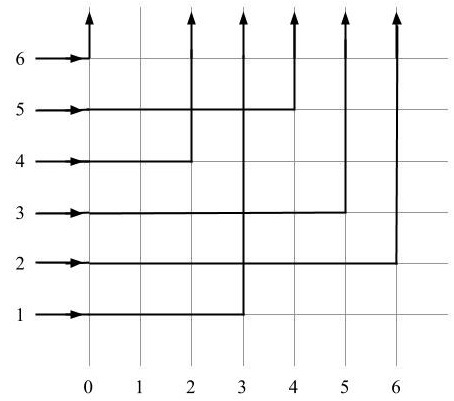}
\captionsetup{width=.9\linewidth}
\vspace{3.8mm}
  \caption{A path collection $\omega \in \mathcal{P}_N$ with $N = 6$. In this example $\lambda^4_1 = 6$, $\lambda^4_2 = 5$, $\lambda^4_3 = 3$ and $\lambda^4_4 = 2$}
\label{S1_1}	
\end{minipage}%
\begin{minipage}{.5\textwidth}
  \centering
\vspace{0mm}
\scalebox{0.5}{\includegraphics{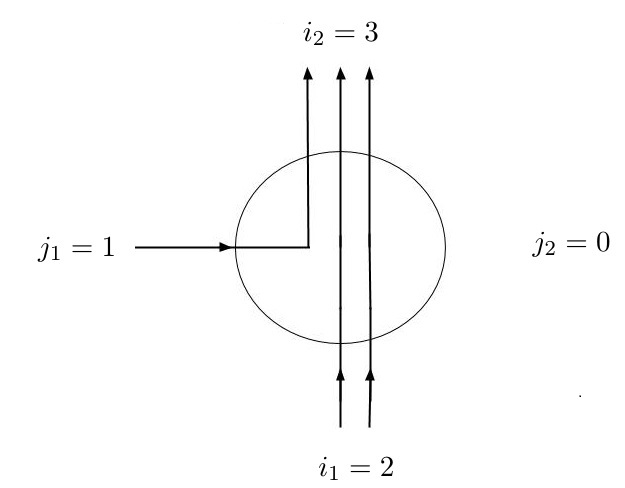}}
\captionsetup{width=.9\linewidth}
 \caption{Incoming and outgoing vertical and horizontal arrows at a vertex, denoted by $(i_1,j_1; i_2, j_2) = (2,1; 3, 0)$}
\label{S1_2}
\end{minipage}
\end{figure}

We encode arrow configurations at a vertex through the numbers $(i_1,j_1; i_2, j_2)$, representing the number of arrows coming from the bottom and left of the vertex, and leaving from the top and right, respectively (see Figure \ref{S1_2}). Let us fix a parameter $s$ and $N$ indeterminates $u_1,...,u_N$, called {\em spectral parameters}. For a spectral parameter $u$, we define the following vertex weights
\begin{equation}\label{VertW}
\begin{split}
w_{u}(0,0;0,0) = 1, \hspace{2mm}   &w_u(1,0;1,0) = \frac{1 - s^{-1}u}{1 - su}, \hspace{2mm} w_u(1,0;0,1) = \frac{(1 - s^2)u}{1 - su} \\
w_{u}(0,1;1,0) = \frac{1 - s^{-2}}{1 - su}, \hspace{2mm} &w_u(0,1; 0,1) = \frac{u - s}{1 - su}, \hspace{2mm} w_u(1,1;1,1) = \frac{u - s^{-1}}{1 - su},
\end{split}
\end{equation}
and set all other vertex weights to zero. The choice of the above parametrization is made after \cite{Bor14}, where higher spin versions of the above vertex weights were considered. Those weights depend on two parameters $s,q$ and they are closely related to the matrix elements of the higher spin $R$-matrix associated with $U_q(\widehat{sl_2})$. Formulas for the higher spin weights are present in (\ref{weights1}) later in the text, and those in (\ref{VertW}) are obtained by setting $q = s^{-2}$. Given $\omega \in \mathcal{P}_N$, we let $\omega(i,j)$ denote the arrow configuration at the vertex $(i,j) \in D_N$ and note that we have six possible arrow configurations for $\omega(i,j)$, corresponding to the weights in (\ref{VertW}).

In addition, let us consider a function $f: \mathsf{Sign}_N^+ \rightarrow \mathbb{R}$. With the above data we define the {\em weight} of a path configuration $\omega$ as
\begin{equation}\label{PathW}
\mathcal{W}^f(\omega):= \prod_{i = 0}^{\infty} \prod_{j = 1}^N w_{u_j}(\omega(i,j)) \times f(\lambda^N(\omega)).
\end{equation}
We observe that all but finitely many of $\omega(i,j)$ equal $(0,0;0,0)$, which by (\ref{VertW}) has weight $1$ and so the product in (\ref{PathW}) is a well-defined rational function. Suppose that for a certain choice of parameters and function $f$ the weights $\mathcal{W}^f(\omega)$ are non-negative, not all $0$ and their sum
$$Z^f := \sum_{\omega \in \mathcal{P}_N} \mathcal{W}^f(\omega) < \infty,$$
then we may define a probability measure on $\mathcal{P}_N$ through $\mathbb{P}^f(\omega) = \frac{\mathcal{W}^f(\omega)}{Z^f}$. The function $f$ can be interpreted as a condition for the top boundary of an arrow configuration on $D_N$, complementing the other boundary conditions of no arrows entering from the bottom, all arrows entering from the left and no arrows propagating to infinity on the right. For example, taking $f(\lambda)$ to be zero unless $\lambda_{N-i + 1} = i -1$ for $i = 1,...,N$ corresponds to the (vertically) inhomogeneous six-vertex model with domain wall boundary condition \cite{Kor82}.\\

As discussed in Section \ref{Section1.1}, one of the novelties of this paper is the development of particular operators $D^k_m$, which can be used to extract a set of observables for measures on $\mathcal{P}_N$ of the form $\mathbb{P}^f$ above. The operators $D_m^k$ are inspired by the Macdonald difference operators, which have been used successfully in deriving asymptotic statements about random plane partitions, directed polymers and particle systems \cite{BorCor, BorCor2, BCF, BCR, ED, FerVet }. To give an example, the first order operator $D^1_m$ acts on functions in $m$ variables and is given by
$$D_m^1 := \sum_{i = 1}^m \prod_{j = 1, j\neq i}^m\left( \frac{u_j - qu_i}{u_j - u_i} \frac{u_j - s}{u_j - sq}\right) \frac{1 - s^2}{1-su_i} T_{u_i,s},$$
where $T_{u_i,s}F(u_1,...,u_m) = F(u_1,...,u_{i-1},s,u_{i+1},...,u_m)$. The formula for the general operator $D_m^k$ is given in Definition \ref{DiffOp}.

As will be explained later in Sections \ref{Section2} and \ref{Section3}, the probability distribution $\mathbb{P}^f$ is related to certain symmetric functions $\mathsf{F}_{\lambda}(u_1,...,u_m)$, parametrized by non-negative signatures $\lambda$. The key property of $D^k_m$ is that they act diagonally on $\mathsf{F}_{\lambda}(u_1,...,u_m)$, whenever $\lambda$ has distinct parts and satisfy
$$D_m^k\mathsf{F}_\lambda(u_1,...,u_m) = {\bf 1}_{\{ \lambda_{m} = 0, \lambda_{m-1} = 1,..., \lambda_{m-k+1} = k-1\}}\mathsf{F}_\lambda(u_1,...,u_m).$$
The above relation is essentially sufficient to prove that for $1 \leq k \leq m \leq N$ we have
\begin{equation}\label{S1eqSpec}
\mathbb{P}^f \left( \{ \omega \in \mathcal{P}_N: \lambda^m_m(\omega) = 0,...,  \lambda^m_{m-k+1}(\omega) = k-1 \} \right) = \frac{D^k_m Z^f}{Z^f},
\end{equation}
where we remark that the partition function $Z^f$ is a function of the variables $u_1,...,u_N$ and $D^k_m$ acts on the first $m$ variables. In words, the above expresses the probability of observing $k$ vertical arrows going from $(m,i)$ to $(m+1,i)$ for $i = 0,...,k-1$, in terms of the partition function $Z^f$ and the result of $D^k_m$ acting on it. 

The validity of (\ref{S1eqSpec}) can be established for a fairly general class of boundary functions $f$; however, in order for the formula to be useful one needs to understand the action of our operators on the partition function $Z^f$.  For general boundary conditions the partition function may not have a closed form or the action of the operators might not be clear. One particular class of functions, on which $D^m_k$ act well are functions that have the product form $F(u_1,...,u_m) = \prod_{i = 1}^m g(u_i)$. Such functions are eigenfunctions for $D^m_k$ with eigenvalues expressed through $k$-fold contour integrals - see Lemmas \ref{LemmaSC} and \ref{LemmaMC}. Whenever a model has a partition function in such a form (this can be achieved by fixing appropriate boundary conditions $f$ and is the case for the models we study in this paper) our method leads to contour integral representations for the probabilities in  (\ref{S1eqSpec}). In general, such representations are useful for asymptotic analysis as one has a lot of freedom in deforming contours and using steepest descent methods.  \\

In what follows we write down the general form of a function $f$ that we will consider and explain the probabilistic meaning of this choice. Define 
$$\mathsf{G}^c_\lambda(\rho) :=  (-1)^N{\bf 1}_{\{n_0 = 0\}}\prod_{i = 1}^{\infty}{\bf 1}_{\{ n_i \leq 1 \}}\prod_{j = 1}^{N}\left(-s \right)^{\lambda_j},$$ 
where $\lambda = 0^{n_0}1^{n_1} \cdots $ is the multiplicative expression for $\lambda$ (see Section \ref{Section2.1}). For an $M$-tuple of real parameters $(v_1,...,v_M)$ we define $f$ as
$$f(\lambda) = \mathsf{G}^c_\lambda(\rho, v_1,...,v_M)= \sum_{ \mu \in \mathsf{Sign}^+_N} \mathsf{G}^c_\mu(\rho) \mathsf{G}^c_{\lambda / \mu}(v_1,...,v_M),$$
where $\mathsf{G}^c_{\lambda / \mu}(v_1,...,v_M)$ is given by Definition \ref{defG} below.

If $M = 0$ we have that $f(\lambda) = \mathsf{G}^c_\lambda(\rho)$ and
$$Z^f = \sum_{\omega \in \mathcal{P}_N} \mathcal{W}^f(\omega) = (s^{-2};s^{-2})_N \prod_{i = 1}^N \frac{1 - s^{-1} u_i}{1 - su_i},$$
where we recall that $(a;q)_n$ denotes the $q$-Pochhammer symbol and equals $(a;q)_n = (1 - a)(1 - aq)\cdots(1 - aq^{n-1})$.  The latter identity is understood as an equality of formal power series and was derived in \cite{BP}. Fixing $s > 1$ and $u_i > s$ has the effect that $\mathcal{W}^f(\omega)  \geq 0$ and that the above identity holds numerically as well. In particular, for this choice of $f$, we have a well-defined probability distribution $\mathbb{P}^f$ on $\mathcal{P}_N$. The latter measure is the (vertically) inhomogeneous stochastic six-vertex model (see Section 6.5 in \cite{BP}). Further setting $u_i = u > s$ for $i = 1,...,N$, one arrives at the stochastic six-vertex model of \cite{BCG14} (see also \cite{BP}).

Given the above discussion, one can understand $f(\lambda) = \mathsf{G}^c_{\lambda / \mu}(v_1,...,v_M)$ as a certain many-parameter generalization of the boundary function of the previous models. As will be explained in Section \ref{Section2} we have for this choice of $f$ that
$$Z^f = \sum_{\omega \in \mathcal{P}_N} \mathcal{W}^f(\omega) =  (s^{-2};s^{-2})_N \prod_{i = 1}^N \left( \frac{1 - s^{-1} u_i}{1 - su_i} \prod_{j = 1}^M \frac{1 - s^{-2}u_i v_j}{1 - u_iv_j}\right) ,$$
where the equality is in the sense of formal power series. As before, we set $s > 1$, $u_i > s$ and, in addition, assume $v_j > 0$ are such that $u_iv_j < 1$ for $i = 1,...,N$ and $j = 1,...,M$. Under these conditions one can show that  $\mathcal{W}^f(\omega)  \geq 0$ and the above identity holds numerically as well. In particular, for this choice of $f$, we have a well-defined probability distribution $\mathbb{P}^f$ on $\mathcal{P}_N$, denoted by $\mathbb{P}_{{\bf u,v}}$. This is the main probabilistic object we will study.

For $m = 0,...,M$ we let $\mathbb{P}_{{\bf u,v}_m}$ denote the above probability distribution, where ${\bf v}_m = (v_1,...,v_m)$. Then one can interpret the distribution $\mathbb{P}_{{\bf u,v}_m}$ as the time $m$ distribution of a Markov chain $\{X_m\}_{m = 0}^M$, whose dynamics is governed by sequential update rules. For more details and an exact formulation we refer the reader to Section \ref{Section8} below as well as Section 6 in \cite{BP}. For a pictorial description of how the configurations $X_m$ evolve as time increases see  Figure \ref{S1_3}. Our primary interest is in understanding the large-time behavior of $X_m$ and we investigate this by studying the measure $\mathbb{P}_{{\bf u,v}}$ as both $M$ and $N$ tend to infinity.
\begin{figure}[h]
\centering
\scalebox{0.4}{\includegraphics{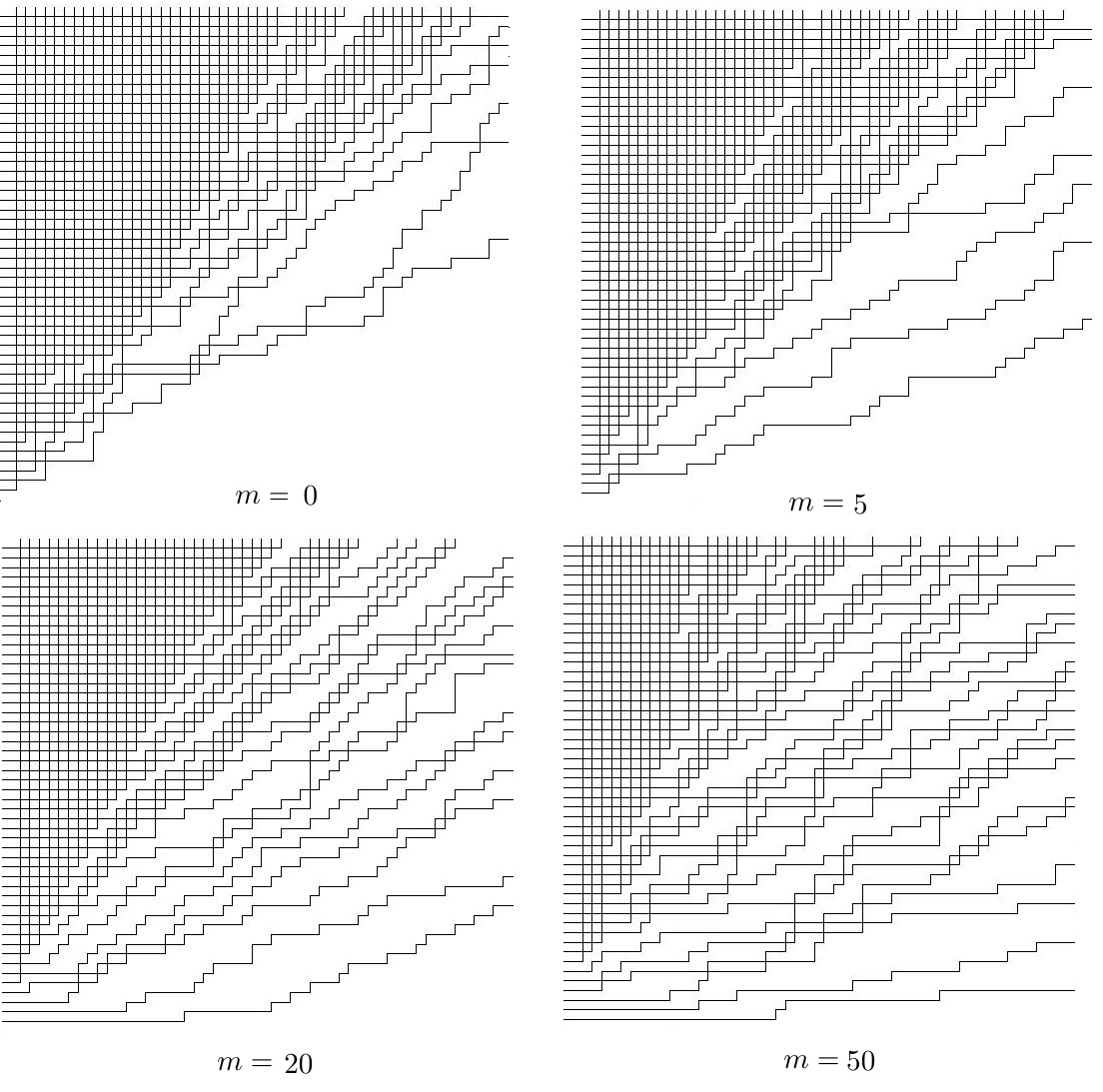}}
\caption{Random paths in $\mathcal{P}_N$, sampled from the Markov chain $X_m$ at times $m = 0,5,20$ and $50$. In this example $N = M =  50$, $s^{-2} = 0.7$, $u_i = u$ for $i = 1,...,N$ and $v_j = v$ for $j = 1,...,M$, where $u = 1.5$ and $v = 0.4$  }
\label{S1_3}
\end{figure}

 While most of the results we have in mind can readily be extended to more general parameter choices for ${\bf u}$ and ${\bf v}$ we keep our discussion simple and assume that $u_i = u$ for $i = 1,...,N$ and $v_j = v$ for $j = 1,...,M$. The resulting measure is denoted by $\mathbb{P}_{u,v}^{N,M}$ (the measure also depends on the parameter $s$ but we suppress it from the notation). The first result about this measure is the following.

\begin{theorem}\label{theorem1}
Suppose $s > 1$, $\frac{s + s^3}{2} > u > s$ and $v \in (0,u^{-1})$. Let $a = \frac{v^{-1} (u - s^{-1}) (u-s)}{u(v^{-1}-s^{-1})(v^{-1} - s)} > 0$ and suppose $\gamma > a$. Let $N(M) \geq \gamma\cdot M$ for all $M >> 1$ and consider the measure $\mathbb{P}_{u,v}^{N,M}$ on $\mathcal{P}_N$, defined above. Then for every $k \in \mathbb{N}$, we have that 
\begin{equation}\label{main0}
\lim_{M \rightarrow \infty} \mathbb{P}_{u,v}^{N,M} \left( \{ \omega \in \mathcal{P}_N : \lambda^N_{N - i + 1}(\omega) = i, 1\leq i \leq k \} \right) = 1.
\end{equation}
\end{theorem}

\begin{remark}
We choose $s > 1$ and $u > s$ to ensure non-negativity of the weights defining $\mathbb{P}_{u,v}^{N,M}$. This choice of parameters lands our Gibbs measure in the {\em ferroelectric regime} of the six-vertex model \cite{Bax} and $s > 1$ covers the entire range of the ferroelectric region - see also the discussion in Section \ref{Section6.1}. One requires $v \in (0, u^{-1})$ in order to ensure finiteness of the partition function and non-negativity of the weights. The condition $\frac{s + s^3}{2} > u$ is technical and assumed in order to simplify some arguments later in the text.
\end{remark}

Informally, Theorem \ref{theorem1} states that the probability $\mathbb{P}_{u,v}^{N,M}$ concentrates on path configurations, which have outgoing vertical arrows at locations $(i, N)$ for $i = 1,..., k$, where $k \in \mathbb{N}$ is fixed but arbitrary, and no such arrow at $(0,N)$. Let us consider such a path configuration $\omega$ and denote by $A_j = \{ (j, i) : i = 1,...,N\}$ the vertical slice of $D_N$ at location $j$. We observe that the left and bottom boundary conditions on $\omega$ imply that there are exactly $N$ arrows going into the set $A_0$ and no vertical outgoing arrow from $(0, N)$. The conservation of arrows over the region $A_0$, implies that all $N$ arrows must leave from the right boundary of $A_0$, and so each arrow that enters $(0,i)$ must continue horizontally (see Figure \ref{S1_4}). When we consider $A_1$, we see that there are still $N$ arrows going in, however, one arrow leaves at $(1,N)$ and so the conservation of arrows implies that there are $N-1$ arrows leaving $A_1$ to the right and entering $A_2$. In general, there will be $N- j + 1$ arrows going into region $A_j$ and one arrow leaving from the top, implying that there are $N-j$ arrows leaving from the right and entering $A_{j+1}$. Let us denote by $Y^j_1 < Y^j_2 < \cdots < Y_j^j$, the ordered vertical positions of the $j$ vertices in $A_j$, that have no outgoing horizontal arrow (alternatively, the vertical coordinates of the empty horizontal edges between $A_j$ and $A_{j+1}$) - see Figure \ref{S1_4}. A direct consequence of the up-right path direction, implies that $Y_i^j$ satisfy the interlacing property
$$Y^{j+1}_1 \leq Y^{j}_1 \leq Y^{j+1}_{2} \leq \cdots \leq Y^j_j \leq Y^{j+1}_{j+1} \mbox{ for $j = 1,...,k-1$}.$$
The above definition can readily be extended to $\omega \in \mathcal{P}_N$, which do not satisfy the condition $\lambda^N_{N - i + 1}(\omega) = i, 1\leq i \leq k$ as follows. We set $Y^j_i$ to be the $i$-th smallest $y$-coordinate of a vertex in $A_j$ with no horizontal outgoing arrow, or $Y^j_i = +\infty$ if the number of such vertices is less than $i$. In this way, we obtain an extended random vector $Y(N, M;k)(\omega) := (Y_i^j)_{1 \leq i \leq j; 1\leq j \leq k} \in (\mathbb{N}\cup \{\infty\})^{\frac{k(k+1)}{2}}$. The statement of Theorem \ref{theorem1} is that with probability going to $1$, the interlacing array $Y(N, M;k)(\omega)$ is actually finite. \\

\vspace{-4mm}
\begin{figure}[h]
\centering
\scalebox{0.70}{\includegraphics{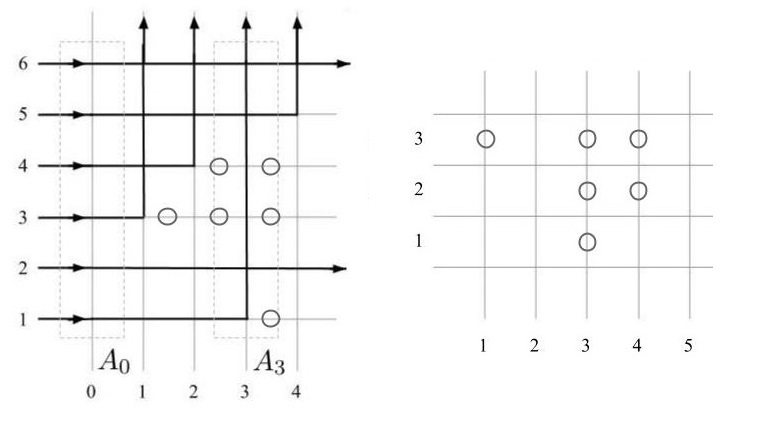}}
\caption{The left figure shows a path collection $\omega$, such that $\lambda^N_{N-i+1}(\omega) = i$ for $i = 1,...,k$ with $N = 6$ and $k = 3$. Circles indicate the positions of the empty edges. The right figure shows the array $(Y_i^j)_{1 \leq i \leq j \leq 3}$; $j$ varies vertically and position is measured horizontally. In this case  $Y_1^1 = Y_1^2 = Y_2^3 = 3$, $Y_2^2 = Y_3^3 = 4$, $Y_1^3= 1$.}
\label{S1_4}
\end{figure}

Recall that the Gaussian Unitary Ensemble (GUE) of rank $k$ is the ensemble of random Hermitian matrices $X = \{X_{ij}\}_{i,j = 1}^k$ with probability density (proportional to) $\exp(-Trace(X^2)/2)$, with respect to Lebesgue measure. For $r = 1,..., k$  we let $\lambda_1^r  \leq \lambda_2^r \leq \cdots \leq \lambda_r^r$ denote the eigenvalues of the top-left $r \times r$ corner $\{X_{ij}\}_{i,j = 1}^r$. The joint distribution of $\lambda_i^j$ $i = 1,...,j$, $j = 1,...,k$ is known as the {\em GUE-corners} process of rank $k$ (sometimes called the GUE-minors process).
The following theorem is the main result of this paper.

\begin{theorem}\label{theorem2}
Assume the same notation as in Theorem \ref{theorem1}, put $q = s^{-2}$ and fix $k \in \mathbb{N}$. Consider the sequence $Y(N, M;k)(\omega)$ with $\omega$ distributed according to $\mathbb{P}_{u,v}^{N,M}$. Then the random vectors
\begin{equation}\label{main1}
 \frac{1}{c\sqrt{M}}\left( Y(N, M;k) - aM \cdot {\bf 1}_{\frac{k(k+1)}{2}} \right)
\end{equation}
converge weakly to the GUE-corners process of rank $k$ as $M \rightarrow \infty$. In the above equation ${\bf 1}_K$ is the vector of $\mathbb{R}^K$ with all entries equal to $1$ and $c = (2a_2)^{1/2}b_1^{-1}$, with
$$a_2 = \frac{(1-q)v^{-1}}{(v^{-1}-s)(v^{-1}-sq)} \left[ \frac{(q+1)s - 2v^{-1}}{(v^{-1}-s)(v^{-1}-sq)} - \frac{(q+1)s - 2u}{(u-s)(u-sq)} \right] \mbox{ and }  b_1 = \frac{1}{u-s} - \frac{1}{q^{-1}u - s}.$$
\end{theorem}

Results similar to Theorem \ref{theorem2} are known for models of random Young diagrams and random tilings, see \cite{Bar01, JN06, Nor09, OR}. Moreover, for random lozenge tilings the GUE-corners process is believed to be a universal scaling limit near the point separating two frozen regions (also called a {\em turning point}) \cite{JN06,OR}. We believe, although we cannot prove, that in our model the GUE-corners process also appears near the point separating two frozen regions. At this time, our methods do not seem to be strong enough to verify a limit-shape phenomenon; however, simulation results seem to indicate that this is indeed the case. Figure \ref{S1_5} illustrates randomly sampled elements according to $\mathbb{P}_{u,v}^{N,M}$ with different parameter choices and we refer the reader to Section \ref{Section8} for a detailed description of the sampling algorithm. As can be seen on Figure \ref{S1_5}, there is a macroscopic frozen region, made of $(0,1;0,1)$ vertices in the bottom left corner and another one, made of $(1,1;1,1)$ vertices in the top left corner. The two regions are separated by a disordered region containing all six types of vertices.

\begin{figure}[h]
\centering
\scalebox{1}{\includegraphics{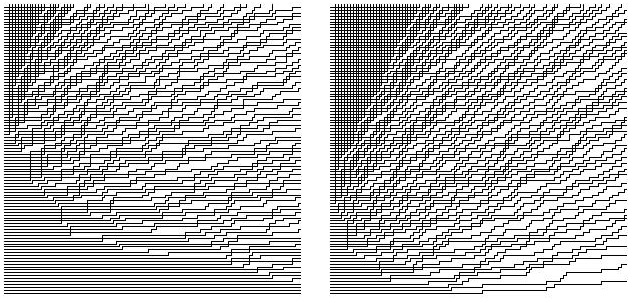}}
\caption{Random paths in $\mathcal{P}_N$, sampled according to $\mathbb{P}_{u,v}^{N,M}$ with $N = M = 100$. For the left picture $s^{-2}  = 0.8$, $u = 1.2$ and $v = 0.8$; for the right $s^{-2}  = 0.5$, $u = 1.5$ and $v = 0.6$.  }
\label{S1_5}
\end{figure}

Recently, \cite{Gor14} proved the equivalent of Theorem \ref{theorem2} for the six-vertex model with DWBC under the uniform distribution, and conjectured that a similar statement should hold for a larger set of measures. In this sense, Theorem \ref{theorem2} establishes such an extension to a more general class of six-vertex models with prescribed boundary data. \\

\begin{figure}[h]
\centering
\scalebox{1}{\includegraphics{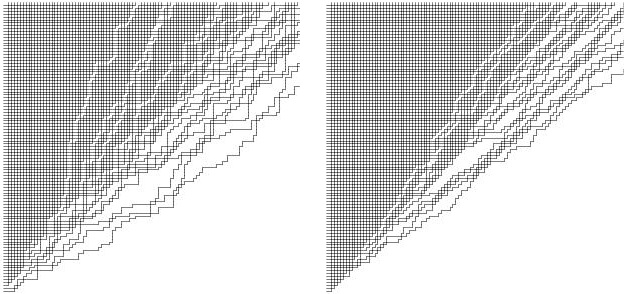}}
\caption{Random paths in $\mathcal{P}_N$, sampled according to $\mathbb{P}^f$ with $N = 100$ when $f(\lambda) = \mathsf{G}^c_\lambda(\rho)$. For the left picture $s^{-2} = 0.8$, $u = 1.5$; for the right $s^{-2} =  0.5$, $u = 2$  }
\label{S1_6}
\end{figure}

\begin{figure}[h]
\centering
\scalebox{0.4}{\includegraphics{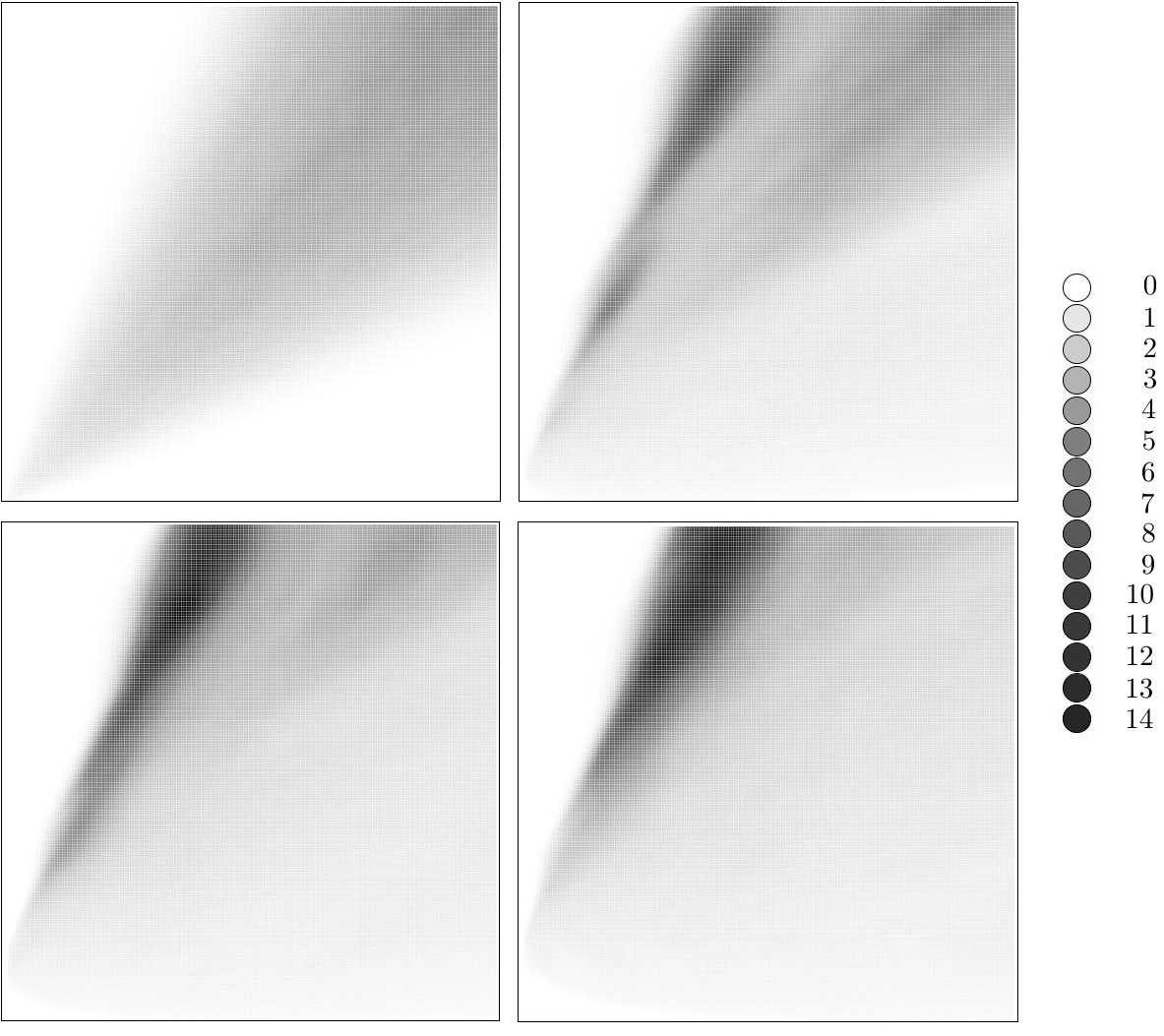}}
\caption{Variance of the height function at different locations for $2000$ samples from $\mathbb{P}_{u,v}^{N,M}$ . For the above simulations $s^{-2} = 0.5, u = 3,v = 0.2$ and $N = 200$ and $M = 0, 30, 60$ and $90$ for the top-left, top-right, bottom-left and bottom-right diagrams respectively. The variance-to-shade correspondence is indicated on the right.   }
\label{S1_7}
\end{figure}

Recall that one way to interpret the measure $\mathbb{P}_{u,v}^{N,m}$ is as the time $m$ distribution of a certain discrete time Markov chain, which at time $0$ is distributed as the stochastic six-vertex model of \cite{BCG14}. In \cite{BCG14} it was shown that configurations sampled from $\mathbb{P}_{u,v}^{N,0}$ converge to a certain deterministic cone-like limit shape (see Figure \ref{S1_6} for sample simulations). Comparing Figures \ref{S1_5} and \ref{S1_6}, we see that the stochastic dynamics has lead to a change in the limit shape. What is remarkable is that Theorem \ref{theorem2} indicates that the bulk fluctuations change as well. For the stochastic six-vertex model it is known that the fluctuations of the height function\footnote{ The height function $h(x,y)$ of the six-vertex model is defined as the number of paths that cross the horizontal line through $y$ to the right or at the point $x$.}  in the bulk are governed by the GUE Tracy-Widom distribution \cite{BCG14}. On the other hand, the bulk fluctuations of the GUE-corners process are described by the Gaussian Free Field (GFF) \cite{Bor10}. Theorem \ref{theorem2} suggests that the stochastic dynamics has transformed height fluctuations from KPZ-like to GFF-like. 

A possible explanation of the above phenomenon was suggested to us by Alexei Borodin and Fabio Toninelli and goes as follows. At large times one has both KPZ and GFF statistics within the model, but they manifest themselves in different portions of the configurations. As path configurations evolve, the KPZ region is pushed away from the origin and in its place GFF statistics emerge. We motivate the latter explanation with some simulations in Figure \ref{S1_7}. One distinguishing feature between KPZ and GFF statistics is the order of growth of the fluctuations, which are algebraic in the former and logarithmic in the latter case. We expect that the variance of the height function in the KPZ region to be of order $N^{2/3}$, while in the GFF region to be of order $\log (N) $. The latter implies that we can use the height variance as a proxy for distinguishing the different regions in our model and the results are presented in Figure \ref{S1_7}. As can be seen, there is indeed a high-variance cone, which is moving away from the origin and a very low variance region takes its place. It would be very interesting to verify that both GFF and KPZ fluctuations coexist in our model, since to our knowledge such a phenomenon has not been observed in other settings.\\

We end this section by briefly outlining the key ideas that go into proving Theorem \ref{theorem2}. The first key observation is that for $\omega \in \mathcal{P}_N$ one has $\lambda^m_m(\omega) = 1,...,  \lambda^m_{m-k+1}(\omega) = k $ if and only if $Y_k^k(\omega) \leq m$. Using this observation and our operators $D^k_m$, we express $\mathbb{P}_{{\bf u}, {\bf v}}(Y_k^k \leq m)$ and more generally $\mathbb{P}_{{\bf u}, {\bf v}}( Y_1^1 \leq m_1, ..., Y_k^k \leq m_k)$ in terms of certain $k$-fold contour integrals. These formulas for the joint cumulative distribution functions (CDFs) of the random vector $(Y_1^1,...,Y_k^k)$ are suitable for asymptotic analysis and can be used to show that under the translation and rescaling of Theorem \ref{theorem2}, this vector converges weakly to $(\lambda_1^1,...,\lambda_k^k)$, where $\lambda_i^j$ $i = 1,...,j$, $j = 1,...,k$ is the GUE-corners process of rank $k$. Using the six-vertex Gibbs property (see Section \ref{Section6.2}) and our convergence result for $(Y_1^1,...,Y_k^k)$, we show that the sequence of random interlacing arrays $Y(N, M;k)$ under the translation and rescaling of Theorem \ref{theorem2} is tight and any subsequential limit satisfies the continuous Gibbs property (see Definition \ref{definitionGT}). The final ingredient, in the proof is a classification result, which identifies the GUE-corners process as the unique probability measure on interlacing arrays that satisfies the continuous Gibbs property and has the correct distribution on the right edge. This shows that any weak subsequential limit of $Y(N, M;k)$ is in fact the GUE-corners process of rank $k$, which together with tightness proves Theorem \ref{theorem2}.

%
\subsection{Outline and acknowledgements} The introductory section above formulated the problem statement and gave the main results of the paper. In Section \ref{Section2} we study the measure $\mathbb{P}_{{\bf u, v}}$ and derive formulas for its finite dimensional distributions. In Section \ref{Section3} we define our operators $D^k_m$ and prove several of their properties. In Section \ref{Section4} we obtain integral formulas for the probabilities $\mathbb{P}_{u,v}^{N,M}( Y_1^1 \leq m_1, ..., Y_k^k \leq m_k)$, study their asymptotics and prove Theorem \ref{theorem1}. In Section \ref{Section5} we study probability measures on Gelfand-Tsetlin cones, which satisfy the continuous Gibbs property. In Section \ref{Section6} we study probability measures on Gelfand-Tsetlin patterns, which satisfy what we call the six-vertex Gibbs property. The proof of Theorem \ref{theorem2} is given in Section \ref{Section7}. In Section \ref{Section8} we describe an exact sampling algorithm for the measure $\mathbb{P}_{{\bf u}, {\bf v}}$. \\

I wish to thank my advisor, Alexei Borodin, for suggesting this problem to me and for his continuous help and guidance. I also thank Vadim Gorin for numerous helpful discussions.

%
\section{Measures on up-right paths } \label{Section2} 
In this section we provide some results about $\mathbb{P}_{{\bf u, v}}$. In particular, we show that it arises as a limit of measures on non-negative signatures, studied in Section 6 of \cite{BP}, and is a well-defined probability measure on the set of oriented up-right paths drawn in the region $D_N= \mathbb{Z}_{\geq 0} \times \{1,...,N\}$. We also provide explicit formulas for its marginal distributions.  In what follows we adopt the notation from \cite{BP} and summarize some of the results from the same paper.

\subsection{Symmetric rational functions}\label{Section2.1} We start by introducing some necessary notation. A {\em signature} of length $N$ is a sequence $\lambda = (\lambda_1 \geq \lambda_2 \geq \cdots \geq \lambda_N), \lambda_i \in \mathbb{Z}$. The set of all signatures of length $N$ is denoted by $\mathsf{Sign}_N$, and $\mathsf{Sign}^+_N$ is the set of signatures with $\lambda_N \geq 0$. We agree that $\mathsf{Sign}_0 = \mathsf{Sign}^+_0$ consists of the single empty signature $\varnothing$ of length $0$. We also denote by $\mathsf{Sign}^+:= \sqcup_{N \geq 0} \mathsf{Sign}^+_N$ the set of all non-negative signatures. An alternative representation of a signature $\mu \in \mathsf{Sign}^+$ is through the multiplicative notation $\mu = 0^{m_0} 1^{m_1} 2^{m_2} \cdots$, which means that $m_j = |\{i : \mu_i = j\}|$ is the number of parts in $\mu$ that are equal to $j$ (also called {\em multiplicity} of $j$ in $\mu$). We also recall the $q$-Pochhammer symbol $(a;q)_n := (1 - a)(1-aq) \cdots (1-aq^{n-1})$.\\

In what follows, we want to define the {\em weight} of a finite collection of up-right paths in some region $D$ of $\mathbb{Z}^2$, which will be given by the product of weights of all vertices that belong to the path collection. Throughout this paper we will always assume that the weight of an empty vertex is $1$ and so alternatively the weight of a path configuration can be defined through the product of the weights of all vertices in $D$. Figures \ref{S1_1} and \ref{S1_3} give examples of collections of up-right paths, see also Figure \ref{S2_1} below.

The configuration at a vertex is determined by four numbers $(i_1, j_1; i_2, j_2)$, representing the number of arrows that enter the vertex from below and right, and that leave from the top and left respectively (see Figure 2). Vertex weights are thus functions of those four variables. We postulate that a configuration $(i_1, j_1; i_2, j_2)$ must satisfy $i_1, j_1, i_2, j_2 \geq 0$, $j_1, j_2 \in \{0,1\}$ and $i_1 + j_1 = i_2 + j_2$ (otherwise its weight is $0$).

We will consider two sets of special vertex weights. They are both defined through two parameters $s,q$ (which are fixed throughout this section) as well as an additional {\em spectral parameter} $u$. We assume all parameters are generic complex numbers, and for the most part ignore possible singularities of the expressions below. The first set of vertex weights is explicitly given by 
\begin{equation}\label{weights1}
\begin{split}
&w_u(g,0;g,0) = \frac{1 - sq^gu}{1-su}, \hspace{20mm} w_u(g+1,0; g,1) = \frac{(1-s^2q^g)u}{1-su}\\
&w_u(g,1;g,1) = \frac{u - sq^g}{1-su}, \hspace{22mm} w_u(g,1;g+1,0) = \frac{1 - q^{g+1}}{1-su},
\end{split}
\end{equation}
where $g$ is any non-negative integer. All other weights are assumed to be zero. We also define the following {\em conjugated} vertex weights
\begin{equation}\label{weights2}
\begin{split}
&w^c_u(g,0;g,0) = \frac{1 - sq^gu}{1-su}, \hspace{20mm} w^c_u(g+1,0; g,1) = \frac{(1-q^{g+1})u}{1-su}\\
&w^c_u(g,1;g,1) = \frac{u - sq^g}{1-su}, \hspace{22mm} w^c_u(g,1;g+1,0) = \frac{1 - s^2q^{g}}{1-su},
\end{split}
\end{equation}
where as before $g \in \mathbb{Z}_{\geq 0}$ and all other weights are zero. We remark that the weights are non-zero only if $j_1, j_2 \in \{0,1\}$, which implies that the multiplicity of the horizontal edges is bounded by $1$. For more background and motivation for this particular choice of weights we refer the reader to Section 2 of \cite{BP}.\\

Let us fix a number $n \in \mathbb{N}$, $n$ indeterminates $u_1,...,u_n$ and the region $D_n = \mathbb{Z}_{\geq 0} \times \{1,...,n\}$. Let $\omega$ be a finite collection of up-right paths in $D_n$, which end in the top boundary, but are allowed to start from the left or bottom boundary of $D_n$. By $\omega(i,j)$ we denote the arrow configuration of the vertex at location $(i,j) \in D_n$. Then the weight of $\omega$ with respect to the two sets of weights above is defined by
$$\mathcal{W}(\omega) = \prod_{i = 0}^\infty \prod_{j = 1}^n w_{u_j}(\omega(i,j)), \hspace{20mm} \mathcal{W}^c(\omega) = \prod_{i = 0}^\infty \prod_{j = 1}^n w^c_{u_j}(\omega(i,j)).$$
We notice that by (\ref{weights1}) and (\ref{weights2}) $w_u(0,0;0,0) = 1 = w_u^c(0,0;0,0)$ and since all but finitely many vertices are empty, the products above are in fact finite. With the above notation we define the following partition functions.

\begin{definition}\label{defG}
Let $N,n \in \mathbb{Z}_{\geq 0}$, $\lambda, \mu \in \mathsf{Sign}_N^+$ and $u_1,...,u_n \in \mathbb{C}$ be given. Let $\mathcal{P}^c_{\lambda/\mu}$ be the collection of up-right paths $\omega$, which
\begin{itemize}
\item start with $N$ vertical edges $(\mu_i,0) \rightarrow (\mu_i,1)$, $i = 1,...,N$;
\item end with $N$ vertical edges $(\lambda_i, n) \rightarrow (\lambda_i, n+1)$, $i = 1,...,N$.
\end{itemize}
Then we define
$$\mathsf{G}^c_{\lambda / \mu}(u_1,...,u_n) := \sum_{\omega \in \mathcal{P}^c_{\lambda/\mu}} \mathcal{W}^c(\omega).$$
\end{definition}
{\raggedleft We will also use the abbreviation $\mathsf{G}^c_\lambda$ for $\mathsf{G}^c_{\lambda / (0,0,...,0)}$. For the second set of weights we have a similar definition.}

\begin{definition}\label{defF}
Let $N,n \in \mathbb{Z}_{\geq 0}$, $\mu \in \mathsf{Sign}_N^+$, $\lambda \in \mathsf{Sign}^+_{N+n}$ and $u_1,...,u_n \in \mathbb{C}$ be given. Let $\mathcal{P}_{\lambda/\mu}$ be the collection of up-right paths, which
\begin{itemize}
\item start with $N$ vertical edges $(\mu_i,0) \rightarrow (\mu_i,1)$, $i = 1,...,N$ and with $n$ horizontal edges $(-1,y) \rightarrow (0,y)$, $y = 1,...,n$;
\item end with $N +n$ vertical edges $(\lambda_i,n) \rightarrow (\lambda_i,n+1)$, $i = 1,...,N+n$.
\end{itemize}
Then we define
$$\mathsf{F}_{\lambda / \mu}(u_1,...,u_n) := \sum_{\omega \in \mathcal{P}_{\lambda/\mu}} \mathcal{W}(\omega).$$
\end{definition}
{\raggedleft We will also use the abbreviation $\mathsf{F}_\lambda = \mathsf{F}_{\lambda/\varnothing}$. Path configurations that belong to $\mathcal{P}_{\lambda/\mu}$ and $\mathcal{P}^c_{\lambda/\mu}$ are depicted in Figure \ref{S2_1}.} \\

\vspace{-4mm}
\begin{figure}[h]
\centering
\scalebox{0.6}{\includegraphics{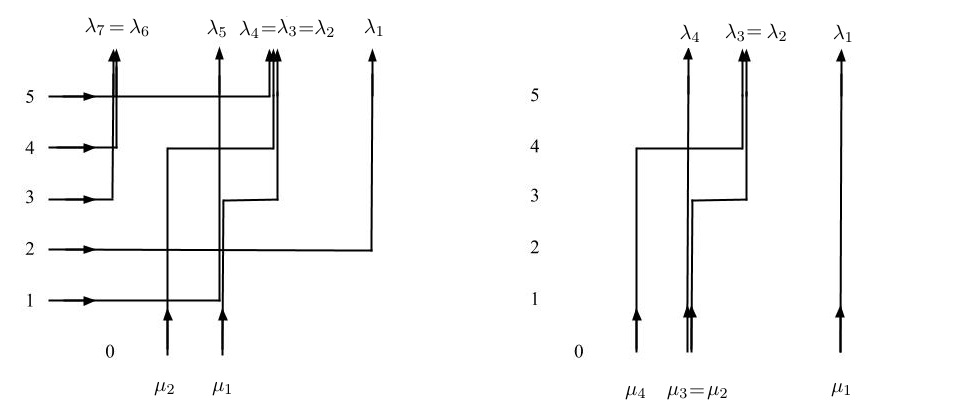}}
\caption{Path collections belonging to $\mathcal{P}_{\lambda/\mu}$ (left) and $\mathcal{P}_{\lambda/\mu}^c$ (right).}
\label{S2_1}
\end{figure}

In the definitions above we define the weight of a collection of paths to be $1$, if it has no interior vertices. Also, the weight of an empty collection of paths is $0$. We now summarize some of the properties of the functions $\mathsf{G}^c_{\lambda/ \mu}$ and $\mathsf{F}_{\lambda/\mu}$ in a sequence of propositions; see Section 4 of \cite{BP} for details.

\begin{proposition}\label{shifting}
Let $N,n,k \in \mathbb{Z}_{\geq 0}$, $\mu \in \mathsf{Sign}_N^+$, $\lambda \in \mathsf{Sign}^+_{N+n}$ and $u_1,...,u_n \in \mathbb{C}$ be given. Suppose $\mu_N \geq k$ and $\lambda_{N+n} \geq k$, and denote by $\mu - (k)^N$ and $\lambda - (k)^{N+n}$ the signatures with parts $\mu_i - k$ and $\lambda_i - k$ respectively. Then we have
\begin{equation}
\mathsf{F}_{\lambda / \mu}(u_1,...,u_n)  = \left( \prod_{i = 1}^n \frac{u_i - s}{1 - su_i} \right)^k \mathsf{F}_{\lambda - (k)^{N+n} / \mu - (k)^N}(u_1,...,u_n).
\end{equation}
\end{proposition}

\begin{proposition}\label{symmetricFun}
The functions $\mathsf{F}_{\lambda/\mu}(u_1,...,u_n)$ and $\mathsf{G}^c_{\lambda/\mu}(u_1,...,u_n)$ defined above are rational symmetric functions in the variables $u_1,...,u_n$.
\end{proposition}

\begin{proposition}\label{Branching}
{\bf 1.} For any $N, n_1, n_2 \in \mathbb{Z}_{\geq 0}$, $\mu \in \mathsf{Sign}_N^+$ and $\lambda \in \mathsf{Sign}_{N + n_1 + n_2}^+$, one has
\begin{equation}\label{branchF}
\mathsf{F}_{\lambda / \mu}(u_1,...,u_{n_1+ n_2}) = \sum_{ \kappa \in \mathsf{Sign}_{N + n_1}^+ }\mathsf{F}_{\lambda / \kappa}(u_{n_1 + 1},...,u_{n_1 + n_2}) \mathsf{F}_{\kappa / \mu}(u_1,...,u_{n_1}).
\end{equation}
{\bf 2.} For any $N, n_1, n_2 \in \mathbb{Z}_{\geq 0}$ and $\lambda, \mu \in \mathsf{Sign}_{N}^+$, one has
\begin{equation}
\mathsf{G}^c_{\lambda/\mu}(u_1,...,u_{n_1 + n_2}) = \sum_{ \kappa \in \mathsf{Sign}_{N }^+ }\mathsf{G}^c_{\lambda / \kappa}(u_{n_1 + 1},...,u_{n_1 + n_2}) \mathsf{G}^c_{\kappa / \mu}(u_1,...,u_{n_1}).
\end{equation}
\end{proposition}
{\raggedleft The properties of the last proposition are known as {\em branching rules}.}

\begin{definition}
We say that two complex numbers $u,v \in \mathbb{C}$ are {\em admissible} with respect to the parameter $s$ if $\left| \frac{u - s}{1 - su} \cdot \frac{v -s }{1 - sv}  \right| < 1.$
\end{definition}

\begin{proposition}\label{PPieri} Let $u_1,...,u_N$ and $v_1,...,v_K$ be complex numbers such that $u_i,v_j$ are admissible for all $i = 1,...,N$ and $j = 1,...,K$. Then for any $\lambda, \nu \in \mathsf{Sign}^+$ one has
\begin{equation}\label{CauchyId}
\sum_{\kappa \in \mathsf{Sign}^+} \hspace{-3mm} \mathsf{G}^c_{\kappa / \lambda} (v_1,...,v_K)\mathsf{F}_{\kappa/\nu}(u_1,...,u_N) = \prod_{i = 1}^N \prod_{j = 1}^K \frac{1 - qu_iv_j}{1 - u_iv_j}\hspace{-3mm} \sum_{ \mu \in \mathsf{Sign}^+} \hspace{-3mm}\mathsf{F}_{\lambda /\mu}(u_1,...,u_N) \mathsf{G}^c_{\nu / \mu} (v_1,...,v_K).
\end{equation}
\end{proposition}
\begin{remark}
Equation (\ref{CauchyId}) is called the {\em skew Cauchy identity} for the symmetric functions $\mathsf{F}_{\lambda / \mu}$ and $\mathsf{G}^c_{\lambda/ \mu}$ because of its similarity with the skew Cauchy identities for Schur, Hall-Littlewood, or Macdonald symmetric functions \cite{Mac}. The sum on the right-hand side (RHS) of (\ref{CauchyId}) has finitely many non-zero terms and is thus well-defined. The left-hand side (LHS) can have infinitely many non-zero terms, but part of the statement of the proposition is that if the variables are admissible, then this sum is absolutely converging and numerically equals the right side.
\end{remark}
A special case of (\ref{CauchyId}), when $\lambda = \varnothing$ and $\nu = (0,0,...,0)$ leads us to the {\em Cauchy identity}
\begin{equation}\label{CauchyS2}
\sum_{ \nu \in \mathsf{Sign}_N^+} \mathsf{F}_\nu(u_1,...,u_N) \mathsf{G}^c_\nu(v_1,...,v_K) = (q;q)_N \prod_{i = 1}^N \left( \frac{1}{1 - su_i} \prod_{j = 1}^K \frac{1 - qu_i v_j}{1 - u_iv_j}\right) .
\end{equation}

We end this section with the {\em symmetrization formulas} for $\mathsf{G}^c_\nu$ and $\mathsf{F}_\mu$ and also formulas for the functions when the variable set forms a geometric progression with parameter $q$.
\begin{proposition}\label{symmF}
{\bf 1.} For any $N \in \mathbb{Z}_{\geq 0}$, $\mu \in \mathsf{Sign}_N^+$ and $u_1,...,u_N \in \mathbb{C}$, one has
\begin{equation}
\mathsf{F}_{\mu}(u_1,...,u_{N}) = \frac{(1-q)^N}{\prod_{i = 1}^N(1 -su_i)}\sum_{\sigma \in S_N} \sigma \left( \prod_{1 \leq \alpha < \beta \leq N} \frac{u_{\alpha} - q u_\beta}{u_\alpha - u_\beta} \left( \frac{ u_i - s}{1 - su_i}\right)^{\mu_i}\right).
\end{equation}
{\bf 2.} Let $n \geq 0$ and $\mathsf{Sign}_n^+ \ni \nu = 0^{n_0}1^{n_1}2^{n_2}\cdots$. Then for any $N \geq n - n_0$ and $u_1,...,u_N \in \mathbb{C}$ we have
\begin{equation}
\begin{split}
\mathsf{G}^c_{\nu}(u_1,...,u_{N}) = \frac{(1-q)^N(q;q)_n}{\prod_{i = 1}^N(1 - su_i)(q;q)_{N- n + n_0}(q;q)_{n_0}} \prod_{k = 1}^{\infty} \frac{(s^2;q)_{n_k}}{(q;q)_{n_k}} \\ 
\times \sum_{\sigma \in S_N} 
\sigma \left( \prod_{1 \leq \alpha < \beta \leq N} \frac{u_{\alpha} - q u_\beta}{u_\alpha - u_\beta} \left( \frac{ u_i - s}{1 - su_i}\right)^{\nu_i} \prod_{i = 1}^{n- n_0}\frac{u_i}{u_i - s}\prod_{j = n-k+1}^N(1 -sq^{n_0}u_j)\right).
\end{split}
\end{equation}
In both equations above, $S_N$ denotes the permutation group on $\{1,...,N\}$ and an element $\sigma \in S_N$ acts on the expression by permuting the variable set to $u_{\sigma(1)},...,u_{\sigma(N)}$. By agreement, we set $\nu_j = 0$ if $j > n$. If $N < n-n_0$, then $\mathsf{G}^c_{\nu}(u_1,...,u_{N})$ is equal to $0$.
\end{proposition}

\begin{proposition}\label{qGeom1}
{\bf 1.} For any $N \in \mathbb{Z}_{\geq 0}$, $\mu \in \mathsf{Sign}_N^+$ and $u \in \mathbb{C}$, one has
\begin{equation}
\mathsf{F}_{\mu}(u,qu,...,q^{N-1}u) = (q;q)_N\prod_{i = 1}^N \left ( \frac{1}{1 - sq^{i-1}u}  \left( \frac{ uq^{i-1} - s}{1 - sq^{i-1}u}\right)^{\mu_i}\right).
\end{equation}
{\bf 2.} Let $n \geq 0$ and $\mathsf{Sign}_n^+ \ni \nu = 0^{n_0}1^{n_1}2^{n_2}\cdots$. Then for any $N \geq n - n_0$ and $u \in \mathbb{C}$ we have
\begin{equation}
\begin{split}
\mathsf{G}^c_{\nu}(u,qu,...,q^{N-1}u) =  \prod_{k = 1}^\infty \frac{(s^2;q)_{n_k}}{(q;q)_{n_k}} \cdot\frac{(q;q)_N (su;q)_{N+n_0}(q;q)_n  \prod_{i = 1}^N \left( \frac{1}{1 - sq^{i-1}u} \left( \frac{q^{i-1}u - s}{1 - sq^{i-1}u}\right)^{\nu_j} \right)}{(q;q)_{N-n+n_0}(su;q)_n(q;q)_{n_0}(su^{-1};q^{-1})_{n-n_0}}.
\end{split}
\end{equation}
\end{proposition}

\subsection{The measure $\mathbb{P}_{{\bf u, v}}$}\label{Section2.2} As discussed in Section \ref{Section1.2} the main probabilistic object we study is the measure $\mathbb{P}_{{\bf u, v}}$ on up-right paths in the half-infinite strip that share no horizontal or vertical pieces. The purpose of this section is to properly define it. 

Let us briefly explain the main steps of the construction of $\mathbb{P}_{{\bf u, v}}$. We begin by considering the bigger space of all up-right paths in the half-infinite strip that share no horizontal piece but are allowed to share vertical pieces. For each such collection of paths we define its weight and show that these weights are absolutely summable and their sum has a product form. Afterwards we specialize one parameter in those weights and perform a limit transition for some of the other parameters. This procedure has the effect of killing the weight of those path configurations that share a vertical piece. Consequently, we are left with weights that are non-zero only for six-vertex configurations, are absolutely summable and their sum has a product form. We check that each weight is non-negative, and define $\mathbb{P}_{{\bf u, v}}$ as the quotient of these weights with the partition function.\\

We fix positive integers $N, M$, $J$, and $K = M + J$, as well as real numbers $q \in (0,1)$ and $s> 1$. In addition, we suppose ${\bf u} = (u_1,...,u_N)$ and ${\bf w} = (w_1,...,w_K)$ are positive real numbers, such that $\max_{i,j} u_i w_j < 1$ and $u:=\min_i u_i > s$. One readily verifies that the latter conditions ensure that $u_i, w_j$ are admissible with respect to $s$ for $i = 1,...,N$ and $j = 1,...,K$.

Let us go back to the setup of Section \ref{Section1.2}. We let $\mathcal{P}'_N$ be the collection of $N$ up-right paths drawn in the sector $D_N = \mathbb{Z}_{\geq 0} \times \{1,...,N\}$ of the square lattice, with all paths starting from a left-to-right arrow entering each of the points $\{ (0,m): 1 \leq m \leq N\}$ on the left boundary and all paths exiting from the top boundary of $D_N$. We still assume that no two paths share a horizontal piece, but sharing vertical pieces is allowed. As in Section \ref{Section1.2} we let $\mathcal{P}_N \subset \mathcal{P}_N'$ be those collections of paths that do not share vertical pieces. For $\omega \in \mathcal{P}'_N$ and $k = 1,...,N$ we let $\lambda^k(\omega) \in \mathsf{Sign}_k^+$ denote the ordered $x$-coordinates of the intersection of $\omega$ with the horizontal line $y = k + 1/2$. We denote by $\omega(i,j)$ the arrow configuration at the vertex in position $(i,j) \in D_N$. We also let $f : \mathsf{Sign}_N^+ \rightarrow \mathbb{R}$ be given by 
$$f(\lambda; {\bf w}):= \sum_{ \mu \in \mathsf{Sign}_N^+} \mathsf{G}^c_{\mu}(w_1,...,w_J) \mathsf{G}^c_{\lambda/ \mu}(w_{J+1},...,w_{K}).$$
With the above data, we define the weight of a collection of paths $\omega$ by
$$\mathcal{W}^f_{{\bf u}, {\bf w}}(\omega) = \prod_{i = 1}^N \prod_{j = 1}^\infty w_{u_i}(\omega(i,j)) \times f(\lambda^N(\omega); {\bf w}).$$
If we perform the summation over $\mu$ and use Proposition \ref{Branching} we see that $f(\lambda; {\bf w}) = \mathsf{G}^c_{\lambda}(w_1,...,w_K)$. This together with Definition \ref{defF} implies that 
$$\mathcal{W}^f_{{\bf u}, {\bf w}}(\omega) =  \mathsf{F}_{\lambda^{1}(\omega)}(u_1) \mathsf{F}_{\lambda^{2}(\omega) / \lambda^{1}(\omega)}(u_2) \cdots \mathsf{F}_{\lambda^{N}(\omega) / \lambda^{N-1}(\omega)}(u_N)  \times \mathsf{G}^c_{\lambda^N(\omega)}(w_1,...,w_K).$$
Using the branching relations for $\mathsf{F}_{\lambda/ \mu}$ from Proposition \ref{Branching} and performing the sum over $\lambda^1,...,\lambda^{N-1}$ we obtain $ \mathsf{F}_{\lambda^N(\omega)}(u_1,...,u_N) \mathsf{G}^c_{\lambda^N(\omega)}(w_1,...,w_K)$. A final summation over $\lambda^N$ and application of the Cauchy identity (\ref{CauchyS2}) leads us to
$$\sum_{\omega \in \mathcal{P}_N'}\mathcal{W}^f_{{\bf u}, {\bf w}}(\omega) = (q;q)_N \prod_{i = 1}^N \left( \frac{1}{1 - su_i} \prod_{j = 1}^K \frac{1 - qu_i w_j}{1 - u_iw_j}\right)=:Z^f({\bf u}; {\bf w}).$$
In view of the admissability conditions satisfied by ${\bf u}$ and ${\bf w}$, the above sum is in fact absolutely convergent, hence the particular order of summation we chose is irrelevant. We remark that the weights $\mathcal{W}^f_{{\bf u}, {\bf w}}(\omega)$ are real and not necessarily non-negative, but they are absolutely summable and their sum equals the above expression.\\

We next wish to specialize some of the variables $w_i$ and relabel the others, in addition we fix $s = q^{-1/2}$. Set $w_i = q^{i-1} \epsilon$ for $i = 1,...,J$ and put $v_j = w_{j + J}$ for $j = 1,..., M$. Here $\epsilon > 0$ is chosen sufficiently small so that the admissibility condition is maintained.

\begin{remark}\label{remarkG}
Choosing $s = q^{-1/2}$ has the effect that if $\mu \in \mathsf{Sign}^+_{N}$ has distinct parts and $\lambda \in \mathsf{Sign}^+_N$ then $\mathsf{G}^c_{\lambda / \mu}(u_1,...,u_n) = 0$, unless $\lambda$ has distinct parts. Indeed, suppose that $k = \lambda_i = \lambda_{i+1}$ for some $i \in \{1,..., N - 1\}$. Let $\omega' \in \mathcal{P}^c_{\lambda / \mu}$ (see Definition \ref{defG}). For $j = 1,...,n+1$ denote by $a_j$ the number of arrows from $(k,j-1)$ to $(k,j)$. As the number of horizontal arrows entering or leaving a given vertex is $0$ or $1$, we see that $a_{j+1} \in \{ a_j, a_j - 1, a_j + 1\}$ for $j = 1,...,n$. Our assumption on $\lambda$ and $\mu$ implies that $a_1 \leq 1$, while $a_{n+1} \geq 2$, thus for some $j \in \{1,...,n\}$ we must have $a_j = 1$ and $a_{j + 1} = 2$. Consequently, any $\omega' \in \mathcal{P}^c_{\lambda / \mu}$ contains a vertex of the form $(1,1;2,0)$. By (\ref{weights2}) the conjugated weight of such a vertex equals $w_u^c(1,1;2,0) = \frac{1-s^2q}{1 -su} = 0$ if $s = q^{-1/2}$. We conclude that $\mathcal{W}^c(\omega') = 0$ for any $\omega' \in \mathcal{P}^c_{\lambda / \mu}$, which by Definition \ref{defG} implies $\mathsf{G}^c_{\lambda/ \mu}(u_1,...,u_n) = 0$.\\
\end{remark}

\begin{remark}\label{remarkF}
A similar argument to the one presented in Remark \ref{remarkG} shows that $s = q^{-1/2}$ has the effect that if $\mu \in \mathsf{Sign}_N^+$, $\lambda \in \mathsf{Sign}^+_{N+n}$ and $\lambda$ has distinct parts then $\mathsf{F}_{\lambda / \mu}(u_1,...,u_n) = 0$ unless $\mu$ has distinct parts.
\end{remark}

We investigate how the new choice of parameters affects the function $f$ and begin with the following useful lemma.
\begin{lemma}\label{GLemma} Suppose $J \geq N$, $q = (0,1)$, $s = q^{-1/2}$ and $\nu \in Sign_N^+$ with $\nu = 0^{n_0} 1^{n_1}2^{n_2} \cdots$. Then for any $v \in (0,s^{-1})$ we have
\begin{equation}\label{GForm}
\begin{split}
\mathsf{G}^c_{\nu}(v,qv,q^2,...,q^{J-1}v) = \frac{(q;q)_N(-q)^{n_0 - N}}{(q;q)_{n_0}}  \frac{(sv;q)_{N - n_0}}{(sv;q)_N}\frac{1}{(sv^{-1}; q^{-1})_{N - n_0}} \times \\
(q^{J -N + n_0 + 1};q)_{N-n_0} (svq^J;q)_{n_0}\prod_{j = 1}^{N-n_0}\left( \frac{1}{1 - svq^{j-1}} \left( \frac{vq^{j-1} - s}{ 1 - svq^{j-1}}\right)^{\nu_j} \right),
\end{split}
\end{equation}
when $n_i \leq 1$ for $i \geq 1$ and $0$ otherwise.
\end{lemma}
\begin{proof}
We begin by dropping the assumption that $s = q^{-1/2}$ and consider  $\mathsf{G}^c_{\nu}(v,qv,q^2,...,q^{J-1}v; s)$, where we record the dependence on $s$ in the notation. The latter is a {\em finite} sum of finite products of weights $w^c_{vq^j}$ and by continuity of the weights (see (\ref{weights2})) we have
$$\mathsf{G}^c_{\nu}(v,qv,q^2,...,q^{J-1}v; q^{-1/2}) = \lim_{s \rightarrow q^{-1/2}} \mathsf{G}^c_{\nu}(v,qv,q^2,...,q^{J-1}v; s).$$
Using Proposition \ref{qGeom1} we have
$$\mathsf{G}^c_{\nu}(u,qu,...,q^{J-1}u; s) =  \prod_{k = 1}^\infty \frac{(s^2;q)_{n_k}}{(q;q)_{n_k}} \cdot \frac{(q;q)_J (su;q)_{J+n_0}(q;q)_N  }{(q;q)_{J-N+n_0}(su;q)_N(q;q)_{n_0}(su^{-1};q^{-1})_{N-n_0}}\times $$
$$\prod_{i = 1}^{N - n_0} \left( \frac{1}{1 - sq^{i-1}u} \left( \frac{q^{i-1}u - s}{1 - sq^{i-1}u}\right)^{\nu_j} \right) \cdot \frac{1}{(svq^{N-n_0};q)_{J - N + n_0}}.$$
Performing some cancellation and rearrangement (see also the proof of Proposition 6.7 in \cite{BP}) we arrive at
$$\mathsf{G}^c_{\nu}(v,qv,q^2,...,q^{J-1}v; s) = \frac{(q;q)_N}{(q;q)_{n_0}}\prod_{k = 1}^\infty \frac{(s^2;q)_{n_k}}{(q;q)_{n_k}}  \frac{1}{(sv;q)_N}\frac{1}{(sv^{-1}; q^{-1})_{N - n_0}} \times $$
$$(q^{J - M + n_0 + 1};q)_{M-n_0} (sv;q)_{N - n_0}(svq^J;q)_{n_0}\prod_{j = 1}^{N-n_0}\left( \frac{1}{1 - svq^{j-1}} \left( \frac{vq^{j-1} - s}{ 1 - svq^{j-1}}\right)^{\nu_j} \right).$$

Finally, letting $s \rightarrow q^{-1/2}$ we see that $\prod_{k =1}^\infty(s^2;q)_{n_k} \rightarrow 0$ unless $n_k \leq 1$ for all $k \geq 1$, i.e. unless the non-zero parts of $\nu$ are all distinct. If $n_k \leq 1$ for all $k \geq 1$, then $\prod_{k = 1}^\infty \frac{(s^2;q)_{n_k}}{(q;q)_{n_k}} \rightarrow (-q)^{n_0 - N}$, which proves the lemma.

\end{proof}

Let us denote $q^J$ by $X$. Then, in view of Lemma \ref{GLemma}, $f$ becomes
\begin{equation}\label{feq-1}
\begin{split}
f_\epsilon(\lambda;{\bf v}, X) := \sum_{ \nu \in \mathsf{Sign}_N^+}  \frac{(q;q)_N(-q)^{n_0 - N}(s\epsilon;q)_{N - n_0}}{(q;q)_{n_0}(s\epsilon;q)_N(s\epsilon^{-1}; q^{-1})_{N - n_0}}
(Xq^{ -N + n_0 + 1};q)_{N-n_0} (s\epsilon X;q)_{n_0} \times \\
\prod_{i = 1}^{\infty}{\bf 1}_{\{ n_i \leq 1 \}}\prod_{j = 1}^{N-n_0}\left( \frac{1}{1 - s\epsilon q^{j-1}} \left( \frac{\epsilon q^{j-1} - s}{ 1 - s\epsilon q^{j-1}}\right)^{\nu_j} \right)\mathsf{G}^c_{\lambda/ \nu}(v_1,...,v_M),
\end{split}
\end{equation}
where $\nu = 0^{n_0} 1^{n_1}2^{n_2} \cdots$ and ${\bf 1}_E$ is the indicator function of an event $E$. In addition, specializing our $w$ variables in $Z^f({\bf u}; {\bf w})$ and replacing $q^J$ with $X$, we get
$$Z^{f_\epsilon}({\bf u}; {\bf v},X) :=  (q;q)_N \prod_{i = 1}^N \left( \frac{1}{1 - su_i}\frac{1 - X\epsilon u_i}{1 - \epsilon u_i} \prod_{j = 1}^M \frac{1 - qu_i v_j}{1 - u_i v_j} \right).$$
Our earlier results now yield  
\begin{equation}\label{PFEqn}
\sum_{\omega \in \mathcal{P}_N'}\prod_{i = 1}^N \prod_{j = 1}^\infty w_{u_i}(\omega(i,j)) \times f_{\epsilon}(\lambda^N(\omega);{\bf v}, X) = (q;q)_N \prod_{i = 1}^N \left( \frac{1}{1 - su_i}\frac{1 - X\epsilon u_i}{1 - \epsilon u_i} \prod_{j = 1}^M \frac{1 - qu_i v_j}{1 - u_iv_j}\right),
\end{equation}
provided $\epsilon$ is sufficiently small and $X = q^J$ with $J \geq N$.

In view of Lemma \ref{GLemma}, we have that $f_{\epsilon}(\lambda;{\bf v},X)$ is a {\em polynomial } in $X$. Moreover, one readily observes that as $X$ varies over compact sets in $\mathbb{C}$ the weights $\prod_{i = 1}^N \prod_{j = 1}^\infty w_{u_i}(\omega(i,j)) \times f(\lambda^N(\omega);{\bf v}, X)$ are absolutely summable (this is a consequence of the admissibility conditions and our choice for $\epsilon$). Hence the LHS of (\ref{PFEqn}) is an entire function in $X$. The RHS of (\ref{PFEqn}) is also clearly entire in $X$ and the two sides agree whenever $X = q^J$ with $J \geq N$. Since $q^J$ is a sequence with a limit point in $\mathbb{C}$, we conclude that (\ref{PFEqn}) holds for all $X$ and we will set $X = (s\epsilon)^{-1}$.\\

When we subsitute $X =  (s\epsilon)^{-1}$ in the expression for $f_\epsilon(\lambda; {\bf v}, X)$ we see that the factor $(s\epsilon X;q)_{n_0} $ vanishes unless $n_0 = 0$, in which case it equals $1$. Denoting $f_\epsilon(\lambda; {\bf v}, (s\epsilon)^{-1})$ by $f_\epsilon(\lambda; {\bf v})$ we thus obtain
$$f_\epsilon(\lambda; {\bf v}) = \sum_{ \nu \in \mathsf{Sign}_N^+}  \frac{(q;q)_N(-q)^{n_0 - N}(s\epsilon;q)_{N - n_0}}{(q;q)_{n_0}(s\epsilon;q)_N(s\epsilon^{-1}; q^{-1})_{N - n_0}}
((s\epsilon)^{-1}q^{ -N + n_0 + 1};q)_{N-n_0} \times $$
$${\bf 1}_{\{n_0 = 0\}}\prod_{i = 1}^{\infty}{\bf 1}_{\{ n_i \leq 1 \}}\prod_{j = 1}^{N-n_0}\left( \frac{1}{1 - s\epsilon q^{j-1}} \left( \frac{\epsilon q^{j-1} - s}{ 1 - s\epsilon q^{j-1}}\right)^{\nu_j} \right)\mathsf{G}^c_{\lambda/ \nu}(v_1,...,v_M),$$
and equation (\ref{PFEqn}) takes the form
\begin{equation}\label{PFEqn2}
\sum_{\omega \in \mathcal{P}_N'}\prod_{i = 1}^N \prod_{j = 1}^\infty w_{u_i}(\omega(i,j)) \times f_\epsilon(\lambda^N(\omega); {\bf v}) = (q;q)_N \prod_{i = 1}^N \left( \frac{1}{1 - su_i}\frac{1 - s^{-1} u_i}{1 - \epsilon u_i} \prod_{j = 1}^M \frac{1 - qu_i v_j}{1 - u_iv_j}\right).
\end{equation}

Since $\mathsf{G}^c_{\lambda/ \nu}(v_1,...,v_M) = 0$ unless $\lambda_i \geq \nu_i$ for $i = 1,...,N$, we conclude that the sum, defining $f_\epsilon(\lambda; {\bf v})$ is finite and taking the limit as $\epsilon$ goes to zero we have
\begin{equation}\label{feq}
f(\lambda; {\bf v}, \rho):= \lim_{\epsilon \rightarrow 0} f_\epsilon(\lambda; {\bf v}) =   (-1)^N(q;q)_N  \hspace{-3mm} \sum_{ \nu \in \mathsf{Sign}_N^+} \hspace{-3mm} {\bf 1}_{\{n_0 = 0\}}\prod_{i = 1}^{\infty}{\bf 1}_{\{ n_i \leq 1 \}}\prod_{j = 1}^{N} \left(-s \right)^{\nu_j} \mathsf{G}^c_{\lambda/ \nu}(v_1,...,v_M),
\end{equation}
where we used that $s^2 = q^{-1}$. Taking the limit as $\epsilon \rightarrow 0^+$ in equation (\ref{PFEqn2}) we conclude that 
\begin{equation}\label{PFEqn3}
\sum_{\omega \in \mathcal{P}_N'}\prod_{i = 1}^N \prod_{j = 1}^\infty w_{u_i}(\omega(i,j)) \times f(\lambda^N(\omega); {\bf v}, \rho) = (q;q)_N \prod_{i = 1}^N \left( \frac{1 - s^{-1} u_i}{1 - su_i} \prod_{j = 1}^M \frac{1 - qu_i v_j}{1 - u_iv_j}\right).
\end{equation}
The change of the order of the limit and the sum is justified, because $u_i$ and $v_j$ are admissible for $i = 1,...,N$ and $j = 1,...,M$, and $u_i > s$.\\

With $f(\lambda; {\bf v}, \rho)$ given by (\ref{feq}), we define the following weight of a collection of paths in $\mathcal{P}'_N$
\begin{equation}\label{WS2}
\mathcal{W}^f_{{\bf u}, {\bf v}}(\omega) = \prod_{i = 1}^N \prod_{j = 1}^\infty w_{u_i}(\omega(i,j)) \times f(\lambda^N(\omega);{\bf v}, \rho).
\end{equation}
So far, we only know that $\mathcal{W}^f_{{\bf u}, {\bf v}}(\omega)$ are finite real numbers, which are absolutely summable and their sum equals the RHS of (\ref{PFEqn3}). We will show below that $\mathcal{W}^f_{{\bf u}, {\bf v}}(\omega) = 0$ unless $\omega \in \mathcal{P}_N$, in which case it is non-negative. This will show that one can define an honest probability measure on $\mathcal{P}_N$ through these weights.

We first investigate when such a weight vanishes. Since $\mathsf{G}^c_{\lambda/ \nu}(v_1,...,v_M)$ vanishes unless $\lambda_i \geq \nu_i$ for $i = 1,...,N$, we see that $f(\lambda; {\bf v}, \rho) = 0$ unless $\lambda_N > 0$. Combining this with Remark \ref{remarkG}, we see that $f(\lambda; {\bf v}, \rho) = 0$ unless $\lambda^N(\omega)$ has all distinct and positive parts. Let $\omega \in \mathcal{P}'_N$, be such that $\lambda^N(\omega)$ has distinct and non-zero parts. Using that
$$\prod_{i = 1}^N \prod_{j = 1}^\infty w_{u_i}(\omega(i,j)) =  \mathsf{F}_{\lambda^{1}(\omega)}(u_1) \mathsf{F}_{\lambda^{2}(\omega) / \lambda^{1}(\omega)}(u_2) \cdots \mathsf{F}_{\lambda^{N}(\omega) / \lambda^{N-1}(\omega)}(u_N),$$
together with Remark \ref{remarkF}, we conclude that $\mathcal{W}^f_{{\bf u}, {\bf v}}(\omega) = 0$ unless $\lambda^i$ have distinct parts for all $i = 1,...,N$, i.e. unless $\omega \in \mathcal{P}_N$.\\

We next investigate the sign of $\mathcal{W}^f_{{\bf u}, {\bf v}}(\omega)$. Since the weight is $0$ otherwise, we may assume that $\omega \in \mathcal{P}_N$. Hence we have six possible choices for the vertices $\omega(i,j)$: $(0,0;0,0)$, $(0,1;1,0)$, $(1,0;1,0)$, $(0,1;1,0)$, $(0,1;0,1)$ and $(1,1;1,1)$. Using the formulas in (\ref{weights1}) we see that the sign of the weight of a vertex is precisely $(-1)^{j_1}$. Consequently, the sign of $ \prod_{i = 1}^N \prod_{j = 1}^\infty w_{u_i}(\omega(i,j))$ is precisely $(-1)^{K(\omega)}$, where $K(\omega)$ is the number of horizontal arrows in the configuration $\omega$. One readily observes that the number of horizontal arrows in $D_N$ is precisely $\sum_{i = 1}^N\lambda^N_i(\omega)$. In addition, we have $N$ horizontal arrows from $(-1,i)$ to $(0,i)$ for $i = 1,..., N$. Thus we conclude that $sign\left( \prod_{i = 1}^N \prod_{j = 1}^\infty w_{u_i}(\omega(i,j))\right) = (-1)^{N + \sum_{i = 1}^N \lambda^N_i(\omega)}$. 

We next consider the sign of $\mathcal{W}^c(\omega')$, where $\omega' \in \mathcal{P}^c_{\lambda / \nu}$ and $\nu$ has distinct and positive parts. Arguing as in Remark \ref{remarkG}, we can assume that no paths in $\omega'$ share a vertical piece, otherwise $\mathcal{W}^c(\omega') = 0$. Consequently, we may assume that $\omega'(i,j)$ is among the six vertex types we had before for all $(i,j) \in D_M$. From (\ref{weights2}) the sign of the conjugated weight of a vertex is again $(-1)^{j_1}$, and so the sign equals $(-1)^{K(\omega')}$, where $K(\omega')$ is the total number of horizontal arrows in $\omega'$. One readily observes that $K(\omega') = \sum_{i = 1}^N\lambda_i - \sum_{i = 1}^N\nu_i$ (notice that in this case we do not have horizontal arrows entering the $0$-th column). We conclude that all weights $\mathcal{W}^c(\omega')$ for $\omega' \in \mathcal{P}^c_{\lambda/ \nu}$ have the same sign, which implies that $sign \left( \mathsf{G}^c_{\lambda / \nu}(v_1,...,v_M)\right) = (-1)^{ \sum_{i = 1}^N\lambda_i - \sum_{i = 1}^N\nu_i}$. 

The result in the last paragraph implies that each summand in (\ref{feq}) has sign $ (-1)^{ \sum_{i = 1}^N\lambda_i}$ and so we conclude that $sign \left( f(\lambda; {\bf v}, \rho)\right) = (-1)^{N +  \sum_{i = 1}^N\lambda_i}$. Combining this with  $sign\left( \prod_{i = 1}^N \prod_{j = 1}^\infty w_{u_i}(\omega(i,j))\right) $ \\
$ =(-1)^{N + \sum_{i = 1}^N \lambda^N_i(\omega)}$, we conclude that $ \mathcal{W}^f_{{\bf u}, {\bf v}}(\omega)  \geq 0$ for all $\omega \in \mathcal{P}_N$. As $\mathcal{W}^f_{{\bf u}, {\bf v}}(\omega) = 0$ for $\omega \in \mathcal{P}_N' / \mathcal{P}_N$, equation (\ref{PFEqn3}) can be rewritten as
\begin{equation}\label{PFEqn4}
\sum_{\omega \in \mathcal{P}_N} \mathcal{W}^f_{{\bf u}, {\bf v}}(\omega) = (q;q)_N \prod_{i = 1}^N \left( \frac{1 - s^{-1} u_i}{1 - su_i} \prod_{j = 1}^M \frac{1 - qu_i v_j}{1 - u_iv_j}\right) =:Z^f({\bf u}).
\end{equation}
As weights are non-negative and the partition function $Z^f({\bf u})$ is positive and finite, we see that 
$$\mathbb{P}_{{\bf u}, {\bf v}}(\omega) := \frac{\mathcal{W}^f_{{\bf u}, {\bf v}}(\omega)}{Z^f({\bf u})},$$
defines an honest probability measure on $\mathcal{P}_N$. For future reference we summarize the parameter choices we have made in the following definition.

\begin{definition}\label{parameters}
Let $N,M \in \mathbb{N}$. We fix $q \in (0,1)$ and $s = q^{-1/2}$, ${\bf u} = (u_1,...,u_N)$ with $u_i > s$ and ${\bf v} = (v_1,...,v_M)$ with $v_j > 0$, and $\max_{i,j} u_iv_j < 1$. With these parameters, we denote $\mathbb{P}_{{\bf u}, {\bf v}}$ to be the probability measure on $\mathcal{P}_N$, defined above.
\end{definition}

\subsection{Projections of $\mathbb{P}_{{\bf u, v}}$}\label{Section2.3} We assume the same notation as in the previous section. Let us fix $k \in \mathbb{N}$, $1\leq m_1 < m_2 < \cdots < m_k \leq N$ and $\mu^{m_i} \in \mathsf{Sign}_{m_i}^+$. Our goal in this section is to derive formulas for the following probabilities 
$$\mathbb{P}_{{\bf u, v}} \left( \lambda^{m_1}(\omega) = \mu^{m_1},..., \lambda^{m_k}(\omega) = \mu^{m_k} \right).$$
Let $A = \{ \omega \in \mathcal{P}_N : \lambda^{m_1}(\omega) = \mu^{m_1},..., \lambda^{m_k}(\omega) = \mu^{m_k}  \}$. Then we have that 
$$\mathbb{P}_{{\bf u, v}}( A) = Z^f({\bf u})^{-1} \sum_{ \omega \in A} \mathcal{W}^f_{{\bf u}, {\bf v}}(\omega) =  Z^f({\bf u})^{-1} \sum_{ \omega \in A}\prod_{i = 1}^N \prod_{j = 1}^\infty w_{u_i}(\omega(i,j)) \times f(\lambda^N(\omega);{\bf v}, \rho) = $$
$$ = Z^f({\bf u})^{-1} \sum_{ \omega \in A} \mathsf{F}_{\lambda^{1}(\omega)}(u_1) \mathsf{F}_{\lambda^{2}(\omega) / \lambda^{1}(\omega)}(u_2) \cdots \mathsf{F}_{\lambda^{N}(\omega) / \lambda^{N-1}(\omega)}(u_N) f(\lambda^N(\omega);{\bf v}, \rho).$$
Let $M = \{m_1,...,m_k\}$. We observe that the rightmost sum above may be replaced with the sum over all $\lambda^i \in \mathsf{Sign}_i^+$, where $i \in \{1,...,N\} / M$. Indeed, from our work in the previous section, the extra terms that we are summing over are all $0$. We thus conclude that
\begin{equation}\label{marg1} 
\mathbb{P}_{{\bf u, v}}( A)  =  Z^f({\bf u})^{-1}\sum_{i \in \{1,...,N\} / M } \sum_{\lambda^i \in \mathsf{Sign}_i^+ } \mathsf{F}_{\lambda^{1}}(u_1) \mathsf{F}_{\lambda^{2}/ \lambda^{1}}(u_2) \cdots \mathsf{F}_{\lambda^{N} / \lambda^{N-1}}(u_N) f(\lambda^N;{\bf v}, \rho),
\end{equation}
where $\lambda^j = \mu^j$ for $j \in M$ are fixed. 

Let us first assume that $m_k = N$. Then the branching relations (\ref{branchF}) yield
$$\sum_{i = m_r + 1 }^{m_{r+1}-1} \sum_{\lambda^i \in \mathsf{Sign}_i^+ } \mathsf{F}_{\lambda^{m_r + 1}/ \lambda^{m_r}}(u_{m_r + 1}) \mathsf{F}_{\lambda^{m_r + 2}/ \lambda^{m_r + 1}}(u_{m_r + 2})\cdots  \mathsf{F}_{\lambda^{m_{r + 1}}/ \lambda^{m_{r+1}-1}}(u_{m_{r + 1}}) =  $$
$$ \mathsf{F}_{\lambda^{m_{r + 1}}/  \lambda^{m_r}}(u_{m_r + 1},..., u_{m_{r + 1}}),$$
when $r = 0,..., k-1$ with the convention that $m_0 = 0$ and $\lambda^0 = \varnothing$. Substituting these expressions in (\ref{marg1}) we see that if $m_k = N$ we have
$$\mathbb{P}_{{\bf u, v}}( A)  =  Z^f({\bf u})^{-1} \times \prod_{r = 0}^{k-1} \mathsf{F}_{\mu^{m_{r + 1}}/  \mu^{m_r}}(u_{m_r + 1},..., u_{m_{r + 1}}) \times  f(\mu^{m_k};{\bf v}, \rho).$$
In the remainder, we assume that $m_k < N$. In this case we may still apply the branching relations as above to conclude that 
\begin{equation}\label{specL}
\begin{split}
&\mathbb{P}_{{\bf u, v}}( A)  =  Z^f({\bf u})^{-1} \times \prod_{r = 0}^{k-1} \mathsf{F}_{\mu^{m_{r + 1}}/  \mu^{m_r}}(u_{m_r + 1},..., u_{m_{r + 1}}) \times  F(\mu^{m_k};{\bf v}, \rho),\mbox{ where }\\
& F(\mu^{m_k};{\bf v}, \rho) := \sum_{\lambda \in \mathsf{Sign}_N^+ } \mathsf{F}_{\lambda / \mu^{m_k}}(u_{m_k + 1},...,u_N) f(\lambda;{\bf v}, \rho).
\end{split}
\end{equation}
An alternative formula for $ F(\mu^{m_k};{\bf v}, \rho)$  is derived in the following lemma.
\begin{lemma}
Let $N, m \in \mathbb{N}$, $q \in (0,1)$, $s = q^{-1/2}$, $\lambda \in \mathsf{Sign}_N^+$ and $\mu \in \mathsf{Sign}_{m-1}^+$. Assume that $u_{m},...,u_N$ and $v_1,...,v_M$ are positive with $u_i > s$ and $u_i, v_j$ admissible with respect to $s$. Then
\begin{equation}\label{marg2}
\sum_{\lambda \in \mathsf{Sign}_N^+ } \mathsf{F}_{\lambda / \mu}(u_{m},...,u_N) f(\lambda;{\bf v}, \rho) =  \prod_{i = m}^N \frac{(1 - q^i)(1 - s^{-1}u_i)}{1 - su_i} \prod_{j = 1}^{M}\frac{ 1 - qu_iv_j}{1 - u_i v_j} f(\mu;{\bf v}, \rho).
\end{equation}
\end{lemma}
\begin{proof}
We start by considering the expression
$$ \sum_{\lambda \in \mathsf{Sign}_N^+ } \mathsf{F}_{\lambda / \mu}(u_{m+1},...,u_N) \mathsf{G}^c_\lambda(w_1,...,w_{J+M}),$$
where as in the previous section $w_i = \epsilon q^{i-1}$ for $i = 1,...,J$ and $w_{j+J} = v_j$ for $j = 1,...,M$. The skew Cauchy identity in (\ref{CauchyId}) yields (see also Corollary 4.11 in \cite{BP}):
$$\sum_{\lambda \in \mathsf{Sign}_N^+ } \mathsf{F}_{\lambda / \mu}(u_{m},...,u_N) \mathsf{G}^c_\lambda(w_1,...,w_{J+M}) = \prod_{i = m}^N \frac{1 - q^i}{1 - su_i} \prod_{j = 1}^{J+M}\frac{ 1 - qu_iw_j}{1 - u_i w_j}\mathsf{G}^c_\mu(w_1,...,w_{J+M}).$$
Substituting $w_i$ in the above expression and denoting $q^J$ by $X$ we arrive at
\begin{equation}\label{marg3}
\sum_{\lambda \in \mathsf{Sign}_N^+ } \mathsf{F}_{\lambda / \mu}(u_{m},...,u_N) f_\epsilon(\lambda;{\bf v}, X) =  \prod_{i = m}^N \frac{1 - q^i}{1 - su_i} \frac{1 - \epsilon Xu_i}{1 - \epsilon u_i} \prod_{j = 1}^{M}\frac{ 1 - qu_iv_j}{1 - u_i v_j}f_\epsilon(\mu;{\bf v}, X),
\end{equation}
where $f_\epsilon(\mu;{\bf v}, X)$ is given in (\ref{feq-1}). As in the previous section we argue that both sides of (\ref{marg3}) are entire functions in $X$, which are equal on a sequence with a limit point in $\mathbb{C}$, hence equality holds for all $X$. Setting $X = (s\epsilon)^{-1}$ and letting $\epsilon$ go to zero we obtain (\ref{marg2}).
\end{proof}
Substituting  $F(\mu^{m_k};{\bf v}, \rho)$  into (\ref{specL}) with the expression in (\ref{marg2}) and performing a bit of cancellations we see that 
\begin{equation}\label{marg5}
\begin{split}
\mathbb{P}_{{\bf u, v}} \left( \lambda^{m_1}(\omega) = \mu^{m_1},..., \lambda^{m_k}(\omega) = \mu^{m_k} \right) = \prod_{r = 0}^{k-1} \mathsf{F}_{\mu^{m_{r + 1}}/  \mu^{m_r}}(u_{m_r + 1},..., u_{m_{r + 1}})  \times  f(\mu^{m_k};{\bf v}, \rho) \times \\ Z^f({\bf u}, {\bf v}; m_k)^{-1} ,  \mbox{ where } Z^f({\bf u}, {\bf v}; m_k) =  (q;q)_{m_k} \prod_{i = 1}^{m_k} \left( \frac{1 - s^{-1} u_i}{1 - su_i} \prod_{j = 1}^M \frac{1 - qu_i v_j}{1 - u_iv_j}\right).
\end{split}
\end{equation}

\section{The operators $D_m^k$}\label{Section3}

In this section we fix a positive integer $m \geq 1$ and provide operators $D^k_m$ for $k = 1,...,m$ that act diagonally on the functions $\mathsf{F}_\lambda(u_1,...,u_m)$ with $\lambda \in \mathsf{Sign}^+_m$ and $\lambda_{i} > \lambda_{i+1}$ for $i = 1,..., m-1$. Specifically, we will show that 
$$D^k_m \mathsf{F}_\lambda(u_1,...,u_m) = {\bf 1}_{ \{\lambda_m = 0, \lambda_{m-1} = 1,..., \lambda_{m-k+1} = k-1 \}}\mathsf{F}_\lambda(u_1,...,u_m) .$$
In addition, we explain how the operators $D^k_m$ can be used to extract formulas for a set of observables and prove several properties that are relevant to the problem we consider.

\subsection{Definition of $D_m^k$ }\label{Section3.1}
 We start with the symmetrization formula for $\mathsf{F}_\lambda(u_1,...,u_m)$  (here $\lambda \in \mathsf{Sign}_m^+$), given in Proposition \ref{symmF}:
\begin{equation}\label{FSym}
\mathsf{F}_\lambda(u_1,...,u_m) = \frac{(1- q)^m}{ \prod_{i = 1}^m (1 - su_i)}\sum_{\sigma \in S_m}  \prod_{1 \leq \alpha < \beta \leq m}\frac{u_{\sigma(\alpha)} - qu_{\sigma(\beta)}}{u_{\sigma(\alpha)} - u_{\sigma(\beta)}} \prod_{i = 1}^m\left(\frac{u_{\sigma(i)} - s}{1 - su_{\sigma(i)}}\right)^{\lambda_i}.
\end{equation}
We are interested in setting $u_m = \cdots = u_{m-k+1} = s$ for each $k \in \{0,...,m\} $ in the above expression, which is the content of the following lemma.

\begin{lemma}\label{LemmaSub} Let $m \geq 1$, $k \in \{0,...,m\}$ and $\lambda \in \mathsf{Sign}_m^+$ with $\lambda_{i} > \lambda_{i+1}$ for $i = 1,..., m-1$. Then we have that 
\begin{equation}\label{FSub}
\begin{split}
\mathsf{F}_\lambda(u_1,...,u_{m-k}, s,...,s) =  \frac{1}{(1-s^2)^k}\frac{(1- q)^m}{ \prod_{i = 1}^{m-k} (1 - su_i)}\left( \frac{s (1-q)}{1-s^2}\right)^{k(k-1)/2} \times\\
 \left( \prod_{i = 1}^{m-k} \frac{u_{i} - qs}{1 - su_{i}}\right)^k \sum_{\sigma \in S_{m-k}}   \prod_{1 \leq \alpha < \beta \leq m-k}\frac{u_{\sigma(\alpha)} - qu_{\sigma(\beta)}}{u_{\sigma(\alpha)} - u_{\sigma(\beta)}}   \prod_{i = 1}^{m - k }\left(\frac{u_{\sigma(i)} - s}{1 - su_{\sigma(i)}}\right)^{\lambda_i -k},
\end{split}
\end{equation}
if $\lambda_{m} = 0$, $\lambda_{m-1} = 1$,..., $\lambda_{m-k+1} = k-1$ (if $k = 0$ this condition is empty). Otherwise $\mathsf{F}_\lambda(u_1,...,u_{m-k}, s,...,s)  = 0$. If $k = m$ the sum over $S_{m-k}$ is replaced by $1$.
\end{lemma}
\begin{proof}
We proceed by induction on $k$ with base case $k = 0$ true by (\ref{FSym}). Supposing the result for $k$ we now show it for $k+1$.\\

By induction hypothesis we may assume that  $\lambda_{m} = 0$, $\lambda_{m-1} = 1$,..., $\lambda_{m-k+1} = k-1$, for otherwise the expression is $0$ for all $u_{m-k}$ in particular for $u_{m-k} =s $ and there is nothing to prove. Consequently, we have that 
$$\mathsf{F}_\lambda(u_1,...,u_{m-k}, s,...,s) =  \frac{1}{(1-s^2)^k}\frac{(1- q)^m}{ \prod_{i = 1}^{m-k} (1 - su_i)}\left( \frac{s (1-q)}{1-s^2}\right)^{k(k-1)/2} \times$$
$$ \left( \prod_{i = 1}^{m-k} \frac{u_{i} - qs}{1 - su_{i}}\right)^k\sum_{\sigma \in S_{m-k}}   \prod_{1 \leq \alpha < \beta \leq m-k}\frac{u_{\sigma(\alpha)} - qu_{\sigma(\beta)}}{u_{\sigma(\alpha)} - u_{\sigma(\beta)}}   \prod_{i = 1}^{m - k }\left(\frac{u_{\sigma(i)} - s}{1 - su_{\sigma(i)}}\right)^{\lambda_i -k}.$$
Since $\lambda_{i} > \lambda_{i+1}$ we know $\lambda_{m-k} \geq \lambda_{m-k+1} + 1 = k$. We notice that $(u_{m-k} -s)^{\lambda_{m-k} - k}$ divides each summand and so the total sum will be $0$ unless $\lambda_{m-k} = k$. Let us assume that $\lambda_{m-k} = k$, which means $\lambda_i > k$ for $i < m-k$. The latter implies that each summand for which $\sigma(m-k) \neq m-k$ is divisible by $(u_{m-k} -s)$ and so vanishes when $u_{m-k} = s$. This reduces the sum over $S_{m-k}$ to a sum over $S_{m-k-1}$ and if we substitute $u_{m-k} =s $ we see that 
$$\mathsf{F}_\lambda(u_1,...,u_{m-k-1},s, s,...,s) =  \frac{1}{(1-s^2)^{k+1}}\frac{(1- q)^m}{ \prod_{i = 1}^{m-k-1} (1 - su_i)}\left( \frac{s (1-q)}{1-s^2}\right)^{k(k-1)/2} \left(  \frac{s - qs}{1 - s^2}\right)^k\times$$
$$ \left( \prod_{i = 1}^{m-k-1} \frac{u_{i} - qs}{1 - su_{i}}\right)^k\sum_{\sigma \in S_{m-k-1}}   \prod_{1 \leq \alpha < \beta \leq m-k-1}\frac{u_{\sigma(\alpha)} - qu_{\sigma(\beta)}}{u_{\sigma(\alpha)} - u_{\sigma(\beta)}} \prod_{i = 1}^{m - k-1 } \frac{u_{\sigma(i)} - qs}{u_{\sigma(i)} -s }  \prod_{i = 1}^{m - k-1 }\left(\frac{u_{\sigma(i)} - s}{1 - su_{\sigma(i)}}\right)^{\lambda_i -k}.$$
Upon rearrangement the above equals the expression in (\ref{FSub}) with $k+1$. The general result now proceeds by induction on $k$.

\end{proof}

Put $M = \{1,..., m\}$. We record the following alternative representation of $\mathsf{F}_\lambda(u_1,...,u_m)$, which can be obtained from (\ref{FSym}) by splitting the sum over the possible variable subsets formed by $\{ u_{\sigma(m)}, ..., u_{\sigma(m-k+1)}\}$ (these correspond to sets $I$ below and $I^c = M / I$)
\begin{equation}\label{FSymK}
\begin{split}
&\mathsf{F}_\lambda(u_1,...,u_m) = \frac{(1- q)^m}{ \prod_{i = 1}^m (1 - su_i)}\sum_{\sigma \in S_k}\sum_{\tau \in S_{m - k}} \sum_{\substack{I = \{i_1,...,i_k\} \subset M\\ I^c =\{j_1,...,j_{m-k}\} }} \prod_{ \alpha = 1}^k \prod_{ \beta = 1}^{m-k} \frac{u_{j_{\tau(\beta)}} - q u_{i_{\sigma(\alpha)}}}{u_{j_{\tau(\beta)}} -  u_{i_{\sigma(\alpha)}}} \times \\ \prod_{1\leq  \alpha_1 < \alpha_2 \leq k} &\hspace{-1mm} \frac{u_{i_{\sigma(\alpha_1)}} - q u_{i_{\sigma(\alpha_2)}}}{u_{i_{\sigma(\alpha_1)}} - u_{i_{\sigma(\alpha_2)}}} 
  \prod_{x= 1}^k \hspace{-1mm} \left(\frac{u_{i_{\sigma(x)}} - s}{1 - su_{i_{\sigma(x)}}}\right)^{\lambda_{m -x+1}} \hspace{-12mm}
 \prod_{1\leq  \beta_1 < \beta_2 \leq m-k} \hspace{-1mm}\frac{u_{j_{\tau(\beta_1)}} - q u_{j_{\tau(\beta_2)}}}{u_{j_{\tau(\beta_1)}} -  u_{j_{\tau(\beta_2)}}} \hspace{-2mm} \prod_{y= 1}^{m-k} \left(\frac{u_{j_{\tau(y)}} - s}{1 - su_{i_{\tau(y)}}}\right)^{\lambda_{y}} \hspace{-3mm}.
\end{split}
\end{equation}

We introduce some necessary notation. Define operators $T_{s,u_i} $ that act on functions of $m$ variables $(u_1,...,u_m)$, by setting $u_i$ to $s$. I.e.
$$T_{s,u_i}F(u_1,...,u_m) = F(u_1,...,u_{i-1},s,u_{i+1},...,u_m).$$
We also consider the function
$$F_k(u_1,...,u_k) =  \frac{(1-q)^k}{\prod_{i = 1}^k(1 - su_i)}\sum_{\sigma \in S_k}  \prod_{1 \leq \alpha < \beta \leq k}\frac{u_{\sigma(\alpha)} - qu_{\sigma(\beta)}}{u_{\sigma(\alpha)} - u_{\sigma(\beta)}} \prod_{i = 1}^k\left(\frac{u_{\sigma(i)} - s}{1 - su_{\sigma(i)}}\right)^{i-1}.$$
 Notice that $F_k(u_1,...,u_k) = \mathsf{F}_\lambda(u_1,...,u_k)$ with $\lambda = (k-1,k-2,...,0)$. In particular, $F_k$ is a symmetric rational function.

Let $k \in M$ be fixed. For a subset $I \subset M$ with $I =\{i_1,...,i_k\} $ we write $F(u_I)$ to mean $F(u_{i_1}, ... u_{i_k})$, whenever $F$ is a symmetric function in $k$ variables. We also write $F_k(S)$ to be $F_k(s,s,...,s)$ and from Lemma \ref{LemmaSub} we have
$$F_k(S) =s^{k(k-1)/2}\left( \frac{ 1-q}{1-s^2}\right)^{k(k+1)/2}.$$
With the above notation we define the following operators.
\begin{definition}\label{DiffOp}
Let $m \in \mathbb{N}$ and $M = \{1,...,m\}$. For $1 \leq k \leq m$ we define the operator $D^k_m$ on functions of $m$ variables to be
\begin{equation}\label{DiffOpEq}
D_m^k := \sum_{I \subset M : |I| = k} \prod_{i \in I; j \not \in I} \frac{u_j - qu_i}{u_j - u_i} \prod_{j \not \in I} \left( \frac{u_j - s}{u_j - sq}\right)^k \frac{F_k(u_I)}{F_k(S)}\prod_{i \in I} T_{u_i,s}.
\end{equation}
\end{definition}
\begin{remark}\label{remarkD}
One readily observes that $D_m^k$ is a linear operator on the set of functions in $m$-variables, and also satisfies the property that if $f_r(u_1,...,u_m)$ converge pointwise to $f(u_1,...,u_m)$, then $D^k_mf_r$ converge pointwise to $D^k_mf$ away from the points $u_i = u_j$ for $i \neq j$.
\end{remark}

The key property of $D^k_m$ is given in the following lemma.
\begin{lemma}\label{LemmaOpFs}
 Let $m \geq 1$, $k \in \{1,...,m\}$ and $\lambda \in Sign_m^+$ with $\lambda_{i} > \lambda_{i+1}$ for $i = 1,..., m-1$. Then we have that 
\begin{equation}\label{OpFs}
\begin{split}
D_m^k\mathsf{F}_\lambda(u_1,...,u_m) = {\bf 1}_{\{ \lambda_{m} = 0, \lambda_{m-1} = 1,..., \lambda_{m-k+1} = k-1\}}\mathsf{F}_\lambda(u_1,...,u_m).
\end{split}
\end{equation}
\end{lemma}
\begin{proof}
Using Lemma \ref{LemmaSub} and that $\mathsf{F}_\lambda$ is symmetric we have that $D_m^k\mathsf{F}_\lambda(u_1,...,u_m)  = 0$ unless $\lambda_{m} = 0$, $\lambda_{m-1} = 1$,..., $\lambda_{m-k+1} = k-1$. We thus assume that $\lambda_{m} = 0$, $\lambda_{m-1} = 1$,..., $\lambda_{m-k+1} = k-1$. Let $\mu \in Sign^+_{m-k}$ be given by $\mu_i = \lambda_i - k$. It follows from Lemma \ref{LemmaSub} and (\ref{FSym}) that
$$D_m^k\mathsf{F}_\lambda(u_1,...,u_m) = \sum_{I \subset M : |I| = k} \prod_{i \in I; j \not \in I} \frac{u_j - qu_i}{u_j - u_i} \prod_{j \not \in I} \left( \frac{u_j - s}{u_j - sq}\right)^k \frac{F_k(u_I)}{F_k(S)} F_k(S) \left( \prod_{j \not \in I} \frac{u_{j} - qs}{1 - su_{j}}\right)^k\mathsf{F}_\mu(u_{I^c})$$
$$ = \sum_{I \subset M : |I| = k} \prod_{i \in I; j \not \in I} \frac{u_j - qu_i}{u_j - u_i}  F_k(u_I)  \left( \prod_{j \not \in I} \frac{u_j - s}{1-su_j}\right)^k\mathsf{F}_\mu(u_{I^c}),$$
where $I^c = M/I$.

Using Proposition \ref{shifting} and (\ref{FSym}) we can rewrite the above as
\begin{equation*}
\begin{split}
&\frac{(1- q)^m}{ \prod_{i = 1}^m (1 - su_i)}\sum_{\substack{I = \{i_1,...,i_k\} \subset M\\ I^c =\{j_1,...,j_{m-k}\} }}\sum_{\sigma \in S_k}\sum_{\tau \in S_{m - k}}  \prod_{ \alpha = 1}^k \prod_{ \beta = 1}^{m-k} \frac{u_{j_{\tau(\beta)}} - q u_{i_{\sigma(\alpha)}}}{u_{j_{\tau(\beta)}} -  u_{i_{\sigma(\alpha)}}} \times \\ \prod_{1\leq  \alpha_1 < \alpha_2 \leq k} & \frac{u_{i_{\sigma(\alpha_1)}} - q u_{i_{\sigma(\alpha_2)}}}{u_{i_{\sigma(\alpha_1)}} - u_{i_{\sigma(\alpha_2)}}} 
  \prod_{x= 1}^k\left(\frac{u_{i_{\sigma(x)}} - s}{1 - su_{i_{\sigma(x)}}}\right)^{x-1}
\hspace{-6mm} \prod_{1\leq  \beta_1 < \beta_2 \leq m-k}\frac{u_{j_{\tau(\beta_1)}} - q u_{j_{\tau(\beta_2)}}}{u_{j_{\tau(\beta_1)}} -  u_{j_{\tau(\beta_2)}}}\prod_{y= 1}^{m-k} \left(\frac{u_{j_{\tau(y)}} - s}{1 - su_{i_{\tau(y)}}}\right)^{\lambda_{y}} \hspace{-2mm}.
\end{split}
\end{equation*}
By virtue of (\ref{FSymK}) the latter is exactly $\mathsf{F}_\lambda(u_1,...,u_m)$ as desired. 
\end{proof}

\subsection{Observables from $D_m^k$}\label{Section3.2} This section is devoted to explaining how one can use the operators $D^k_m$ to analyze the probability measures $\mathbb{P}^f$ on $\mathcal{P}_N$. These measures were discussed in the beginning of Section \ref{Section1.2} and $\mathbb{P}_{{\bf u, v}} $ is a particular example. In addition, we will prove an interesting property for the first operator $D^1_m$, which we believe to be of separate interest. Throughout this section we require that $q = s^{-2}$.

Let us summarize the assumptions we need to make the statements in this section valid.

{\raggedleft {\bf Assumptions}:}
\begin{itemize}
\item $N \in \mathbb{N}$ and $u_1,...,u_N$ are pairwise distinct complex numbers;
\item $f: \mathsf{Sign}_N^+ \rightarrow \mathbb{R}$ is supported on signatures with distinct parts; 
\item for $\omega \in \mathcal{P}_N$ we define the weights
\begin{equation}\label{wt2}
\mathcal{W}^f(\omega; {\bf z}):= \prod_{i = 0}^{\infty} \prod_{j = 1}^N w_{z_j}(\omega(i,j)) \times f(\lambda^N(\omega));
\end{equation}
\item  the weights in (\ref{wt2}) are absolutely summable in some neighborhood of the point $(u_1,...,u_N)$ and we denote
$$\sum_{\omega \in \mathcal{P}_N} \mathcal{W}^f(\omega; {\bf z}) =: Z^f(z_1,...,z_N) = Z^f({\bf z});$$
\item  for every $\omega \in \mathcal{P}_N$ we have $\mathcal{W}^f(\omega; {\bf u}) \geq 0$ and $Z^f({\bf u}) > 0$.
\end{itemize}
Notice that under the above assumptions $\mathbb{P}^f(\omega): = \frac{\mathcal{W}^f(\omega; {\bf u})}{Z^f({\bf u}) }$ is a probability measure on $\mathcal{P}_N$. For the remainder of this section we will work under the above assumptions.

Let us introduce the following definitions 
\begin{definition}\label{S3DefObs}
For $m,r \geq 0$ we define
\begin{equation*}
\mathsf{Sign}_{m}^*=\{ \lambda \in \mathsf{Sign}_m^+:  \lambda_m < \lambda_{m-1} < \cdots < \lambda_{1}\},\hspace{2mm}
\mathsf{Sign}_{m,r}^*= \{ \lambda \in \mathsf{Sign}_m^*: \lambda_m = 0, ... ,\lambda_{m-r+1} = r - 1\}.
\end{equation*}
Suppose $k \in \{1,...,N\}$ and $1 \leq m_1 \leq m_2 \leq \cdots \leq m_k \leq N$ are given. Set $S_i = \mathsf{Sign}^*_{m_i,i}$ for $i = 1,...,k$ and define
\begin{equation}\label{ObservableEvent}
A({\bf m}) = A(m_1,...,m_k) = \{\omega :  \lambda^{m_i}(\omega) \in S_i, \hspace{1mm} i = 1,...,k\} \mbox{ and } \mathcal{W}^f({\bf m}; {\bf z}) = \sum_{\omega \in A({\bf m})} \mathcal{W}^f(\omega; {\bf z}).
\end{equation} 
\end{definition}

\begin{lemma}\label{S3Lemma}
Assume the same notation as in Definition \ref{S3DefObs}. Then we have
\begin{equation}\label{sum5}
\mathbb{P}^f\left( A({\bf m})  \right) = \frac{\mathcal{W}^f({\bf m}; {\bf u})}{Z^f({\bf u})} = \frac{(D^1_{m_1}D^{2}_{m_2} \cdots D^{k}_{m_k}Z^f)({\bf u})}{Z^f({\bf u})}.
\end{equation}
\end{lemma}
We view Lemma \ref{S3Lemma} as one of the main results of this article. Under very mild conditions on the function $f$ it provides formulas for the observables $\mathbb{P}^f\left( A({\bf m})  \right)$, which form a large class of correlation functions that can be used to analyze the six-vertex model. In the context of this paper (\ref{sum5}) plays the role of a starting point for our asymptotic analysis, and we hope that it will be useful for studying other six-vertex models in the future.

\begin{proof}
 Repeating some of the arguments from Section \ref{Section2.3}, we have that
\begin{equation}\label{Cauchy2}
\begin{split}
&\mathcal{W}^f({\bf m}; {\bf z}) = \sum_{\mu^k \in S_k} \sum_{\mu^{k-1} \in S_{k-1}} \cdots \sum_{\mu^1 \in S_1} \prod_{i = 1}^k \mathsf{F}_{\mu^i  / \mu^{i-1} }(z_{m_{i-1} + 1},...,z_{m_i})  F(\mu^k), \mbox{ and }\\
&Z^f({\bf z}) = \sum_{\mu \in \mathsf{Sign}_{m_k}^*} \mathsf{F}_{\mu}(z_1,...,z_{m_k}) F( \mu), \mbox{ where } F( \mu) = \sum_{\lambda \in \mathsf{Sign}_{N}^*} \mathsf{F}_{\lambda/\mu}(z_{m_k+1},...,z_N ) f(\lambda).
\end{split}
\end{equation}
The statement of the lemma will be produced if we apply $D^1_{m_1}D^{2}_{m_2} \cdots D^{k}_{m_k}$ (in the $z$-variables) to both sides of the second line of (\ref{Cauchy2}),  set ${\bf z} = {\bf u}$ and divide by $Z^f({\bf u})$. We provide the details below.

We start by applying $D^k_{m_k}$ to get
$$D_{m_k}^k\sum_{\mu \in \mathsf{Sign}_{m_k}^*} \mathsf{F}_{\mu}(z_1,...,z_{m_k}) F(\mu) =\sum_{\mu \in \mathsf{Sign}_{m_k}^*}D_{m_k}^k\mathsf{F}_{\mu}(z_1,...,z_{m_k}) F(\mu)= \sum_{\mu \in S_k}\mathsf{F}_{\mu}(z_1,...,z_{m_k}) F(\mu).$$
The change of the order of the sum and the operator is allowed by the linearity of $D^k_{m_k}$ and the absolute convergence of the sum (see Remark \ref{remarkD}), while the second equality follows from Lemma \ref{LemmaOpFs}. We next use Proposition \ref{Branching} and rewrite the above as
\begin{equation}\label{sum1}
\sum_{\mu^k \in S_k} \sum_{\lambda \in \mathsf{Sign}_{m_{k-1}}^+} \mathsf{F}_{\lambda}(z_1,...,z_{m_{k-1}}) \mathsf{F}_{\mu^k / \lambda}(z_{m_{k-1} + 1},...,z_{m_k}) F(\mu^k).
\end{equation}
If $\mu^{k} \in S_k$, we know that it has all distinct parts. The latter implies by Remark \ref{remarkF} that $ \mathsf{F}_{\mu^k  / \lambda}(z_{m_{k-1} + 1},...,z_{m_k}) = 0$ unless $\lambda$ has distinct parts. Consequently we may rewrite (\ref{sum1}) as
\begin{equation}\label{sum2}
\sum_{\mu^k \in S_k} \sum_{\lambda \in \mathsf{Sign}_{m_{k-1}}^*} \hspace{-5mm}  \mathsf{F}_{\lambda}(z_1,...,z_{m_{k-1}}) \mathsf{F}_{\mu^k / \lambda}(z_{m_{k-1} + 1},...,z_{m_k}) F(\mu^k).
\end{equation}
Applying $D^{k-1}_{m_{k-1}}$ to (\ref{sum2}), using its linearity and Lemma \ref{LemmaOpFs}, we get
\begin{equation}\label{sum3}
\begin{split}
 \sum_{\mu^k \in S_k} \sum_{\mu^{k-1} \in S_{k-1}}\hspace{-5mm} \mathsf{F}_{\mu^{k-1}}(z_1,...,z_{m_{k-1}}) \mathsf{F}_{\mu^k / \mu^{k-1}}(z_{m_{k-1} + 1},...,z_{m_k}) F(\mu^k).
\end{split}
\end{equation}
Repeating the above argument for $k-2$,...,$1$, we see that the result of applying $D^1_{m_1}D^{2}_{m_2} \cdots D^{k}_{m_k}$ to the RHS of (\ref{Cauchy2}) is
\begin{equation}\label{sum4}
 \sum_{\mu^k \in S_k} \sum_{\mu^{k-1} \in S_{k-1}} \cdots \sum_{\mu^1 \in S_1} \prod_{i = 1}^k \mathsf{F}_{\mu^i  / \mu^{i-1} }(z_{m_{i-1} + 1},...,z_{m_i})  F(\mu^k),
\end{equation}
with the convention that $m_0 = 0$ and $\mu^0 = \varnothing$. From (\ref{Cauchy2}) the latter equals $\mathcal{W}^f({\bf m}; {\bf z})$ and so $(D^1_{m_1}D^{2}_{m_2} \cdots D^{k}_{m_k}Z^f)({\bf u}) = \mathcal{W}^f({\bf m}; {\bf u})$. Dividing both sides by $Z^f({\bf u})$ and recalling that  $\mathbb{P}^f\left( A({\bf m})  \right) = Z^f({\bf u})^{-1} \mathcal{W}^f({\bf m}; {\bf u})$ proves the lemma.
\end{proof}

In the remainder of this section, we explain how our first order operator $D^1_m$ can be used to derive an interesting recurrence relation for $\mathcal{W}^f({\bf m}; {\bf z})$ in terms of the same quantity for a system of fewer parameters. The exact statement is given in the following lemma.
\begin{lemma}\label{S3LemmaNew}
Assume the same notation as in Definition \ref{S3DefObs}. Let $\hat {\bf m} = (m_2 - 1,...,m_k - 1)$ and ${\bf z}/ \{z_i\}$ be the variable set $(z_1,...,z_{i-1},z_{i+1},...,z_N)$. Then we have
\begin{equation}\label{recurrence}
\begin{split}
&\mathcal{W}^f({\bf m}; {\bf z})   =  \prod_{j = m_1 + 1}^N \frac{z_j - sq}{1 - s z_j}  \sum_{i = 1}^{m_1} \frac{1-q}{1-sz_i}  \prod_{j = 1, j\neq i}^{m_1}\left( \frac{z_j - qz_i}{z_j - z_i} \frac{z_j - s}{1 - sz_{j}}\right)  \mathcal{W}^g(\hat {\bf m}; {\bf z}/ \{z_i\}),
\end{split}
\end{equation}
where $g: \mathsf{Sign}_{N-1}^+ \rightarrow \mathbb{R}$ is given by $g(\mu) = f(\mu + 1^{N-1})$ and $\lambda = \mu + 1^{N-1}$ is such that $\lambda_i = \mu_i + 1$ for $i \leq N-1$ and $\lambda_N = 0$. 
\end{lemma}

 This result will not be used in the remainder of the paper, but we believe it to be of separate interest as we explain now. In order to use $\mathbb{P}^f\left( A({\bf m})  \right)$ to analyze a six-vertex model it is desirable to have closed formulas for these quantities.  In this paper we will work with a particular model, for which $Z^f$ has a product form. This will allow us to find contour integral formulas for the RHS of (\ref{sum5}) as will be explained in the next section. For other boundary conditions; however, one might not be able to use  (\ref{sum5}) to derive formulas for $\mathbb{P}^f\left( A({\bf m})  \right)$ and a different approach needs to be taken.  Having a recurrence relation for $\mathcal{W}^f({\bf m}; {\bf z})$ provides a possible route for finding closed formulas for these correlation functions. In the base case, which occurs when $k = 0$ or equivalently ${\bf m} = \varnothing$, we have that $\mathcal{W}^f(\varnothing; {\bf z}) = Z^f({\bf z})$. If one has a closed formula for $Z^f({\bf z})$ then (\ref{recurrence}) can be potentially used to guess a formula for $\mathcal{W}^f({\bf m}; {\bf z})$, by matching the base case and showing it satisfies the above recurrence relation. A similar approach was used in \cite{CPS}, where a determinant formula for $\mathcal{W}^f({\bf m}; {\bf z})$ was guessed for the six-vertex model with DWBC and shown to satisfy such a recurrence relation. The key point here, is that the recurrence relation we will prove holds for general boundary conditions.

\begin{proof}
For $\mu \in \mathsf{Sign}^*_{m}$ we define $\hat \mu \in \mathsf{Sign}^*_{m-1}$ by $\hat \mu_i = \mu_i-1$ for $i = 1,...,m - 1$. We apply $D^1_{m_1}$ (in the $z$-variables) to both sides of the first line of (\ref{Cauchy2}) and get 
$$D^1_{m_1}\mathcal{W}^f({\bf m}; {\bf z}) = \sum_{\mu^k \in S_k} \sum_{\mu^{k-1} \in S_{k-1}} \cdots \sum_{\mu^1 \in S_1} \prod_{i = 2}^k \mathsf{F}_{\mu^i  / \mu^{i-1} }(z_{m_{i-1} + 1},...,z_{m_i})  F(\mu^k) D^1_{m_1}\mathsf{F}_{\mu^1 } (z_1,...,z_{m_1}).$$
In obtaining the above we used the linearity of $D^1_{m_1}$ and the convergence of the sum to change the order of the sum and operator. Using that $ D^1_{m_1}\mathsf{F}_{\mu^1 } (z_1,...,z_{m_1}) = \mathsf{F}_{\mu^1 } (z_1,...,z_{m_1}) $ whenever $\mu^1 \in S_1$, we deduce
\begin{equation}\label{ok1}
D^1_{m_1}\mathcal{W}^f({\bf m}; {\bf z})  =  \mathcal{W}^f({\bf m}; {\bf z}) .
\end{equation}
On the other hand, using the definition of $D^1_{m_1}$ and Lemma \ref{LemmaSub}, we have
\begin{equation}\label{ok2}
\begin{split}
D^1_{m_1}\mathsf{F}_{\mu^1 } (z_1,...,z_{m_1}) = \sum_{i = 1}^{m_1} \frac{1-q}{1-sz_i}  \prod_{j = 1, j\neq i}^{m_1}\left( \frac{z_j - qz_i}{z_j - z_i} \frac{z_j - s}{1 - sz_{j}}\right) \mathsf{F}_{\hat \mu^{1}}({\bf z}_{m_1} / \{z_i\}),
\end{split}
\end{equation}
where ${\bf z}_{m_1} / \{z_i\}$ stands for the variable set $(z_1,...,z_{i-1}, z_{i+1},...,z_{m_1})$. Replacing (\ref{ok2}) in our earlier expression for $D^1_{m_1}\mathcal{W}^f({\bf m}; {\bf z}) $ and utilizing (\ref{ok1}) we conclude that 
\begin{equation}\label{ok3}
\begin{split}
&\mathcal{W}^f({\bf m}; {\bf z})   =   \sum_{i = 1}^{m_1} \frac{1-q}{1-sz_i}  \prod_{j = 1, j\neq i}^{m_1}\left( \frac{z_j - qz_i}{z_j - z_i} \frac{z_j - s}{1 - sz_{j}}\right) \times \\
& \sum_{\mu^k \in S_k} \sum_{\mu^{k-1} \in S_{k-1}} \cdots \sum_{\mu^1 \in S_1} \prod_{i = 2}^k \mathsf{F}_{\mu^i  / \mu^{i-1} }(z_{m_{i-1} + 1},...,z_{m_i})  F(\mu^k)\mathsf{F}_{\hat \mu^{1}}({\bf z}_{m_1} / \{z_i\}).
\end{split}
\end{equation}
We notice from the definition of $\mathsf{F}_{\lambda / \mu}$ that for $\lambda \in \mathsf{Sign}^*_{b,1}$ and $\mu \in \mathsf{Sign}^*_{a,1}$ we have
$$\mathsf{F}_{\lambda/ \mu}(z_{a+1},...,z_b)  = \prod_{j = a+1}^b \frac{z_j - sq}{1 - s z_j}\times \mathsf{F}_{\hat \lambda/ \hat \mu}(z_{a+1},...,z_b).$$
Substituting this and the definition of $F$ in (\ref{ok3}), we arrive at
\begin{equation*}
\begin{split}
\mathcal{W}^f({\bf m}; {\bf z})   =  &\prod_{j = m_1 + 1}^N \frac{z_j - sq}{1 - s z_j}  \sum_{i = 1}^{m_1} \frac{1-q}{1-sz_i} \prod_{j = 1, j\neq i}^{m_1}\left( \frac{z_j - qz_i}{z_j - z_i} \frac{z_j - s}{1 - sz_{j}}\right)   \sum_{\hat \mu^{k} \in T_{k-1}} \sum_{\hat \mu^{k-1} \in T_{k-2}} \cdots \sum_{\hat \mu^{2} \in T_1}\\
 &\sum_{\hat \mu^{1} \in \mathsf{Sign}_{m_1 - 1}^*}  \prod_{i = 2}^k \mathsf{F}_{\hat \mu^{i}  / \hat \mu^{i-1} }(z_{m_{i-1} + 1},...,z_{m_i})  G(\hat \mu^{k}) \mathsf{F}_{\hat \mu^{1}}({\bf z}_{m_1} / \{z_i\}).
\end{split}
\end{equation*}
In the above $T_i = \mathsf{Sign}^*_{m_{i+1} - 1,i}$ and $G(\hat \mu) = \sum_{\lambda \in \mathsf{Sign}_{N}^*} \mathsf{F}_{\hat \lambda/\hat \mu}(z_{m_k+1},...,z_N ) f(\lambda)$. Using the branching relations (\ref{branchF}), (\ref{Cauchy2}) and the definition of $g$ we recognize the above identity as (\ref{recurrence}).
\end{proof}

\begin{remark}\label{keyRem}
So far in this paper we have considered the vertically inhomogeneous six-vertex model; however, one can introduce horizontal inhomogeneities as well. A particular way to do this is given in \cite{BP}, where the weights depend on an additional set $\Xi = \{\xi_j\}_{j = 0,1,...}$ of inhomogeneity parameters (our model corresponds to setting $\xi_i = 1$ for all $i$). We denote the partition function in this case by $\mathsf{F}_\lambda(u_1,...,u_m|\Xi)$ and refer the reader to (1.4) in \cite{BP} for the exact formula (the variables $s_x$ in that formula need to be set to $q^{-1/2}$). In a certain sense, one can interpret $D^k_m$ as acting on the first $k$ columns of the six-vertex model. If the first $k$ inhomogeneity parameters $\xi_0,...,\xi_{k-1}$ are all the same, then we can find an equivalent to Lemma \ref{LemmaOpFs}, but in general no such extension seems possible. Let us explain how this can be done in the case $k = 1$. If we set
$$D_m^1 := \sum_{i = 1}^n \prod_{ j \neq i} \frac{u_j - qu_i}{u_j - u_i} \prod_{j  \neq i } \left( \frac{u_j - s\xi_0^{-1}}{u_j - sq\xi_0^{-1}}\right) \frac{\phi_0(u_i)}{\phi_0(s\xi_0^{-1})} T_{u_i,s\xi_0^{-1}},$$
then one readily verifies, as done above, that $D_m^1  \mathsf{F}_\lambda(u_1,...,u_m|\Xi) = {\bf 1}_{\{\lambda_m = 0\}}\mathsf{F}_\lambda(u_1,...,u_m|\Xi)$, whenever $\lambda$ has distinct parts. The latter can be used to derive a recurrence relation for $\mathcal{W}^f({\bf m}; {\bf u}| \Xi)$ in terms of $\mathcal{W}^g(\hat {\bf m}; {\bf u}/ \{u_i\}| \tau_1 \Xi)$ (here  $\tau_1 \Xi = \{\xi_j\}_{j = 1,2,...}$), which generalizes (\ref{recurrence}). The proof is essentially the same as the one presented above.
\end{remark}

\begin{remark}
In the case of the domain wall boundary condition for the six-vertex model, which corresponds to $f(\lambda) = {\bf 1}_{\{ \lambda_1 = N-1,...,\lambda_N = 0\}}$ above, the quantity $Z^f({\bf u})^{-1}\mathcal{W}^f({\bf m}; {\bf u})$ was investigated in \cite{CPS} under the name  generalized emptiness formation probability (GEFP). In this setting, (\ref{recurrence}) naturally corresponds to equation (3.6) of \cite{CPS}, which is the key ingredient in finding closed determinant formulas for the GEFP. The derivation of (3.6) in \cite{CPS} is based on the quantum inverse scattering method, and we see that the operators $D^1_m$ (and their generalization outlined in Remark \ref{keyRem}) provide an alternative route for establishing the recurrence relation.
\end{remark}

\subsection{Action on product functions}\label{Section3.3}
Equation (\ref{sum5}) shows that understanding $\mathbb{P}^f\left( A({\bf m})  \right) $ requires knowledge of how $D^k_m$ act on the partition function $Z^f$. In this section, we will see that if $Z^f$ has a product form, then the action of the operators is relatively simple.

 In the following sequence of lemmas we investigate how $D^1_{m_1}D^{2}_{m_2} \cdots D^{k}_{m_k}$ acts on a function $F({\bf z})$ of the form $F({\bf z}) = F(z_1,...,z_m) = \prod_{i = 1}^m f(z_i)$.
\begin{lemma}\label{LemmaSC}
Let $m \geq 1$ and $1 \leq k \leq m$ be given. Suppose that $q \in (0,1)$, $s > 1$, $ u_1,...,u_m > s$ and $u_i \neq u_j$ when $i\neq j$. Let $f(z)$ be a holomorphic non-vanishing function in a neighborhood of an interval containing $s, u_1,...,u_m$. Put $F({\bf z}) = F(z_1,...,z_m) = \prod_{i = 1}^m f(z_i)$. Then we have that 
\begin{equation}\label{Op}
\begin{split}
&(D^k_mF)({\bf u}) = F({\bf u})  \cdot q^{\frac{-k(k-1)}{2}}\prod_{i = 1}^m\left( \frac{u_i - s}{u_i - sq}\right)^k  \frac{f(s)^k}{(2\pi \iota)^k k!} \times\\
& \int_\gamma \cdots \int_{\gamma} \det \left[ \frac{1} { qz_i - z_j} \right]_{i,j = 1}^k \frac{F_k(z_1,...,z_k)}{F_k(S)} \prod_{j = 1}^k\left( \prod_{i = 1}^m \frac{qz_j - u_i}{z_j - u_i} \right)\left( \frac{z_j- sq}{z_j - s}\right)^k  \frac{dz_j}{f(z_j)}.
\end{split}
\end{equation}
The contour $\gamma$ is a positively oriented contour around the points $u_1,...,u_m$, and does not contain other singularities of the integrand. Such a contour will exist, provided $u_i$ are sufficiently close to each other.
\end{lemma}
\begin{proof}
The proof is essentially the same as that of Proposition 2.11 in \cite{BorCor}. Firstly, we notice that the contours will always exist, provided $u_i$ are sufficiently close to each other. Indeed, the singularities of the integrand that are not singularities of $f$ are precisely at $u_i$, $s$, $0$, $z_i = q^{-1}z_j$ and at $s^{-1}$ (the latter one is a singularity of $F_k$). Since $u_i$ are bounded away to the right from $s$ (and hence $s^{-1}$ and $0$ ) and the function $f$ does not vanish in a neighborhood of an interval containing $u_i$ we may pick the contour $\gamma$ so as to exclude all singularities of the integrand, except possibly for $z_i = q^{-1}z_j$. However, if $u_i$ are sufficiently close then we can choose $\gamma$ to be a small circle around those points, which is disjoint from $q \cdot \gamma$. This excludes the remaining singularities. \\

We substitute in (\ref{Op}) the Cauchy determinant identity
$$\det \left[ \frac{1} { z_i - qz_j} \right]_{i,j = 1}^k = \frac{q^{\frac{k(k-1)}{2}} \prod_{ 1 \leq i < j \leq k} (z_i - z_j)(z_j - z_i)}{\prod_{i,j = 1}^k (qz_i - z_j)}$$
and calculate the residues at $z_j = u_{l_j}$. The Vandermonde determinants in the numerator prevent any of the $l_j$'s to be the same. If they are distinct and $I = \{l_1,...,l_k\}$ one calculates the residue to be 
$$\frac{1}{k!}\cdot \prod_{ j \not \in I}\left( \frac{u_j- s}{u_j- sq}\right)^k\frac{F_k(u_I)}{F_k(S)}\prod_{i \in I; j \not \in I} \frac{u_j - qu_i}{u_j - u_i} \left[  \prod_{i \in I} \frac{f(s)}{f(u_{l_i})}F({\bf u}) \right] . $$
The expression in the bracket is precisely $\prod_{i \in I} T_{u_i,s}F({\bf u})$. Summing over all permutations of $I$ removes the $k!$ above and summing over $I$ we recognize precisely $(D_m^kF)({\bf u})$ as desired.
\end{proof}

For $k \leq r \leq m$ we let $D^k_r$ be the operator that acts on the variables $u_1,...,u_r$. Then we have the following result.
\begin{lemma}
Suppose $1\leq k \leq m_1 \leq m_2 \leq \cdots \leq m_k \leq m$. Denote by $M_i = \{1,...,m_i\}$ for $i = 1,...,k$. Then 
\begin{equation}\label{EqMOp}
\begin{split}
D^1_{m_1}D^{2}_{m_2} \cdots D^{k}_{m_k} = \sum_{i_1 \in M_1} \sum_{i_2 \in M_2 / I_1} \cdots \sum_{i_k \in M_k / I_{k-1}} \prod_{r = 1}^k\left( \prod_{j \in M_{r} / I_r} \frac{u_j - q u_{i_r}}{u_j - u_{i_r}} \frac{u_j - s}{u_j - sq} \right) \times \\
\frac{\prod_{r = 1}^{k} \frac{1-q}{1 - su_{i_{r}}}\left( \frac{u_{i_{r}} - sq}{1 - su_{i_{r}}} \right)^{r-1} }{F_{k}(S)} \prod_{r = 1}^k T_{u_{i_r},s} \hspace{5mm} \mbox{, where } I_r = \{i_1,...,i_r\}.
\end{split}
\end{equation}
The above is understood as an equality of operators on functions in $m$ variables.
\end{lemma}
\begin{proof}
We proceed by induction on $k$ with base case $k = 1$ being just the definition of $D_{m_1}^1$. Suppose the result is known for $k$ and we wish to show it for $k+1$. Substituting the definition of $D^{k+1}_{m_{k+1}}$ and the induction hypothesis we have
$$D^1_{m_1}D^{2}_{m_2} \cdots D^{k}_{m_k} D^{k+1}_{m_{k+1}}= \sum_{i_1 \in M_1} \sum_{i_2 \in M_2 / I_1} \cdots \sum_{i_k \in M_k / I_{k-1}} \prod_{r = 1}^k\left( \prod_{j \in M_{r} / I_r} \frac{u_j - q u_{i_r}}{u_j - u_{i_r}} \frac{u_j - s}{u_j - sq} \right) \times $$
$$ \frac{\prod_{r = 1}^{k} \frac{1}{1 - su_{i_{r}}}\left( \frac{u_{i_{r}} - sq}{1 - su_{i_{r}}} \right)^{r-1} }{F_{k}(S)} \prod_{r = 1}^{k} T_{u_{i_r},s} \sum_{\substack{ I \subset M_{k+1}\\  |I| = k+1}} \prod_{\substack{ i \in I \\ j \in M_{k+1}/I }} \hspace{-4mm} \frac{u_j - qu_i}{u_j - u_i} \hspace{-2mm} \prod_{ j \in M_{k+1}/I} \left( \frac{u_j - s}{u_j - sq}\right)^{k + 1} \frac{F_{k+1}(u_I)}{F_{k+1}(S)}\prod_{i \in I} T_{u_i,s}.$$
Suppose that $i_k \not \in I$. Then 
$$T_{u_{i_{k}},s } \prod_{i \in I;j \in M_{k+1}/I} \frac{u_j - qu_i}{u_j - u_i} \prod_{j \in M_{k+1}/I} \left( \frac{u_j - s}{u_j - sq}\right)^{k + 1} \frac{F_{k+1}(u_I)}{F_{k+1}(S)}\prod_{i \in I} T_{u_i,s} = 0,$$
since one of the factors in the above expressions is $(u_{i_{k}} - s)^{k+1}$ and it vanishes when $u_{i_k} = s$. It follows that to get a non-zero contribution we must have $i_k \in I$. Repeating the argument we see that $i_r \in I$ for all $r = 1,..., k$. Thus $I = I_k \sqcup \{i_{k+1}\}$ for some $i_{k + 1} \in M_{k+1}$ are the only cases that lead to a non-zero contribution. If $I$ does have this form we see that 
$$ \prod_{r = 1}^k T_{u_{i_r},s} \prod_{i \in I; j \in M_{k+1}/I} \frac{u_j - qu_i}{u_j - u_i} \prod_{j \in M_{k+1}/I} \left( \frac{u_j - s}{u_j - sq}\right)^{k + 1} \frac{F_{k+1}(u_I)}{F_{k+1}(S)}\prod_{i \in I} T_{u_i,s} = $$
$$ = \prod_{ j \in M_{k+1}/I} \frac{u_j - qu_{i+1}}{u_j - u_{i+1}} \left(\frac{u_j - qs}{u_j - s}\right)^k\prod_{j \in M_{k+1} / I } \left( \frac{u_j - s}{u_j - sq}\right)^{k + 1} \frac{F_{k+1}(s,...,s,u_{i_{k+1}})}{F_{k+1}(S)}\prod_{r = 1}^{k+1} T_{u_{i_r},s} = $$
$$ = \prod_{ j \in M_{k+1}/I} \frac{u_j - qu_{i+1}}{u_j - u_{i+1}}\prod_{j \in M_{k+1} / I } \left( \frac{u_j - s}{u_j - sq}\right) \frac{F_{k+1}(s,...,s,u_{i_{k+1}})}{F_{k+1}(S)}\prod_{r = 1}^{k+1} T_{u_{i_r},s}.$$
From Lemma \ref{LemmaSub} we know that $F_{k+1}(s,...,s,u_{i_{k+1}}) =F_k(S)\frac{1-q}{1 - su_{i_{k+1}}}\left( \frac{u_{i_{k+1}} - sq}{1 - su_{i_{k+1}}} \right)^k$. Subsituting this above and cancelling $F_{k}(S)$ we get
$$D^1_{m_1}D^{2}_{m_2} \cdots D^{k}_{m_k} D^{k+1}_{m_{k+1}}=\sum_{i_1 \in M_1} \sum_{i_2 \in M_2 / I_1} \cdots \sum_{i_k \in M_k / I_{k-1}}\sum_{i_{k+1} \in M_{k+1} / I_{k}} \prod_{r = 1}^{k+1}\left( \prod_{j \in M_{r} / I_r} \frac{u_j - q u_{i_r}}{u_j - u_{i_r}} \frac{u_j - s}{u_j - sq}  \right)$$
$$ \times \frac{\prod_{r = 1}^{k+1} \frac{1-q}{1 - su_{i_{r}}}\left( \frac{u_{i_{r}} - sq}{1 - su_{i_{r}}} \right)^{r-1} }{F_{k+1}(S)} \prod_{r = 1}^{k+1} T_{u_{i_r},s}.$$
This proves the case $k+1$ and the general result now follows by induction.
\end{proof}

\begin{lemma}\label{LemmaMC}
Suppose $1\leq k \leq m_1 \leq m_2 \leq \cdots \leq m_k \leq m$. Suppose that $q \in (0,1)$, $s > 1$, $ u_1,...,u_m > s$ and $u_i \neq u_j$ when $i\neq j$. Let $f(z)$ be a holomorphic non-vanishing function in a neighborhood of an interval containing $s, u_1,...,u_m$. Put $F({\bf z}) = F(z_1,...,z_m) = \prod_{i = 1}^m f(z_i)$. Then we have that
\begin{equation}\label{MultiOp}
\begin{split}
(D^1_{m_1}D^{2}_{m_2} \cdots D^{k}_{m_k}F)({\bf u}) = F({\bf u}) \cdot \prod_{r= 1}^k \hspace{-1mm}  \left( \prod_{i = 1}^{m_r}\frac{u_i - s}{u_i - sq}\right)\hspace{-1mm}    \frac{f(s)^k}{(2\pi \iota)^k} \hspace{-1mm} \int_\gamma \cdots \int_{\gamma}  \prod_{ 1 \leq i < j \leq k}  \frac{ z_i - z_j}{ z_i - qz_j} \times \\
  \frac{\prod_{r = 1}^{k} \frac{1-q}{1 - sz_{r}}\left( \frac{z_{r} - sq}{1 - sz_{r}} \right)^{r-1} }{F_{k}(S)}  \prod_{r = 1}^k\left( \prod_{i = 1}^{m_r} \frac{qz_r - u_i}{z_r - u_i} \right)\left( \frac{z_r- sq}{z_r - s}\right)^{k-r+1}  \frac{dz_r}{f(z_r)z_r(q-1)}.
\end{split}
\end{equation}
The contour $\gamma$ is a positively oriented contour around the points $u_1,...,u_m$, and does not contain other singularities of the integrand. Such a contour will exist, provided $u_i$ are sufficiently close to each other.
\end{lemma}
\begin{proof}
The proof is similar to that of Lemma \ref{LemmaSC} and by the same arguments we know that the contour $\gamma$ exists, provided $u_i$ are sufficiently close.\\

We calculate the residues at $z_{r} = u_{i_r}$. The Vandermonde determinant in the numerator prevents any of the $i_r$'s to be the same. The residue at $z_1 = u_{i_1}, ... ,z_k = u_{i_k}$ is given by 
\begin{equation*}
\begin{split}
\prod_{r= 1}^k\left( \prod_{i = 1}^{m_r}\frac{u_i - s}{u_i - sq}\right)\prod_{ 1 \leq r <p \leq k}  \frac{u_{i_r} - u_{i_p}}{ u_{i_r} - qu_{i_p}} \prod_{r = 1}^k\frac{1}{u_{i_r}(q-1)}\frac{\prod_{r = 1}^{k} \frac{1-q}{1 - su_{i_r}}\left( \frac{u_{i_r} - sq}{1 - su_{i_r}} \right)^{r-1} }{F_{k}(S)}\times\\
\prod_{r = 1}^k (q-1)u_{i_r} \left( \prod_{i = 1, i\neq i_r}^{m_r} \frac{qu_{i_r} - u_i}{u_{i_r} - u_i} \right)\left( \frac{u_{i_r}- sq}{u_{i_r} - s}\right)^{k-r+1}  \left[ \frac{F({\bf u}) f(s)^k }{\prod_{r = 1}^k f(u_{i_r})}\right].
\end{split}
\end{equation*}
Performing some cancellations and recognizing the term inside the square brackets as $\prod_{r = 1}^{k} T_{u_{i_r},s}F({\bf u})$ we recognize precisely the term on the RHS of (\ref{EqMOp}) corresponding to $i_1,...,i_k$. Summing over all the residues we arrive at the desired identity.
\end{proof}

\section{Weak convergence of $(Y_1^1,...,Y_k^k)$}\label{Section4}
In this section we use our results from Section \ref{Section3} to derive formulas for $\mathbb{P}_{{\bf u}, {\bf v}}( Y_1^1 \leq m_1, ..., Y_k^k \leq m_k)$. Afterwards we specialize our formulas to the case when all $u$ and all $v$ parameters are the same and show that under the scaling of Theorem \ref{theorem2} the joint CDFs of the vectors $ (Y_1^1 , ..., Y_k^k )$ converge to a fixed function as the size of the six-vertex model increases. We finish by identifying the limit as the joint CDF of the right edge of the GUE-corners process of rank $k$ and proving Theorem \ref{theorem1}.

\subsection{Pre-limit formulas}\label{Section4.1} The goal of this section is to use the results from Section \ref{Section3} to obtain formulas for $\mathbb{P}_{{\bf u}, {\bf v}}( Y_1^1 \leq m_1, ..., Y_k^k \leq m_k)$, where $m_i \in \mathbb{N}$ for $i = 1,...,k$ and $Y_i^j$ are defined in Section \ref{Section1.2}. We summarize the result in the following proposition.
\begin{proposition}\label{corePropProb} Fix parameters as in Definition \ref{parameters}. Let $k$ and $m_i$ for $i = 1,...,k$ be positive integers such that $1 \leq k \leq m_1 \leq m_2 \leq \cdots \leq m_k \leq N$. Then we have
\begin{equation}\label{coreProb}
\begin{split}
\mathbb{P}_{{\bf u}, {\bf v}}( Y_1^1 \leq m_1, ..., Y_k^k \leq m_k)  = \left(\prod_{j = 1}^{M} \frac{1 - qsv_j}{1 - sv_j} \right)^k \prod_{r= 1}^k   \left( \prod_{i = 1}^{m_r}\frac{u_i - s}{u_i - sq}\right)   \frac{q^{-k}}{(2\pi \iota)^k}   \int_\gamma   \cdots   \int_{\gamma}  \\
   \prod_{1 \leq i < j \leq k}  \frac{z_i - z_j}{z_i - qz_j}
\prod_{r = 1}^k \left( \prod_{i = 1}^{m_r}   \frac{qz_r - u_i}{z_r - u_i}  \right)   \left(  \frac{z_r- sq}{z_r - s} \right)^{k-r+1}  \prod_{j = 1}^M   \frac{1 - z_rv_j}{1 - qz_rv_j} \frac{dz_r}{z_r(1-sz_r)}.
\end{split}
\end{equation}
The contour $\gamma$ is a positively oriented contour that contains $u_i$'s and excludes all other singularities of the integrand. Such a contour will exist, provided $u_i$ are sufficiently close to each other.
\end{proposition}
\begin{proof}
In what follows we adopt notation from Sections \ref{Section1.2} and \ref{Section2.2}.

Let $E = \{\omega \in \mathcal{P}_N: \mbox{$\omega(0,j) = (0,1;0,1)$ for $i = 1,...,N$ }\}$. From our discussion in Section \ref{Section2}, we know that $\mathbb{P}_{{\bf u}, {\bf v}}(E) = 1$. Consider the map $h: E \rightarrow \mathcal{P}_N$, given by $h(\omega)(i,j) = \omega(i+1,j)$. I.e. $h(\omega)$ is just the collection of up-right paths $\omega$, with the zeroth column deleted. One readily observes that $h$ is a bijection and the distribution of $h(\omega)$, induced by the distribution of $\omega$, is given by $\mathbb{P}^g$, where $g(\mu) =  \frac{(-s)^N}{(q;q)_N}\times f(\mu + (1)^N;{\bf v}, \rho)$. We recall that $\lambda = \mu + (1)^N$ is the signature with $\lambda_i = \mu_i + 1$ for $i = 1,...,N$ and $f(\lambda; {\bf v}, \rho)$ is given in (\ref{feq}).

Indeed, we have for $\omega \in E$, that
$$\mathcal{W}^g(h(\omega); {\bf u}) = \frac{(-s)^N}{(q;q)_N}\prod_{i = 1}^{\infty} \prod_{j = 1}^N w_{u_j}(\omega(i,j)) f(\lambda^N(\omega);{\bf v}, \rho) =  \frac{(-s)^N}{(q;q)_N}\prod_{i = 1}^N\frac{1 - su_i}{u_i - s} \times \mathcal{W}^f(\omega; {\bf u} ).$$
The above shows that the weights $\mathcal{W}^g(h(\omega); {\bf u}) $ are constant multiples of $\mathcal{W}^f(\omega; {\bf u})$, and so the probability distributions they define are the same. The partition function $Z^g({\bf u})$ differs from $Z^f({\bf u})$ by the same constant factor $ \frac{(-s)^N}{(q;q)_N}\prod_{i = 1}^N\frac{1 - su_i}{u_i - s}$, and by (\ref{PFEqn4}) equals
\begin{equation}\label{partS}
Z^g({\bf u}) = \prod_{i = 1}^N \left( \prod_{j = 1}^M \frac{1 - qu_i v_j}{1 - u_iv_j}\right).
\end{equation}

One easily observes the following equality of events
$$ \{\omega \in E: \lambda^{m_i}(h(\omega)) \in \mathsf{Sign}_{m_i,i}^*, \hspace{1mm} i = 1,...,k\}  = \{\omega \in E :  Y_i^i(\omega) \leq m_i, \hspace{1mm} i = 1,...,k \}.$$
For example $\lambda^{m_1}(h(\omega))\in \mathsf{Sign}_{m_1,1}^*$ is equivalent to $\lambda^{m_1}_{m_1}(\omega) = 1$, which by the conservation of arrows in the region $\{ (1,y) \in D_N: y = 1,...,m_1\}$ is equivalent to $Y_1^1(\omega) \leq m_1$. The above equality of events, coupled with $\mathbb{P}_{{\bf u}, {\bf v}}(E) = 1$ and the previous two paragraphs implies 
\begin{equation}\label{S4eqEv}
\mathbb{P}_{{\bf u}, {\bf v}}( Y_1^1 \leq m_1, ..., Y_k^k \leq m_k) = \mathbb{P}^g(A({\bf m})),
\end{equation}
where $A({\bf m})$ is as in (\ref{ObservableEvent}). In view of (\ref{sum5}) and (\ref{partS}), we conclude that if $u_1,...,u_N$ are pairwise distinct
\begin{equation}\label{sum6}
  \hspace{-6mm} \mathbb{P}_{{\bf u}, {\bf v}}( Y_1^1 \leq m_1, ..., Y_k^k \leq m_k) = \frac{(D^1_{m_1}D^{2}_{m_2} \cdots D^{k}_{m_k}Z^g)({\bf u})}{Z^g({\bf u})}, \mbox{ where } Z^g({\bf u}) =\prod_{i = 1}^N  \hspace{-1mm} \left(\prod_{j = 1}^M \frac{1 - qu_i v_j}{1 - u_iv_j}  \right)  \hspace{-1mm}.
\end{equation}

The result of the proposition now follows from (\ref{sum6}) and Lemma \ref{LemmaMC} when $u_1,...,u_N$ are pairwise distinct. By continuity it also holds if some are equal.
\end{proof}

\subsection{Asymptotic analysis}\label{Section4.2}

While most of the results below can be extended to a more general choice of parameters, we keep discussion simple and assume that all $u$ and all $v$ parameters are the same, and that $\frac{s + s^3}{2} > u > s$. With this in mind we have the following definition.

\begin{definition}\label{parameters2}
Let $N,M \in \mathbb{N}$ and fix $q \in (0,1)$, $s = q^{-1/2}$, $\frac{s + s^3}{2} > u > s$ and $v \in (0, u^{-1})$. We denote by $\mathbb{P}_{u,v}^{N,M}$ the probability measure $\mathbb{P}_{{\bf u}, {\bf v}}$ of Definition \ref{parameters}, with $u_i = u$ and $v_j = v$ for $i = 1,...,N$ and $j = 1,...,M$. 
\end{definition}
With the above definition, we have the following consequence of Proposition \ref{corePropProb}. If $1 \leq k \leq m_1 \leq m_2 \leq \cdots \leq m_k \leq N$, then
\begin{equation}\label{mainEq}
\begin{split}
&\mathbb{P}_{{\bf u}, {\bf v}}( Y_1^1 \leq m_1, ..., Y_k^k \leq m_k) = \frac{q^{-k}}{(2\pi \iota)^k} \int_\gamma \cdots \int_{\gamma}
  \prod_{1 \leq i < j \leq k}\frac{z_i - z_j}{z_i - qz_j} \times
\\ 
&\prod_{r = 1}^k\left(  \frac{qz_r - u}{z_r - u}  \frac{u - s}{u- sq}\right)^{m_r}\left( \frac{z_r- sq}{z_r - s}\right)^{k-r+1}  \left(\frac{1 - z_rv}{1 - qz_rv}\frac{1 - qsv}{1 - sv} \right)^M \frac{dz_r}{z_r(1 - sz_{r})},
\end{split}
\end{equation}
where $\gamma$ is a contour, containing $u$ and excluding other singularities of the integrand. Equation (\ref{mainEq}) is prime for asymptotic analysis and we use it to prove the following proposition.
\begin{proposition}\label{propEq5}
Let $\mathbb{P}_{u,v}^{N,M}$ be as in Definition \ref{parameters2}. Put $a = \frac{v^{-1} (u - sq) (u-s)}{u(v^{-1}-sq)(v^{-1}-s)}$ and $c = (2a_2)^{1/2}b_1^{-1}$, where
$$a_2 = \frac{(1-q)v^{-1}}{(v^{-1}-s)(v^{-1}-sq)} \left[ \frac{(q+1)s - 2v^{-1}}{(v^{-1}-s)(v^{-1}-sq)} - \frac{(q+1)s - 2u}{(u-s)(u-sq)} \right] \mbox{ and }  b_1 = \frac{1}{u-s} - \frac{1}{q^{-1}u - s}.$$
Let $\gamma > a$ and assume that $N(M) \geq \gamma \cdot M$ for all $M >> 1$. Then for any $k \geq 1$ and $x_1\leq ...\leq x_k$, $x_i \in \mathbb{R}$ we have
\begin{equation}\label{mainEq5}
\lim_{ M \rightarrow \infty} \mathbb{P}_{u,v}^{N,M}\left(  \frac{ Y_i^i - aM}{c\sqrt{M}}  \leq x_i; i = 1,...,k\right) =  \det\left[\frac{1}{2\pi \iota} \int_{1 + \iota \mathbb{R}}y^{j-i-1 }e^{y^2/2 + x_jy } dy  \right]_{i,j = 1}^k.
\end{equation}
\end{proposition}
\begin{proof}
Put $m_i =  aM + c x_i\sqrt{M} + h_i$ for $i = 1,...,k$, where $h_i \in (-1,1)$ are such that $m_i \in \mathbb{N}$ and $m_1 \leq m_2\leq ... \leq m_k$ for $M$ sufficiently large. Using (\ref{mainEq}) we reduce the proof of (\ref{mainEq5}) to the following statement
\begin{equation}\label{mainEq2}
\begin{split}
\lim_{M \rightarrow \infty} \frac{q^{-k}}{(2\pi \iota)^k} \int_{\gamma_M}  \cdots \int_{\gamma_M} &
  \prod_{1 \leq i < j \leq k}\frac{z_i - z_j}{z_i - qz_j}  \prod_{r = 1}^{k}  \left(  \frac{z_r- sq}{z_r - s}   \right)^{k-r+1}  \frac{e^{MG(z_r) + c\sqrt{M}x_rg(z_r) + h_rg(z_r)}dz_r}{z_r(1 - sz_{r})} \\
&=  \det\left[\frac{1}{2\pi \iota} \int_{1 + \iota \mathbb{R}}y^{j-i-1 }e^{y^2/2 + x_jy } dy  \right]_{i,j = 1}^k,
\end{split}
\end{equation}
where $\gamma_M$ are contours that contain $u$ and do not include $0$, $q^{-1}v^{-1}$, $s$ or points from $q\cdot \gamma_M$ and 
\begin{equation}\label{DefG}
G(z) = a\log\left( \frac{qz - u}{z - u}  \frac{u - s}{u- sq} \right) + \log\left(\frac{v^{-1} - z}{v^{-1} - qz}\frac{v^{-1} - sq}{v^{-1} - s} \right) \mbox{ and } g(z) = \log\left( \frac{qz - u}{z - u}  \frac{u - s}{u- sq} \right).
\end{equation}

Our goal is to find the $M \rightarrow \infty$ limit of the LHS of $(\ref{mainEq2})$ and match it with the RHS. Let us briefly explain what the strategy is. We will find specific contours $\gamma_M = \gamma_M(0) \sqcup \gamma_M(1)$, such that $Re[ G(z)] < 0$ on $\gamma_M(1)$ and the integrand (upon a change of variables) has a clear limit on $\gamma_M(0)$. The condition $Re[G(z)] < 0$
will show that the integral over $\gamma_M(1)$ decays exponentially fast, and hence does not contribute to the limit. The non-vanishing contribution, coming from $\gamma_M(0)$, will then be shown to equal the RHS of (\ref{mainEq2}). The latter approach is typically referred to as the {\em method of steepest descent} in the literature. \\

To simplify formulas in the sequel we denote $v^{-1}$ by $w$. We start by analyzing the functions $G$ and $g$. From (\ref{DefG}) we have $G'(z) = a \frac{(1-q)u}{(z-u)(qz-u)} - \frac{(1-q)w}{(z-w)(qz-w)}$ and so $G(s) = 0$ and $G'(s) = 0$ by our choice of $a$. We observe 
$$\frac{d}{dy} Re \left[ G(s + \iota y)\right] \hspace{-0.5mm}=  Im \hspace{-1mm} \left[ \frac{(1-q)w}{(s + \iota y-w)(q(s + \iota y)-w)} - \frac{a(1-q)u}{(s + \iota y -u)(q(s + \iota y)-u)} \right] \hspace{-1.5mm}=w(1-q)y A(y),$$
where 
$$A(y) =  \frac{(u-s)(u-sq)}{(w-s)(w-sq)} \frac{2q - (q+1)u}{((s-u)^2 + y^2)( (qs - u)^2 + q^2y^2)} -   \frac{2q - (q+1)w}{((s-w)^2 + y^2)( (qs - w)^2 + q^2y^2)}.$$
We observe that 
$$ A(0) = \frac{1}{(w-s)(w-sq)} \left[ \frac{2q - (q+1)u}{(u-s)(u-sq)} - \frac{2q - (q+1)w}{(w-s)(w-sq)}\right] < 0,$$
where we used $u,w > s$, $q \in (0,1)$ and $w > u$. In addition, if we put the two fractions in the definition of $A(y)$ under a common denominator, we see that the sign of $A(y)$ agrees with the sign of a certain quadratic polynomial in $y^2$ with a positive leading coefficient. This implies that as $y$ goes from $0$ to $\infty$, $A(y)$ is initially negative and then becomes positive, i.e. $ Re \left[ G(s + \iota y)\right]$ initially decreases and then increases in $y > 0$. A similar statement holds when $y < 0$. In particular, we can find $\epsilon > 0$ small such that $Re \left[ G(s + \iota y) \right]\leq 0$ for $y \in [-\epsilon, \epsilon]$ and $Re \left[ G(s \pm  \iota \epsilon) \right] < 0$. 

Using $u,w > s$, $q \in (0,1)$ and $w > u$, we notice
$$G(0) = a\log\left(   \frac{u - s}{u- sq} \right) + \log\left(\frac{w - sq}{w - s} \right) =: -c_0 < 0.$$
We next observe that 
$$a_2 = G''(s) = \frac{(1-q)w}{(w-s)(w-sq)} \left[ \frac{(q+1)s - 2w}{(w-s)(w-sq)} - \frac{(q+1)s - 2u}{(u-s)(u-sq)} \right] > 0.$$
Consequently, we have that near $s$ we have $G(z) = a_2(z-s)^2 + a_3(z - s)^3 + \cdots$ and $g(z) = b_1(z-s) + b_2(z-s)^2 + \cdots$. In particular, if we choose $\epsilon$ sufficiently small we can ensure that 
\begin{equation}\label{expest}
\mbox{ $|G(z) - a_2(z-s)^2| \leq R|z-s|^3$ and $|g(z) - b_1(z-s)| \leq R|z-s|^2$, when $|z-s| < \epsilon$,}
\end{equation}
where $R$ can be taken to be $|b_2| + |a_3| + 1$. For the remainder we fix $\epsilon > 0$ sufficiently small so that (\ref{expest}) holds and $c_\epsilon := - Re \left[ G(s \pm  \iota \epsilon) \right] > 0$. \\

In what follows we define the contour $\gamma_M$. Let $B_1$ and $C_1$ be the points $u$ and $q^{-1}u$ in the complex plane respectively, and also denote  $w$ and $q^{-1}w$ by $B_2$ and $C_2$ respectively. For $i = 1,2$ we let $\omega_i$ be the Appolonius circle\footnote{For $r \in (0,1)$, the Appolonius circle of a segment $BC$ with ratio $r$ is the set of points $X$ such that $\frac{XB}{XC} = r$. For points inside the circle we have $\frac{XB}{XC} < r$ and for those outside $\frac{XB}{XC} > r$. If $X$ and $Y$ denote the (unique) points on the line $BC$, which satisfy $\frac{XB}{XC} = r = \frac{YB}{YC}$, with $X$ lying inside and $Y$ outside the segment $BC$, the Appolonius circle of $BC$ with ratio $r$, is the circle with diamater $XY$. } of the segment $B_iC_i$, which passes through the origin. By properties of the Appolonius circle, we know that $X_iY_i$ is a diameter for $\omega_i$, where $X_i = 0$, $Y_1 = \frac{2u}{1+q}$ and $Y_2 = \frac{2w}{1+q}$. Observe that since $w > u$ we have that $\omega_1$ is internally tangent to $\omega_2$ at $0$.

Let $s \pm \iota y_1$ be the points on $\omega_1$ that lie on the vertical line through $s$, with $y_1 > 0$. $\gamma_M$ starts from $s - \iota y_1$ and follows $\omega_1$ to $s + \iota y_1$ counterclockwise, afterwards it goes down to to $s + \iota M^{-1/2}$, follows the right half of the circle of radius $M^{-1/2}$ around $s$ to $s - \iota M^{-1/2}$, and then continues down to $s - \iota y_1$. See the left part of  Figure \ref{S4_1}. Observe that by construction $u$ is enclosed by $\gamma_M$ and $0,s,s^{-1}$ are not. In addition, we notice that since $q\cdot \omega_1$ lies to the left of $q \cdot Y_1 = \frac{2uq}{1+q}$ and the latter is less than $s$ if $u < \frac{s + s^3}{2}$, then $q \cdot \gamma_M$ lies to the left of $\gamma_M$. This means that $\gamma_M$ satisfies the conditions we stated after (\ref{mainEq2}).

\vspace{-4mm}
\begin{figure}[h]
\centering
\scalebox{0.6}{\includegraphics{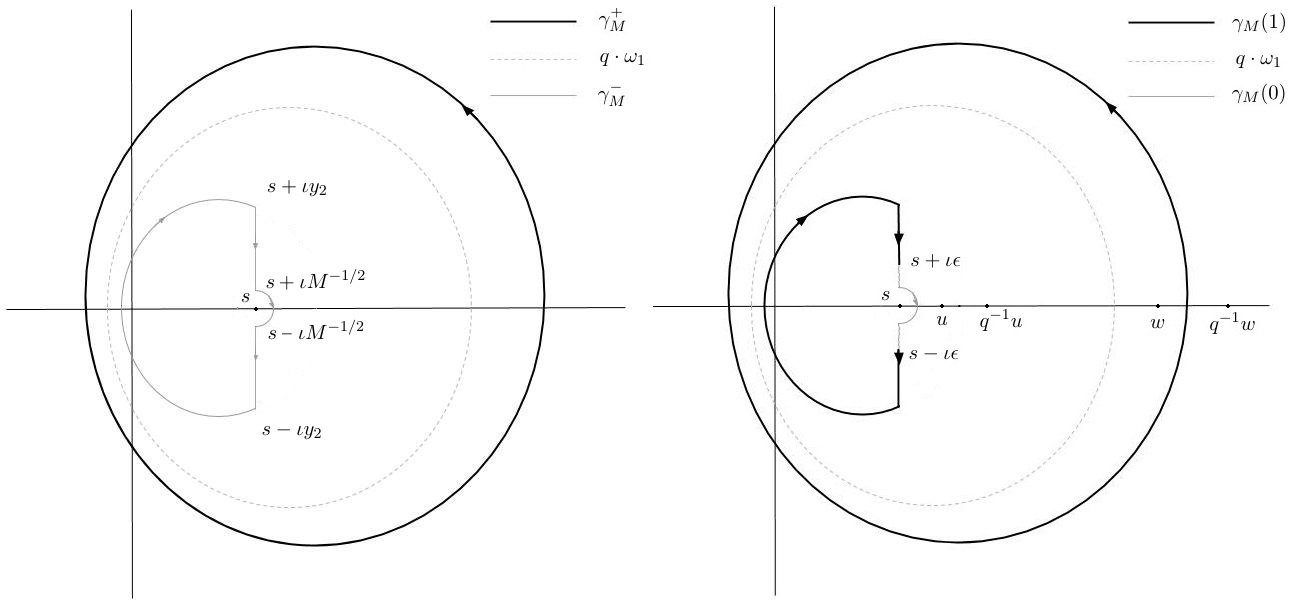}}
\caption{The contour $\gamma_M$ (left) and $\gamma_M(0)$ and $\gamma_M(1)$ (right). }
\label{S4_1}
\end{figure}

We now investigate the real part of $G(z)$ on $\gamma_M$. Using the properties of the Appolonius circle, we see that for $z \in \omega_1$ we have
$$Re \left[a\log\left( \frac{qz- u }{z - u}  \frac{u - s}{u- sq} \right)\right] = a\log\left( \frac{|z - q^{-1}u|}{|z - u|}\right) + a\log\left(  \frac{qu - qs}{u- sq} \right) =  a\log\left( \frac{|X_2C_2|}{|X_2B_2|}\right) + a\log\left(  \frac{qu - qs}{u- sq} \right),$$
while on the other hand
$$Re \left[\log\left( \frac{w - z}{w - qz}  \frac{w - sq}{w- s} \right)\right] = \log\left( \frac{|w - z|}{|q^{-1}w - z|}\right) + \log \left(  \frac{w - sq}{qw- qs} \right) \leq 
\log\left( \frac{|X_1B_1|}{|X_1C_1|}\right) + \log \left(  \frac{w - sq}{qw- qs} \right).$$
Adding the above inequalities we see that for $z \in \omega_1$, we have
\begin{equation}\label{estG}
Re\left[ G(z) \right] \leq Re\left[ G(0) \right] = -c_0< 0.
\end{equation}
Equation (\ref{estG}) in particular says that $Re[G(s \pm \iota y_1)] \leq - c_0$, and since $Re[G(s + \iota y)]$ decreases and then increases in $|y|$, while $Re[G(s \pm \iota \epsilon)] \leq - c_\epsilon < 0$, we know that $Re[G(s \pm \iota y)] \leq - \min(c_0, c_\epsilon)$, for $|y| \in [\epsilon, y_2]$. Let us denote by $\gamma_M(0)$ the portion of $\gamma_M^-$, which connects $s \pm \iota \epsilon$ near $s$, and by $\gamma_M(1)$, the rest of $\gamma_M$ - see the right part of Figure \ref{S4_1}.

The above estimates show that $Re[G(z)] \leq - \min(c_0, c_\epsilon)$ for $z \in \gamma_M(1)$. This suggests, that asymptotically, we may ignore $\gamma_M(1)$, as its contribution goes to zero exponentially fast. Explicitly, if we denote by $H(z_1,...,z_k)$ the integrand in (\ref{mainEq2}) then we have
\begin{equation}\label{truncateSpec}
\lim_{M \rightarrow \infty} \left| \int_{\gamma_M} \cdots  \int_{\gamma_M} H(z_1,...,z_k) d{\bf z} - \int_{\gamma_M(0)} \cdots  \int_{\gamma_M(0)} H(z_1,...,z_k) d{\bf z} \right| = 0.
\end{equation}
We isolate the proof of the above statement in Proposition \ref{PropTr} below and continue with the proof of (\ref{mainEq2}). \\

In view of (\ref{truncateSpec}), the limit as $M \rightarrow \infty$ of the LHS of (\ref{mainEq2}) is the same as that of
\begin{equation}\label{mainEq3}
\frac{q^{-k}}{(2\pi \iota)^k} \int_{\gamma_M(0)} \hspace{-5mm} \cdots \int_{\gamma_M(0)} 
  \prod_{1 \leq i < j \leq k}\frac{z_i - z_j}{z_i - qz_j}  \prod_{r = 1}^{k}  \left(  \frac{z_r- sq}{z_r - s}   \right)^{k-r+1}  \frac{e^{MG(z_r) + c \sqrt{M}x_rg(z_r) - a_rg(z_r)}dz_r}{z_r(1 - sz_{r})},
\end{equation}
We do the change of variables $y_i = (z_i - s)M^{1/2}$ and set $\Gamma$ to be the contour that goes up from $-\iota \infty$ to $-\iota $, follows the right half of the circle of radius $1$ around $0$ to $  \iota $, and then continues up to $ \iota \infty$. Using (\ref{expest}), we observe that (\ref{mainEq3}) equals
\begin{equation}\label{mainEq4}
\begin{split}
\frac{(-q)^{-k}}{(2\pi \iota)^k} \int_\Gamma \cdots \int_{\Gamma}
  \prod_{1 \leq i < j \leq k}\frac{y_i - y_j}{M^{-1/2}(y_i - qy_j) + s(1-q)} \times  \\
 \prod_{r = 1}^{k}  \left( \frac{M^{-1/2}y_r +s(1-q)}{y_r}\right)^{k-r+1} \frac{e^{a_2y_r^2 +c b_1x_ry_r + O(M^{-1/2})}{\bf 1}_{\{|y_r| \leq M^{-1/2}\epsilon\}}dy_r}{(1 - s^2 - M^{-1/2}sy_r)(M^{-1/2} y_r + s)}.
\end{split}
\end{equation}
 The pointwise limit of the integrand as $M \rightarrow \infty$ is given by
$$(-q)^k\prod_{1 \leq i < j \leq k}(y_i - y_j) \prod_{r = 1}^{k}  e^{a_2y_r^2 + cb_1x_ry_r }\frac{dy_r}{y_r^{k-r+1}}$$
Since $a_2 > 0$ we see that the integrand in (\ref{mainEq4}) is dominated by $C \prod_{r = 1}^ke^{-a_2 |y_r|^2/2}$. It follows from the Dominated Convergence Theorem that the limit as $M \rightarrow \infty$ of (\ref{mainEq3}), and hence the LHS (\ref{mainEq2}) is
\begin{equation}\label{mainEq42}
\frac{1}{(2\pi \iota)^k} \int_\Gamma \cdots \int_{\Gamma}\prod_{1 \leq i < j \leq k}(y_i - y_j) \prod_{r = 1}^{k} e^{a_2y_r^2 + cb_1x_ry_r }\frac{dy_r}{y_r^{k-r+1}}.
\end{equation}

What remains is to show that (\ref{mainEq42}) and the RHS of (\ref{mainEq2}) agree. We perform the change of variables $y_i \rightarrow (2a_2)^{-1/2} y_{i}$, replace $\prod_{1 \leq i < j \leq k}(y_i - y_j)\prod_{r = 1}^{k}\frac{1}{y_r^{k-r+1}} $ with $\det[ y_i^{i-j-1}]_{i,j = 1}^k$, set $\Gamma' = (2a_2)^{1/2}\Gamma$  and use $(2a_2)^{-1/2}b_1 c  = 1$. This allows us to rewrite (\ref{mainEq42}) as
$$\frac{1}{(2\pi \iota)^k} \int_{\Gamma'} \hspace{-1mm} \cdots \int_{\Gamma'} \det[ y_i^{i-j-1}]_{i,j = 1}^k\prod_{r = 1}^{k} e^{y_r^2/2 + x_ry_r }dy_r.$$
Using properties of determinants we rewrite the above as
\begin{equation}\label{mainEq6}
 \det\left[\frac{1}{2\pi \iota} \int_{\Gamma'}y^{i -  j-1}e^{y^2/2 + x_iy } dy  \right]_{i,j = 1}^k.
\end{equation}
By Cauchy's theorem and the rapid decay of $e^{y^2/2}$ near $\pm \iota \infty$, we may deform $\Gamma'$ to $1 + \iota \mathbb{R}$, without changing the value of the integral. Replacing the matrix in the determinant with its transpose, finally transforms (\ref{mainEq6}) into the RHS of (\ref{mainEq2}).

\end{proof}

\begin{proposition}\label{PropTr}
Denote by $H(z_1,...,z_k)$ the integrand in (\ref{mainEq2}). Then we have
\begin{equation}\label{truncate}
\lim_{M \rightarrow \infty} \left| \int_{\gamma_M} \cdots  \int_{\gamma_M} H(z_1,...,z_k) d{\bf z} - \int_{\gamma_M(0)} \cdots  \int_{\gamma_M(0)} H(z_1,...,z_k) d{\bf z} \right| = 0
\end{equation}
\end{proposition}
\begin{proof}
We adopt the same notation as in the proof of Proposition \ref{propEq5}. We write 
$$\int_{\gamma_M} \cdots  \int_{\gamma_M} H(z_1,...,z_k) d{\bf z} = \sum_{\epsilon_1,...,\epsilon_k \in \{0,1\}}  \int_{\gamma_M(\epsilon_1)} \cdots  \int_{\gamma_M(\epsilon_k)} H(z_1,...,z_k) d{\bf z},$$
and so we observe that the expression in the absolute value in (\ref{truncate}) is a finite sum of terms
$$ \int_{\gamma_M(\epsilon_1)} \cdots  \int_{\gamma_M(\epsilon_k)} H(z_1,...,z_k) d{\bf z},$$
where $\epsilon_i$ are not all equal to $1$. Recall from (\ref{mainEq2})
$$
 H(z_1,...,z_k)  = \prod_{1 \leq i < j \leq k}\frac{z_i - z_j}{z_i - qz_j}  \prod_{r = 1}^{k}  \left(  \frac{z_r- sq}{z_r - s}   \right)^{k-r+1}  \frac{e^{MG(z_r) + c\sqrt{M}x_rg(z_r) + h_rg(z_r)}}{z_r(1 - sz_{r})}$$
When $z_i \in \gamma_M$, we know that 
\begin{equation}\label{tr1}
\left|  \prod_{1 \leq i < j \leq k}\frac{z_i - z_j}{z_i - qz_j}  \prod_{r = 1}^{k}    \frac{\left( z_r- sq   \right)^{k-r+1}}{z_r(1 - sz_{r})} \right| \leq C \mbox{ and }  \left| \prod_{r = 1}^{k}  \left(  \frac{1}{z_r - s}   \right)^{k-r+1}  \right| \leq CM^{k(k+1)/4},
\end{equation}
for some constant $C > 0$, where we used that $z\in \gamma_M$ is at least a distance $M^{-1/2}$ from the point $s$, and is uniformly bounded away from other singularities.\\

Further, from our earlier analysis of the real part of $G(z)$ on $\gamma_M$, we know that when $z \in \gamma_M(1)$, we have for some (maybe different than before) constant $C> 0$
\begin{equation}\label{tr2}
\left| e^{MG(z_r) + c\sqrt{M}x_rg(z_r) + h_rg(z_r)}\right| \leq Ce^{-c'M}, \mbox{ where $c' = \min(c_0, c_\epsilon)$}.
\end{equation}
Finally, if $z \in \gamma_M(0)$, we know that 
\begin{equation}\label{tr3}
\begin{split}
&\left| e^{MG(z_r) + c\sqrt{M}x_rg(z_r) +h_rg(z_r)}\right| \leq Ce^{K\sqrt{M}}\hspace{3mm} \mbox{ if $M^{-1/2} \leq |Im(z)| \leq \epsilon$, and }\\
&\left| e^{MG(z_r) + c\sqrt{M}x_rg(z_r) +h_rg(z_r)}\right| \leq C \hspace{3mm} \mbox{ if $|Im(z)| \leq M^{-1/2}$}.
\end{split}
\end{equation}
In (\ref{tr3}), $K$ is a constant that dominates $|cx_rg(z)|$, for $z \in \gamma_M$ and $r = 1,..., k$. In obtaining the first estimate in (\ref{tr3}), we used that $Re[G(s + \iota y)] \leq 0$ for $|y| \in [M^{-1/2}, \epsilon]$, while for the second one we used (\ref{expest}). \\

If we combine the statements in (\ref{tr1}), (\ref{tr2}) and (\ref{tr3}) and use the compactness of $\gamma_M$, we see that 
$$\left| \int_{\gamma_M(\epsilon_1)} \cdots  \int_{\gamma_M(\epsilon_k)} H(z_1,...,z_k) d{\bf z} \right| \leq C e^{-c' (\epsilon_1 + \cdots + \epsilon_k)M } M^{k(k+1)/4} e^{k \cdot K \sqrt{M}},$$
and if $\epsilon_i$ are not all $0$, we see that the above decays to $0$ as $M \rightarrow \infty$.
\end{proof}

\subsection{Limit identification and proof of Theorem \ref{theorem1}}\label{Section4.3}

We start this section by showing that the RHS of (\ref{mainEq5}) equals $\mathbb{P}( \lambda^1_1 \leq x_1, \cdots , \lambda_k^k \leq x_k)$ when $x_1 \leq x_2 \leq \cdots \leq x_k$ and $x_i \in \mathbb{R}$. Here $\lambda_i^j$ $i = 1,...,j$, $j = 1,...,k$ is the GUE-corners process (see Section \ref{Section1.2}). The density of $\lambda_1^1,...,\lambda_k^k$ was calculated in \cite{War} to equal
$$\rho(x_1,...,x_k) = {\bf 1}_{\{x_1 \leq x_2 \leq \cdots \leq x_k\}} \det \left[ \Phi^{i- j}(x_j) \right]_{i,j = 1}^k.$$
In the above we have that $\Phi^n$ for $n \geq 1$ is the $n$-th order iterated integral of the Gaussian density $\phi(x) = \frac{e^{-x^2/2}}{\sqrt{2\pi}}$
\begin{equation}\label{phiPos}
\Phi^n(y) = \int_{- \infty}^y \frac{(y-x)^{n-1}}{(n-1)!} \phi(x) dx,
\end{equation}
and when $n \geq 0$, $\Phi^{-n}$ denotes the $n$-th order derivative of $\phi$. Let us denote
$$\Psi^m(y):= \frac{1}{2\pi \iota} \int_{1 + \iota \mathbb{R}}x^{m}e^{x^2/2 + yx} dx.$$
Then to show that the  RHS of (\ref{mainEq5}) equals $\mathbb{P}( \lambda^1_1 \leq x_1, \cdots , \lambda_k^k \leq x_k)$, it suffices to show that, when $x_1 \leq x_2 \leq \cdots \leq x_k$,
\begin{equation}\label{limitDist}
\int_{-\infty}^{x_1} dy_1 \int_{y_1}^{x_2} dy_2 \cdots \int_{y_{n-1}}^{x_n} dy_n \det \left[ \Phi^{i- j}(y_j) \right]_{i,j = 1}^k = \det\left[\Psi^{j-i-1}(x_j)  \right]_{i,j = 1}^k.
\end{equation}

The rapid decay of $e^{y^2/2}$ near $\pm \iota \infty$ shows that $\Psi^m(y)$ is differentiable, and its derivative equals
$$\frac{d}{dy} \Psi^m(y) = \frac{1}{2\pi \iota} \int_{1 + \iota \mathbb{R}}x^{m}\frac{d}{dy} e^{x^2/2 + yx} dx = \frac{1}{2\pi \iota} \int_{1 + \iota \mathbb{R}}x^{m+1} e^{x^2/2 + yx} dx = \Psi^{m+1}(y).$$
The other properties of $\Psi^m$ that we will need are that $\Psi^0(y) = \phi(y)$ and $\lim_{y \rightarrow -\infty} \Psi^m(y) = 0$. To see the former, we complete the square in the exponential of $\Psi^0(y)$ and change variables $x = 1 + \iota z$ to see
$$\Psi^0(y)= \frac{e^{-y^2/2}}{2\pi} \int_{ \mathbb{R}}\hspace{-1mm}e^{(1 + y + \iota z )^2/2} dz =  \frac{e^{-y^2/2}}{2\pi}\hspace{-1mm} \int_{ \mathbb{R}}\hspace{-1mm}e^{-(z - \iota (y+1) )^2/2} dz = \frac{e^{-y^2/2}}{2\pi}\hspace{-1mm} \int_{ \mathbb{R}}\hspace{-1mm}e^{-z^2/2} dz =  \frac{e^{-y^2/2}}{\sqrt{2\pi}} = \phi(y).$$
The middle equality follows from the usual shift of $\mathbb{R}$ to $\mathbb{R} + \iota (y +1)$, which does not change the integral by Cauchy's theorem. Performing the same change of variables we see that for any $m \in \mathbb{Z}$ and $y \leq -1$, we have that 
$$\Psi^m(y) = \frac{e^{-y^2/2}}{2\pi} \int_{ \mathbb{R}}(1 + \iota z)^m e^{-(z - \iota (y+1) )^2/2} dz = \frac{e^{-y^2/2}}{2\pi} \int_{ \mathbb{R}}(\iota z - y)^m e^{-z^2/2} dz, $$
where the last equality follows from the shift of $\mathbb{R}$ to $\mathbb{R} + \iota (y +1)$, which does not change the integral by Cauchy's theorem, as the possible pole at $z = \iota$ is never crossed when $y < 0$. When $m \leq 0$, we notice that $|(\iota z - y)^m| \leq 1$, when $y \leq -1$, while when $m \geq 0$, we can bound the same expression by $C(|y|^m + |z|^m + 1)$, uniformly in $z \in \mathbb{R}$ and $y \leq -1$. The upshot is that 
$$|\Psi^m(y)| \leq \frac{ e^{-y^2/2}}{2\pi}\int_{ \mathbb{R}}C(|y|^m + |z|^m + 1) e^{-z^2/2} dz \leq C(m)|y|^m e^{-y^2/2}, \mbox{ and hence} \lim_{y \rightarrow -\infty} \Psi^m(y) = 0.$$
Similar arguments also show that  $\lim_{y \rightarrow -\infty} \Phi^m(y) = 0$ for any $m \in \mathbb{Z}$. \\

We next show that $\Phi^{-m}(y) = \Psi^{m}(y)$ for all $m \in \mathbb{Z}$. From the previous paragraph we know this to be the case when $m = 0$. Since $\Psi^{m+1}(y) =\frac{d}{dy} \Psi^m(y) $ and $\Phi^{-m-1}(y) = \frac{d}{dy} \Phi^{-m}(y)$, when $m \geq 0$, we have equality when $m \geq 0$. Finally, we prove the result for $-m \geq 0$ by induction on $-m$. Suppose, we know that $\Phi^{k}(y) = \Psi^{-k}(y)$, for $k \geq 0$. Then we have
$$\frac{d}{dy}\Phi^{k+1}(y)  = \frac{d}{dy}\Psi^{-k-1}(y) \mbox{ and so } \Phi^{k+1}(y) - \Psi^{-k-1}(y) \mbox{ is constant.}$$
As both $ \Phi^{k+1}(y)$ and $\Psi^{-k-1}(y)$ vanish as $y \rightarrow -\infty$, we see that the constant is $0$, and we have $\Phi^{k+1}(y) = \Psi^{-k-1}(y)$. The general result now follows by induction.\\

We now turn to the proof of $(\ref{limitDist})$. From our discussion above we know that both sides define continuously differentiable functions in $x_1$. When $x_1$ goes to $-\infty$, we have that the first column in the matrix on the RHS goes to $0$ and so the determinant vanishes. The LHS also vanishes, as it is dominated by $\mathbb{P}( \lambda_1^1 \leq x_1)$. Consequently, it suffices to show that the derivatives w.r.t. $x_1$ on both sides agree. Replacing $\Phi^m$ with $\Psi^{-m}$, what we want is to show that when $x_1 \leq x_2 \leq \cdots \leq x_k$ and $y_1 = x_1$ 
$$\int_{x_1}^{x_2}dy_2 \int_{y_2}^{x_3} dy_3 \cdots \int_{y_{n-1}}^{x_n}dy_n \det \left[ \Psi^{j- i}(y_j) \right]_{i,j = 1}^k = \frac{d}{dx_1}\det\left[\Psi^{j-i-1}(x_j)  \right]_{i,j = 1}^k.$$
Using that $\frac{d}{dy} \Psi^m(y) = \Psi^{m+1}(y)$, we see that RHS above is the determinant of a matrix, whose first column is $\Psi^{0}(x_1),...,\Psi^{k-1}(x_1)$ and its $j$-th column for $2 \leq j \leq k$ is $\Psi^{j-2}(x_j),\Psi^{j-3}(x_j),...,\Psi^{j - k - 1}(x_j)$. In particular, when $x_2 = x_1$ the first two columns are the same and so the determinant vanishes. The LHS also vanishes because of the integral $\int_{x_1}^{x_2}dy_2$, and so to show equality it suffices to show equality of the derivatives w.r.t. $x_2$. I.e. we want when $x_1 \leq x_2 \leq \cdots \leq x_k$ and $y_1 = x_1$, $y_2 = x_2$ 
$$\int_{x_2}^{x_3} dy_3 \cdots \int_{y_{n-1}}^{x_n}dy_n \det \left[ \Psi^{j- i}(y_j) \right]_{i,j = 1}^k =\frac{d}{dx_2} \frac{d}{dx_1}\det\left[\Psi^{j-i-1}(x_j)  \right]_{i,j = 1}^k.$$
In this case, when $x_3 = x_2$, the RHS vanishes as the second and third column of the matrix become the same, while the LHS vanishes because of $\int_{x_2}^{x_3} dy_3$. Thus it is enough to show that the derivatives w.r.t. $x_3$ are equal. Continuing in this fashion for $x_3,...,x_k$, we see that (\ref{limitDist}) will follow if we know that
$$\det \left[ \Psi^{j- i}(x_j) \right]_{i,j = 1}^k =  \frac{d}{dx_k} \cdots \frac{d}{dx_2} \frac{d}{dx_1}\det\left[\Psi^{j-i-1}(x_j)  \right]_{i,j = 1}^k.$$
The above is now a trivial consequence of $\frac{d}{dy} \Psi^m(y) = \Psi^{m+1}(y)$ and so we conclude the validity of (\ref{limitDist}). \\

Our work above together with Proposition \ref{propEq5} show that when $x_1 \leq x_2 \leq \cdots \leq x_k$
$$\lim_{ M \rightarrow \infty} \mathbb{P}_{u,v}^{N,M}\left(  \frac{ Y_i^i - aM}{c\sqrt{M}}  \leq x_i; i = 1,...,k\right)  = \mathbb{P}( \lambda^1_1 \leq x_1, \cdots , \lambda_k^k \leq x_k).$$
Since with probability $1$, we have $Y_1^1 \leq Y_2^2 \leq \cdots \leq Y_k^k$ and $ \lambda^1_1 \leq \lambda_2^2 \leq \cdots \leq \lambda_k^k$, the above equality readily extends to all $x_1,...,x_k \in \mathbb{R}$. In particular, we obtain the following lemma.

\begin{lemma}\label{keyProbLemma}
Assume the same notation as in Theorem \ref{theorem2}. For any $k \geq 1$, we have that 
$$\frac{1}{c \sqrt{M}} (Y_1^1(M) - aM, \cdots, Y_k^k(M) - aM),$$
converge weakly to the vector $(\lambda_1^1,...,\lambda_k^k)$, where $\lambda_i^j$ for $i = 1,..., j$ and $j = 1,..., k$ is the GUE-corners process of rank $k$.
\end{lemma}

The above lemma will be one of the central ingredients necessary for the proof of Theorem \ref{theorem2} and we use it below to prove Theorem \ref{theorem1}

\begin{proof}(Theorem \ref{theorem1})
Assume the same notation as in Theorem \ref{theorem1}. It follows from our discussion in the proof of Proposition \ref{corePropProb} that 
$$\mathbb{P}_{u,v}^{N,M} \left( \{ \omega : \lambda^N_{N - i + 1}(\omega) = i, 1\leq i \leq k \} \right) = \mathbb{P}_{u,v}^{N,M} (Y^1_1 \leq N,..., Y_k^k \leq N) = \mathbb{P}_{u,v}^{N,M} ( Y_k^k \leq N) .$$
Let $x \in \mathbb{R}$ and notice that as $N \geq \gamma \cdot M$ with $\gamma > a$, we have that for all large $M$,
$$ \mathbb{P}_{u,v}^{N,M} (Y^1_1 \leq N,..., Y_k^k \leq N) \geq \mathbb{P}_{u,v}^{N,M}\left(  \frac{ Y_i^i - aM}{c\sqrt{M}}  \leq x; i = 1,...,k\right) .$$
By Lemma \ref{keyProbLemma}, the latter expression converges to $ \mathbb{P}\left(  \lambda_i^i \leq x; i = 1,...,k\right) = \mathbb{P}\left(  \lambda_k^k \leq x \right)$ as $M \rightarrow \infty$. Thus we have
$$ \liminf_{ M \rightarrow \infty} \mathbb{P}_{u,v}^{N,M} \left( \{ \omega : \lambda^N_{N - i + 1}(\omega) = i, 1\leq i \leq k \} \right)  \geq  \mathbb{P}\left(  \lambda_k^k \leq x \right).$$
The above being true for all $x \in \mathbb{R}$, we may send $x \rightarrow \infty$ and conclude the statement of the theorem.

\end{proof}

\section{Gibbs measures on Gelfand-Tsetlin cones} \label{Section5}
 In this section we investigate probability measures on Gelfand-Tsetlin cones in $\mathbb{R}^{n(n+1)/2}$, which satisfy what is known as the continuous Gibbs property (see Definition \ref{definitionGT} below). An example of such a measure is given by the GUE-corners process $\lambda_i^j$, $i = 1,...,j$, $j = 1,...,n$ of rank $n$. The main result of this section is Proposition \ref{pGibbs}, which can be understood as a classification result for the GUE-corners process. Essentially, it distinguishes the GUE-corners process as the unique probability measure on the Gelfand-Tsetlin cone $GT^n$ (defined in Section \ref{Section5.1} below), which satisfies the continuous Gibbs property and has a certain marginal distribution. A similar result, which we also use, is given by Proposition 6 in \cite{Gor14}.

It is well known that Gibbs measures on $\mathcal{C}_n$ are related to measures on $n\times n$ Hermitian matrices, that are invariant under the action of the unitary group $U(n)$ (see e.g.\cite{Def10}). The study of unitarily invariant measures on Hermitian matrices is a rich subject with connections to many branches of mathematics. A towering result in this area is the classification of the ergodic unitarily invariant Borel probability measures on infinite Hermitian matrices \cite{OV}, which can be viewed as the origin of our GUE-corners process classification result.

\subsection{The continuous Gibbs property}\label{Section5.1}
In what follows we adopt some of the terminology from \cite{Def10} and \cite{Gor14}. Let $\mathcal{C}_n$ be the {\em Weyl chamber} in $\mathbb{R}^n$ i.e.
$$\mathcal{C}_n := \{ (x_1,...,x_n) \in \mathbb{R}^n: x_1 \leq  x_2 \leq  \cdots \leq  x_n \}.$$
For $x \in \mathbb{R}^n$ and $y \in \mathbb{R}^{n-1}$ we write $x \succeq y$ to mean that 
$$x_1 \leq y_1 \leq x_2 \leq y_2 \leq \cdots \leq x_{n-1} \leq y_{n-1} \leq x_n.$$
For $x = (x_1,...,x_n) \in \mathcal{C}_n$ we define the {\em Gelfand-Tsetlin polytope} to be
$$GT_n(x): = \{ (x^{1},...,x^{n}): x^{n} = x, x^{k} \in \mathbb{R}^k, x^{k} \succeq x^{k-1}, 2 \leq k \leq n\}.$$
We explain what we mean by the uniform measure on a Gelfand-Tsetlin polytope $GT_n(x)$. The latter set is a bounded convex set $C$ of a real vector space. We define its volume, as we do for any bounded convex set, to be to be its measure according to the Lebesgue measure on the real affine subspace that it spans (if the subspace is of dimension $0$, i.e. $x_1 = \cdots = x_n$ the Lebesgue measure is given by the delta mass at $x_1 = \cdots = x_n$) and denote it by $vol(C)$. We define the Lebesgue measure on $C$ as this Lebesgue measure restricted to $C$ and the uniform probability measure on $C$ as the normalized Lebesgue measure on $C$ by $vol(C)$. The inclusion $x^{k} \in \mathbb{R}^k$ identifies $GT_n(x)$ as a subset of $\mathbb{R}^{n(n-1)/2}$ and we can naturally think of measures on $GT_n(x)$ as measures on $\mathbb{R}^{n(n-1)/2}$.\\ 

If $\lambda \in \mathcal{C}_n$ we denote by $\mu_\lambda$ the image of the uniform measure on $GT_n(\lambda)$ by the map $p_{n-1}: x \in GT_n(\lambda) \rightarrow x^{n-1} \in \mathcal{C}_{n-1}$. Let $l_\lambda$ be the Lebesgue measure on the convex set $p_{n-1}(GT_n(\lambda))$. Then Lemma 3.8 of \cite{Def10} shows that $\mu_\lambda$ is a probability measure on the set $\{ x^{n-1} \in \mathcal{C}_{n-1}: \lambda \succeq x^{n-1}\}$ and
$$\mu_\lambda(d\beta) = \frac{d_{n-1}(\beta)}{d_n(\lambda)} l_\lambda(d\beta),$$
where $d_k(\lambda)$ for $\lambda \in \mathcal{C}_k$ denotes $vol(GT_k(\lambda))$. Lemma 3.7 in \cite{Def10} shows that $d_n(\lambda)$ is explicitly given by
$$d_n(\lambda) =  \prod_{\substack{ 1 \leq i < j \leq n \\ \lambda_i \neq \lambda_j}}\frac{\lambda_j - \lambda_i}{j - i}.$$

For $\lambda \in \mathcal{C}_n$ we define $\mathbb{E}^{\mu_\lambda}$ to be the expectation with respect to $\mu_\lambda$ as defined above and we also set $\mathbb{E}^\lambda$ to be the expectation with respect to the uniform measure on $GT_n(\lambda)$ as defined above. We summarize some of the properties of these expectations in a sequence of lemmas, whose proof is deferred to Section \ref{Section5.2}.

\begin{lemma}\label{LemmaC1} Fix $n \geq 2$. Let $\lambda \in \mathcal{C}_n$ and $\lambda^k \in \mathcal{C}_n$ be such that $\lim_{k \rightarrow \infty} | \lambda - \lambda^k| = 0$. Suppose $f: \mathbb{R}^{n-1} \rightarrow \mathbb{C}$ is a bounded continuous function. Then we have
$$\lim_{k \rightarrow \infty}\mathbb{E}^{\mu_{\lambda^k}} \left[ f(x) \right] = \mathbb{E}^{\mu_{\lambda}} \left[ f(x) \right].$$
\end{lemma}

\begin{lemma}\label{LemmaC3} Let $n \geq 2$ and $f: \mathbb{R}^{n(n+1)/2} \rightarrow \mathbb{C}$ be bounded and continuous. Then the function 
$$g (y) := \mathbb{E}^{y} \left[ f(y, x^{n-1},...,x^{1}) \right], \mbox{ is bounded and continuous on $\mathcal{C}_n$.}$$
\end{lemma}

\begin{lemma}\label{LemmaC2} Let $n \geq 2$, $\lambda \in \mathcal{C}_n$ and $\lambda^k \in \mathcal{C}_n$ be such that $\lim_{k \rightarrow \infty} | \lambda - \lambda^k| = 0$. Suppose $f: \mathbb{R}^{n(n-1)/2} \rightarrow \mathbb{C}$ is a bounded continuous function. Then we have
$$\lim_{k \rightarrow \infty}\mathbb{E}^{\lambda^k} \left[ f(x^{n-1},...,x^{1}) \right] = \mathbb{E}^{\lambda} \left[ f(x^{n-1},...,x^{1}) \right].$$
\end{lemma}

We define the {\em Gelfand-Tsetlin cone} $GT^n$ to be
$$GT^n = \{ y \in \mathbb{R}^{n(n+1)/2}: y_{i }^{j+1} \leq y_{i}^j \leq y_{i+1}^{j+1}, \hspace{2mm} 1 \leq i \leq j \leq n-1\}.$$
Alternatively, we have $GT^n = \cup_{\lambda \in \mathcal{C}_n} GT_n(\lambda)$. We make the following definition after \cite{Gor14}.
\begin{definition}\label{definitionGT}
A probability measure $\mu$ on $GT^n$ is said to satisfy the {\em continuous Gibbs property} if conditioned on $y^n$ the distribution of $(y^1,...,y^{n-1})$ under $\mu$ is uniform on $GT_n(y^n)$. Equivalently, for any bounded continuous function $f : \mathbb{R}^{n(n+1)/2} \rightarrow \mathbb{C}$ we have that 
$$\mathbb{E}^\mu \left[ f(y^n,...,y^1) \right] = \mathbb{E}^{\mu^n}\left[ \mathbb{E}^{y^n}\left[f(y^n, y^{n-1},...,y^1\right]\right],$$
where $\mu^n$ is the pushforward of $\mu$ to the top row $y^n$ of the Gelfand-Tsetlin cone $GT^n$.
\end{definition}
\begin{remark}
It follows from Lemma \ref{LemmaC3} that $ \mathbb{E}^{y^n}\left[f(y^n, y^{n-1},...,y^1\right]$ is a continuous function of $y^n$ and so its expectation with respect to $\mu^n$ is a well-defined quantity.
\end{remark}

The main result of this section is as follows.

\begin{proposition}\label{pGibbs} Suppose that $\mu$ is a probability distributions on $GT^n$, which satisfies the continuous Gibbs property (Definition \ref{definitionGT}). Suppose that the joint distribution of $(y_1^1,...,y_n^n)$ under $\mu$ agrees with the law of $(\lambda_1^1,...,\lambda_n^n)$, where  $\lambda_i^j$, $i = 1,...,j$, $j = 1,...,n$ is the GUE-corners process of rank $n$. Then $\mu$ is the GUE-corners process of rank $n$.
\end{proposition}
The above proposition relies on the following lemmas, whose proof is deferred to Section \ref{Section5.3}.

\begin{lemma}\label{lSym1} For $x^i \in \mathcal{C}_i$, $i = 1,..,n$ and $t = (t_1,...,t_n)$ with $t_i \in \mathbb{R}$ define
$$f_n(t,x^n, x^{n-1},...,x^{1}) :=  \prod_{i = 1}^n \exp( \iota t_i (|x^{i}| - |x^{i-1}|)),$$
where $|x^k| = x^k_1 + \cdots + x^k_k$ and $|x^0| = 0$. 
Suppose $n \geq 2$ and $x^{n} \in \mathcal{C}_n$ with $x^n_n > x^n_{n-1} \cdots > x^n_1$ and $t = (t_1,...,t_n)$ with $t_i $ pairwise distinct. Then
$$d_n(x^n) \cdot \mathbb{E}^{x^n}\left[ f_n(t, x^n,...,x^1)\right] = \prod_{1\leq i < j \leq n}\frac{1}{\iota( t_j - t_i)} \times \sum_{\sigma \in S_n} sign(\sigma)\exp\left( \iota \sum_{i = 1}^n t_{\sigma(i)}x^n_{i}\right).$$
\end{lemma}

\begin{lemma}\label{lSym2}
Suppose $n \geq 2$ and $x^{n} \in \mathcal{C}_n$ with $x^n_n > x^n_{n-1} \cdots > x^n_1$. Let $t = (t_1,...,t_n)$ with $t_i \in \mathbb{R}$. For $\sigma \in S_n$ we define $t_\sigma := (t_{\sigma(1)},...,t_{\sigma(n)})$ and we set
\hspace{-5mm}
$$g_n(t,x^n, x^{n-1},...,x^{1}) :=  \prod_{i = 1}^n \exp( \iota t_ix^i_i).$$
If $t_i$ are all nonzero we have
$$d_n(x^n) \cdot \sum_{\sigma \in S_n}sign(\sigma)  \mathbb{E}^{x^n}\left[ g_n(t_\sigma, x^n,...,x^1)\right] \prod_{i = 1}^n ( \iota t_{\sigma(i)})^{n-i}  =  (-1)^{ \frac{n(n-1)}{2}}\sum_{\sigma \in S_n} sign(\sigma)\exp\left( \iota \sum_{i = 1}^n t_{\sigma(i)}x^n_{i}\right).$$
\end{lemma}

\begin{proof} (Proposition \ref{pGibbs})
Suppose $t = (t_1,...,t_n)$ with $t_i \in \mathbb{R}$ is such that $t_i$ are pairwise distinct and non-zero. It follows from Lemmas \ref{lSym1} and \ref{lSym2} that if $x^n_n > x^n_{n-1} > \cdots > x^n_1$, we have
$$\mathbb{E}^{x^n}\left[ f_n(t, x^n,...,x^1)\right] = \prod_{1\leq i < j \leq n}\frac{(-1)}{\iota( t_j - t_i)}   \sum_{\sigma \in S_n}sign(\sigma)  \mathbb{E}^{x^n}\left[ g_n(t_\sigma, x^n,...,x^1)\right] \prod_{i = 1}^n ( \iota t_{\sigma(i)})^{n-i}  .$$
From Lemma \ref{LemmaC2} we know that both sides of the above equality are continuous in $x^n$ and so the equality holds for all $x^n \in \mathcal{C}_n$.\\

Taking the expectation with respect to $\mu$ on both sides we recognize the LHS as the characteristic function of $(|x^n| - |x^{n-1}|,...,|x^2| - |x^1|, |x^1|)$ under the law $\mu$. The RHS is a linear combination of the characteristic functions of $(x_1^1,...,x_n^n)$ under the law $\mu$. By assumption, $(x_1^1,...,x_n^n)$ has the same law under $\mu$ as $(\lambda_1^1,...,\lambda_n^n)$, from which we conclude that 
$$\mathbb{E}^{\mu}\left[ \exp \left(\sum_{i = 1}^n \iota t_i (|x^i| - |x^{i-1}| \right)\right]  = \mathbb{E}\left[ \exp \left( \sum_{i = 1}^n \iota t_i (|\lambda^i| - |\lambda^{i-1}| \right)\right] ,$$
whenever $t_i$ are pairwise distinct and non-zero (recall $|x^0| = 0$). Since the characteristic functions are continuous in $t$ it follows that the above equality holds for all $t \in \mathbb{R}^n$. As the characteristic function of a distribution uniquely defines it we conclude that $(|x^n| - |x^{n-1}|,...,|x^2| - |x^1|, |x^1|)$ are i.i.d. Gaussian random variables with mean $0$ and variance $1$. The latter together with the continuous Gibbs property, satisfied by $\mu$, implies that $\mu$ is the GUE corners process by Proposition 6 in \cite{Gor14}.
\end{proof}

\subsection{Proof of Lemmas \ref{LemmaC1}, \ref{LemmaC3} and \ref{LemmaC2}}\label{Section5.2}
We adopt the same notation as in Section \ref{Section5.1}.

\begin{proof}(Lemma \ref{LemmaC1})
We begin by first assuming that $f(x) = \prod_{i = 1}^{n-1}f_i(x_i)$ where $f_i$ are bounded, continuous and non-negative real-valued functions. Let $1\leq n_1 < n_2 <\cdots <n_r <  n-1$ and $m_1,...,m_r > 1$ be such that
\begin{itemize}
\item $\lambda_{i} = \lambda_{j} \mbox{ if } i,j \in M_q \mbox{ for some } q = 1,...,r; $
\item $\lambda_{i} < \lambda_{j} \mbox{ if } i < j \mbox{ and } \{i,j\} \not \subset M_q \mbox{ for any } q = 1,...,r,$
\end{itemize}
where $M_q = \{ n_q, ..., n_q + m_q -1\}$. We also set $J : = \{ j : 1\leq j \leq n-1, \mbox{ and } \{j , j+1\} \not \subset M_q \mbox{ for any } q = 1,...,r\}$ and $M'_q =  \{ n_q, ..., n_q + m_q -2\}$. Then by the definition of $\mu_\lambda$ we have that 
$$\mathbb{E}^{\mu_{\lambda}} \left[  \prod_{i = 1}^{n-1}f_i(x_i) \right] =  \prod_{\substack{ 1 \leq i < j \leq n \\ \lambda_i \neq \lambda_j}}\frac{j - i}{\lambda_j - \lambda_i} \times \prod_{q = 1}^r  \prod_{ j \in M_q'}f_{j}(\lambda_{n_q}) \times  \prod_{1 \leq i < j \leq r}\prod_{s \in M_i'}\prod_{t \in M_j'} \frac{\lambda_{t} - \lambda_{s}}{t - s} \times F(\lambda),$$
$$ \mbox{ where } F(\lambda) =   \left( \prod_{j \in J} \int_{ \lambda_{j}}^{\lambda_{j+1}} \hspace{-3mm} dx_j\right) \prod_{j \in J}f_j(x_j)\prod_{\substack{ 1 \leq i < j \leq n-1 \\ i,j \in J}}\frac{x_j - x_i}{j - i} \prod_{\substack{ 1 \leq i < j \leq n-1 \\ i \in J, j\not \in J}}\frac{ x_j - \lambda_i}{j - i}\prod_{\substack{ 1 \leq i < j \leq n-1 \\ i \not \in J, j \in J}}\frac{\lambda_j - x_i}{j - i}. $$

Let us assume that for each $k$ we have $\lambda_1^k < \lambda_{2}^k < \cdots < \lambda_n^k$. Then the above formula yields
$$\mathbb{E}^{\mu_{\lambda^k}} \left[ \prod_{i = 1}^{n-1}f_i(x_i) \right] = \prod_{1 \leq i < j \leq n }\frac{j - i}{\lambda^k_j - \lambda^k_i} \left( \prod_{j = 1}^{n-1} \int_{ \lambda^k_{j}}^{\lambda^k_{j+1}}dx_j\right) \prod_{j = 1}^{n-1}f_j(x_j) \prod_{ 1 \leq i < j \leq n -1}\frac{x_j - x_i}{j - i}.$$

Suppose $\epsilon > 0$ is given. Then if $k$ is sufficiently large we know by the continuity of the functions that for all $j \in J$ we have $|f_j(x_j) - f(\lambda_j)| < \epsilon$ for all $x_j \in [\lambda^k_{j}, \lambda^k_{j+1}]$. Using that $f_i$ are uniformly bounded by some $M$ we conclude that 
\begin{equation}\label{nr1}
\left| \mathbb{E}^{\mu_{\lambda^k}} \left[ \prod_{i = 1}^{n-1}f_i(x_i) \right] - \hspace{-2mm}\prod_{1 \leq i < j \leq n } \hspace{-1mm}\frac{j - i}{\lambda^k_j - \lambda^k_i} \hspace{-1mm}\left( \prod_{j = 1}^{n-1} \int_{ \lambda^k_{j+1}}^{\lambda^k_{j}}dx_j \hspace{-1mm}\right) \hspace{-2mm}\prod_{j \not \in J}f_j(\lambda_j) \prod_{j \in J}f_j(x_j) \hspace{-5mm}\prod_{ 1 \leq i < j \leq n -1}\frac{x_j- x_i}{j - i}\right| < C\epsilon,
\end{equation}
for all sufficiently large $k$, where $C$ can be taken to be $(n-1)(1+M)^{n-1}$. Observing that $\prod_{j \not \in J}f_j(\lambda_j) = \prod_{q = 1}^r \left( \prod_{ j = 0}^{m_q-2} f_{j + n_q}(\lambda_{n_q}) \right)$ and using (\ref{nr1}) we get
\begin{equation}\label{nr2}
\limsup_{k \rightarrow \infty} \left| \mathbb{E}^{\mu_{\lambda}} \left[  \prod_{i = 1}^{n-1}f_i(x_i) \right]  - \mathbb{E}^{\mu_{\lambda^k}} \left[ \prod_{i = 1}^{n-1}f_i(x_i) \right]\right| \leq C\epsilon + (M+1)^n \limsup_{k \rightarrow \infty} |G_1(\lambda) - G_2(\lambda^k)|, 
\end{equation} 
$$\mbox{ where }G_1(\lambda) = \bigl( \prod_{\substack{ i, j = 1,...,r \\ i < j}}\prod_{s \in M_i'}\prod_{t \in M_j'} \frac{\lambda_{t} - \lambda_{s}}{t - s}\bigr) \prod_{\substack{ 1 \leq i < j \leq n \\ \lambda_i \neq \lambda_j}}\frac{j - i}{\lambda_j - \lambda_i}F(\lambda), \mbox{ and }$$
$$G_2(\lambda^k) =  \prod_{1 \leq i < j \leq n }\frac{j - i}{\lambda^k_j - \lambda^k_i} \left( \prod_{j = 1}^{n-1} \int_{ \lambda^k_{j}}^{\lambda^k_{j+1}}dx_j\right)  \prod_{j \in J}f_j(x_j) \prod_{ 1 \leq i < j \leq n -1}\frac{x_j - x_i}{j - i}.$$

For $j \not \in J$ denote by $\phi(j)$ the $q$ such that $\{j, j+1\} \subset M_q$. We define $G_2^+$ and $G_2^-$ as follows
$$G^+_2(\lambda^k) =  \prod_{1 \leq i < j \leq n }\frac{j - i}{\lambda^k_j - \lambda^k_i} \left( \prod_{j = 1}^{n-1} \int_{ \lambda^k_{j}}^{\lambda^k_{j+1}}dx_j\right)  \prod_{j \in J}f_j(x_j) \prod_{\substack{ 1 \leq i < j \leq n-1 \\ i,j \in J}}\frac{x_j - x_i}{j - i}$$
$$\times \prod_{\substack{j \in J, i\not \in J\\  i < j }}\frac{x_j - \lambda^k_{i}}{j - i}  \prod_{\substack{i \in J, j\not \in J\\  i < j }}\frac{\lambda_{j+1}^k - x_i}{j - i}  \prod_{q = 1}^r\prod_{\substack{i < j, \\ i,j \in M_q'} }\frac{x_j - x_i}{j - i} \prod_{\substack{i,j \not \in J, i < j, \\j \not \in M'_{\phi(i)}}}\frac{\lambda^k_{j+1} -\lambda^k_{i}}{j - i};$$
$$G^-_2(\lambda^k) =  \prod_{1 \leq i < j \leq n }\frac{j - i}{\lambda^k_j - \lambda^k_i} \left( \prod_{j = 1}^{n-1} \int_{ \lambda^k_{j}}^{\lambda^k_{j+1}}dx_j\right)  \prod_{j \in J}f_j(x_j) \prod_{\substack{ 1 \leq i < j \leq n-1 \\ i,j \in J}}\frac{x_j - x_i}{j - i}$$
$$\times \prod_{\substack{j \in J, i\not \in J\\  i < j }}\frac{x_j - \lambda^k_{i+1}}{j - i}  \prod_{\substack{i \in J, j\not \in J\\  i < j }}\frac{\lambda_{j}^k - x_i}{j - i}   \prod_{q = 1}^r\prod_{\substack{i < j, \\ i,j \in M_q'} }\frac{x_j - x_i}{j - i} \prod_{\substack{i,j \not \in J, i < j, \\j \not \in M'_{\phi(i)}}}\frac{\lambda^k_{j} -\lambda^k_{i+1}}{j - i};$$
Using the non-negativity of $f_i$ we observe that $G^-_2(\lambda^k) \leq G_2(\lambda^k) \leq G^+_2(\lambda^k).$ 

Performing the integration over $x_j$ for $j \not\in J$ we may rewrite $G^+_2(\lambda^k)$ as
$$G^+_2(\lambda^k) = \prod_{\substack{ 1 \leq i < j \leq n \\ \lambda_i \neq \lambda_j}}\frac{j - i}{\lambda^k_j - \lambda^k_i} \prod_{\substack{i,j \not \in J, i < j, \\j \not \in M'_{\phi(i)}}}\frac{\lambda^k_{j+1} -\lambda^k_{i}}{j - i} \left( \prod_{j \in J} \int_{ \lambda^k_{j}}^{\lambda^k_{j+1}}dx_j\right) \prod_{j \in J}f_j(x_j) \times $$ 
$$ \prod_{\substack{ 1 \leq i < j \leq n-1 \\ i,j \in J}}\frac{x_j - x_i}{j - i} \prod_{\substack{j \in J, i\not \in J\\  i < j }}\frac{x_j - \lambda^k_i }{j - i}  \prod_{\substack{i \in J, j\not \in J\\  i < j }}\frac{ \lambda^k_{j+1} - x_i }{j - i}.$$
Similarly, we have
$$G^-_2(\lambda^k) = \prod_{\substack{ 1 \leq i < j \leq n \\ \lambda_i \neq \lambda_j}}\frac{j - i}{\lambda^k_j - \lambda^k_i} \prod_{\substack{i,j \not \in J, i < j, \\j \not \in M'_{\phi(i)}}}\frac{ \lambda^k_{j} - \lambda^k_{i+1}}{j - i} \left( \prod_{j \in J} \int_{ \lambda^k_{j}}^{\lambda^k_{j+1}}dx_j\right) \prod_{j \in J}f_j(x_j) \times $$ 
$$ \prod_{\substack{ 1 \leq i < j \leq n-1 \\ i,j \in J}}\frac{x_j - x_i}{j - i} \prod_{\substack{j \in J, i\not \in J\\  i < j }}\frac{ x_j - \lambda^k_{i+1}}{j - i}  \prod_{\substack{i \in J, j\not \in J\\  i < j }}\frac{\lambda^k_{j} - x_i }{j - i}.$$

We observe that 
\begin{equation}\label{nr3}
\begin{split}
\lim_{k \rightarrow \infty}\prod_{\substack{ 1 \leq i < j \leq n \\ \lambda_i \neq \lambda_j}}\frac{j - i}{\lambda^k_j - \lambda^k_i} &\prod_{\substack{i,j \not \in J, i < j, \\j \not \in M'_{\phi(i)}}}\frac{\lambda^k_{j+1} -\lambda^k_{i}}{j - i} =
 \lim_{k \rightarrow \infty} \prod_{\substack{ 1 \leq i < j \leq n \\ \lambda_i \neq \lambda_j}}\frac{j - i}{\lambda^k_j- \lambda^k_i} \prod_{\substack{i,j \not \in J, i < j, \\j \not \in M'_{\phi(i)}}}\frac{\lambda^k_{j} -\lambda^k_{i+1}}{j - i} = \\
&\prod_{\substack{ 1 \leq i < j \leq n \\ \lambda_i \neq \lambda_j}}\frac{j - i}{\lambda_j - \lambda_i} \times \prod_{\substack{ i, j = 1,...,r \\ i < j}}\prod_{s \in M_i'}\prod_{t \in M_j'} \frac{\lambda_{t} - \lambda_{s}}{t - s}  .
\end{split}
\end{equation}
Moreover, by the Bounded Convergence Theorem we conclude that 
\begin{equation}\label{nr4}
\begin{split}
\lim_{k \rightarrow \infty} \left( \prod_{j \in J} \int_{ \lambda^k_{j}}^{\lambda^k_{j+1}}dx_j\right) \prod_{j \in J}f_j(x_j)
\prod_{\substack{ 1 \leq i < j \leq n-1 \\ i,j \in J}}\frac{x_j - x_i}{j - i} \prod_{\substack{j \in J, i\not \in J\\  i < j }}\frac{x_j - \lambda^k_i }{j - i}  \prod_{\substack{i \in J, j\not \in J\\  i < j }}\frac{ \lambda^k_{j+1} - x_i }{j - i} = F(\lambda),\\
 \lim_{k \rightarrow \infty} \left( \prod_{j \in J} \int_{ \lambda^k_{j}}^{\lambda^k_{j+1}}dx_j\right) \prod_{j \in J}f_j(x_j) \prod_{\substack{ 1 \leq i < j \leq n-1 \\ i,j \in J}}\frac{x_j - x_i}{j - i} \prod_{\substack{j \in J, i\not \in J\\  i < j }}\frac{ x_j - \lambda^k_{i+1} }{j - i}  \prod_{\substack{i \in J, j\not \in J\\  i < j }}\frac{\lambda^k_{j} - x_i }{j - i} = F(\lambda).
\end{split}
\end{equation}
From (\ref{nr3}) and (\ref{nr4}) we conclude that $\lim_{k \rightarrow \infty}G^{\pm}_2(\lambda^k) = G_1(\lambda)$ and since $G^-_2(\lambda^k) \leq G_2(\lambda^k) \leq G^+_2(\lambda^k)$ we conclude that $\lim_{k \rightarrow \infty} G_2(\lambda^k) = G_1(\lambda)$. The latter implies from (\ref{nr2}) that 
$$\limsup_{k \rightarrow \infty} \left| \mathbb{E}^{\mu_{\lambda}} \left[  \prod_{i = 1}^{n-1}f_i(x_i) \right]  - \mathbb{E}^{\mu_{\lambda^k}} \left[ \prod_{i = 1}^{n-1}f_i(x_i) \right]\right| \leq C\epsilon.$$
Since $\epsilon > 0$ was arbitrary we conclude that 
$$\limsup_{k \rightarrow \infty} \left| \mathbb{E}^{\mu_{\lambda}} \left[  \prod_{i = 1}^{n-1}f_i(x_i) \right]  - \mathbb{E}^{\mu_{\lambda^k}} \left[ \prod_{i = 1}^{n-1}f_i(x_i) \right]\right| = 0.$$

We next suppose that $\lambda^k$ do not necessarily satisfy $\lambda_1^k < \lambda_{2}^k < \cdots < \lambda_n^k$. If we are given a $\lambda^k$, then from our earlier work we may find $\nu^k$ such that
\begin{enumerate}
\item $|\nu^k - \lambda^k| < 1/k$, 
\item $\nu_1^k < \nu_{2}^k < \cdots < \nu_n^k$,
\item $\left| \mathbb{E}^{\mu_{\lambda^k}} \left[  \prod_{i = 1}^{n-1}f_i(x_i) \right]  - \mathbb{E}^{\mu_{\nu^k}} \left[ \prod_{i = 1}^{n-1}f_i(x_i) \right]\right| < 1/k$.
\end{enumerate}
Condition (1) above implies that $\nu^k$ converges to $\lambda$ and by (2) our earlier work applies so we get
$$\limsup_{k \rightarrow \infty} \left| \mathbb{E}^{\mu_{\lambda}} \left[  \prod_{i = 1}^{n-1}f_i(x_i) \right]  - \mathbb{E}^{\mu_{\nu^k}} \left[ \prod_{i = 1}^{n-1}f_i(x_i) \right]\right| = 0.$$
Finally, by the triangle inequality and condition (3) we conclude that 
$$\limsup_{k \rightarrow \infty} \left| \mathbb{E}^{\mu_{\lambda}} \left[  \prod_{i = 1}^{n-1}f_i(x_i) \right]  - \mathbb{E}^{\mu_{\lambda^k}} \left[ \prod_{i = 1}^{n-1}f_i(x_i) \right]\right| = 0.$$
This proves the statement of the lemma, whenever $f(x) = \prod_{i = 1}^{n-1}f_i(x_i)$ with $f_i$ bounded, continuous and non-negative real-valued functions.\\

Using linearity of expectation and our earlier result we concude the statement of the lemma, whenever $f(x) $ is a finite linear combination of functions of the form $\prod_{i = 1}^{n-1}f_i(x_i)$ with $f_i$ bounded and continuous. In particular, we know the result whenever $f$ equals $P(x) \cdot  {\bf 1}_{B_R}$, where $R > 0$, $B_R = \{ x \in \mathbb{R}^{n-1} | |x_i| \leq R \mbox{ for } i = 1,...,n-1\}$ and $P(x)$ is a polynomial.

If $f(x)$ is any bounded continuous function, we may replace it with $f(x) {\bf 1}_{B_R}$, where $R = 1 + \max( |\lambda_1|, |\lambda_n|)$, without affecting the statement of the lemma, since for large $k$, the support of $\mu_{\lambda^k}$ lies in $ B_R$. By the Stone-Weierstrass Theorem we can find a polynomial $g(x)$ such that $\sup_{x \in \mathbb{R}^{n-1}} |f(x){\bf 1}_{B_R} - g(x) {\bf 1}_{B_R}|< \epsilon$. The triangle inequality and our result for polynomials now show  
$$\limsup_{k \rightarrow \infty} \left| \mathbb{E}^{\mu_{\lambda}} \left[  f(x) \right]  - \mathbb{E}^{\mu_{\lambda^k}} \left[f(x) \right]\right|  = \limsup_{k \rightarrow \infty} \left| \mathbb{E}^{\mu_{\lambda}} \left[  f(x){\bf 1}_{B_R} \right]  - \mathbb{E}^{\mu_{\lambda^k}} \left[f(x){\bf 1}_{B_R} \right]\right| \leq $$
$$ \limsup_{k \rightarrow \infty} ( \left| \mathbb{E}^{\mu_{\lambda}} \left[  f(x){\bf 1}_{B_R} \right]  - \mathbb{E}^{\mu_{\lambda}} \left[g(x) {\bf 1}_{B_R}\right]\right|  +\left| \mathbb{E}^{\mu_{\lambda^k}} \left[  f(x){\bf 1}_{B_R}\right]  - \mathbb{E}^{\mu_{\lambda^k}} \left[g(x){\bf 1}_{B_R} \right]\right| +$$
$$\left| \mathbb{E}^{\mu_{\lambda}} \left[  g(x) {\bf 1}_{B_R}\right]  - \mathbb{E}^{\mu_{\lambda^k}} \left[g(x){\bf 1}_{B_R} \right]\right| ) \leq 2\epsilon.$$
Since $\epsilon > 0$ was arbitrary we conclude that $\limsup_{k \rightarrow \infty} \left| \mathbb{E}^{\mu_{\lambda}} \left[  f(x) \right]  - \mathbb{E}^{\mu_{\lambda^k}} \left[f(x) \right]\right| = 0$.

\end{proof}

\begin{proof}(Lemma \ref{LemmaC3})
We begin by assuming that $ f(x^{n},...,x^{1}) = f_1(x^{n})f_2(x^{n-1},...,x^{1})$ with $f_1, f_2$ bounded and continuous. Fix $\beta \in \mathcal{C}_n$ and suppose $\mathcal{C}_n \ni \beta^k \rightarrow \beta$  as $k \rightarrow \infty$. From Lemma \ref{LemmaC1} we have 
$$ \lim_{k \rightarrow \infty}\mathbb{E}^{\beta^k}[ f(\beta^k, x^{n-1},...,x^{1})] =  \lim_{k \rightarrow \infty}f_1(\beta^k) \mathbb{E}^{\beta^k}[ f_2(x^{n-1},...,x^{1})] = f_1(\beta) \mathbb{E}^{\beta}[ f_2(x^{n-1},...,x^{1})] .$$
Using linearity of expectation and the above we have that $\mathbb{E}^{\beta}[ f(\beta, x^{n-1},...,x^{1})]$ is a continuous function in $\beta$, whenever $f$ is of the form $P(x^n, x^{n-1},...,x^{1}) \cdot  {\bf 1}_{B_R}$, where $R > 0$, $B_R = \{ x \in \mathbb{R}^{n(n+1)/2} | |x^j_i| \leq R \mbox{ for } i = 1,...,j; j = 1,...n\}$ and $P(x)$ is a polynomial.\\

Suppose now $f$ is any bounded continuous function, fix $\beta \in \mathcal{C}_n$ and suppose $\mathcal{C}_n \ni \beta^k \rightarrow \beta$  as $k \rightarrow \infty$. For all large $k$ we have that $\beta^k$ lie in the compact set $B_R$, with $R = 1 + \max( |\beta_1|, |\beta_{n}|)$. By the Stone-Weierstrass Theorem we can find a polynomial $g(x)$ such that $\sup_{x \in \mathbb{R}^{n(n+1)/2}} |f(x){\bf 1}_{B_R} - g(x) {\bf 1}_{B_R}|< \epsilon$. The triangle inequality and our result for polynomials now show  
$$ \limsup_{k \rightarrow \infty}\left| \mathbb{E}^{\beta^k}[ f(\beta^k, x^{n-1},...,x^{1})] - \mathbb{E}^{\beta}[ f(\beta, x^{n-1},...,x^{1})]\right| =   $$
$$\limsup_{k \rightarrow \infty}\left| \mathbb{E}^{\beta^k}[ f(\beta^k, x^{n-1},...,x^{1}){\bf 1}_{B_R}] - \mathbb{E}^{\beta}[ f(\beta, x^{n-1},...,x^{1}){\bf 1}_{B_R}]\right| \leq $$
$$ \limsup_{k \rightarrow \infty}\left| \mathbb{E}^{\beta^k}[ g(\beta^k, x^{n-1},...,x^{1}){\bf 1}_{B_R}] - \mathbb{E}^{\beta}[g(\beta, x^{n-1},...,x^{1}){\bf 1}_{B_R}]\right| + 2\epsilon = 2\epsilon.$$
As $\epsilon > 0$ was arbitrary we conclude continuity, while boundedness is immediate from the boundedness of $f$.\\
\end{proof}

\begin{proof}(Lemma \ref{LemmaC2})
We proceed by induction on $n$ with $n = 2$ being true by Lemma \ref{LemmaC1}. Suppose the result holds for $n-1 \geq 2$ and we want to prove it for $n$.

For any $\nu \in \mathcal{C}_n$ we have
$$\mathbb{E}^{\nu} \left[ f(x^{n-1},...,x^{1}) \right] = \int_{\mathcal{C}_{n-1}} \mu_\nu(d\beta) \mathbb{E}^{\beta}[ f(\beta, x^{n-2},...,x^{1})] = \mathbb{E}^{\mu_\nu} \left[ \mathbb{E}^{\beta}[ f(\beta, x^{n-2},...,x^{1})]\right]  .$$
By Lemma \ref{LemmaC3}, we have $\mathbb{E}^{\beta}[ f(\beta, x^{n-2},...,x^{1})]$ is a bounded and continuous function in $\beta \in \mathcal{C}_n$. From Lemma \ref{LemmaC1} we conclude that 
$$\lim_{k \rightarrow \infty}\mathbb{E}^{\mu_{\lambda^k}} \left[ \mathbb{E}^{\beta}[ f(\beta, x^{n-2},...,x^{1})]\right] = \mathbb{E}^{\mu_{\lambda}} \left[ \mathbb{E}^{\beta}[ f(\beta, x^{n-2},...,x^{1})]\right].$$
This proves the result for $n$ and the general result follows by induction.
\end{proof}

\subsection{Proof of Lemmas \ref{lSym1} and \ref{lSym2}}\label{Section5.3}
We adopt the same notation as in Section \ref{Section5.1}.

\begin{proof} (Lemma \ref{lSym1})
We proceed by induction on $n$. When $n = 2$ we have that 
$$d_2(x^n) \cdot \mathbb{E}^{x^2}\left[ f_2(t, x^2,x^1)\right] = e^{\iota t_2(x^2_1 + x^2_2)} \int_{x_1^2}^{x^2_2} e^{\iota(t_1 - t_2)x}dx =  e^{\iota t_2(x^2_1 + x^2_2)}\frac{ e^{\iota (t_1 - t_2)x_2^2} -  e^{\iota (t_1 - t_2)x_1^2}}{\iota (t_1 - t_2)} = $$
$$\frac{1}{\iota (t_1 - t_2)} \times \left[ \exp(\iota  (t_2 x_1^2 + t_1 x_2^2)) - \exp(\iota  (t_1 x_1^2 + t_2 x_2^2))\right],$$
which proves the base case.\\

Suppose the result holds for $n-1 \geq 2$ and we wish to prove it for $n$. We have
$$d_n(x) \cdot \mathbb{E}^{x^n}\left[ f(t, x^n,...,x^1)\right] = \int_{x^n_{n-1}}^{x^n_{n}} \cdots \int_{x^{n}_{1}}^{x^n_2}e^{it_n|x^n|}dy_{n-1}\cdots dy_1 d_{n-1}(y) \cdot \mathbb{E}^{y}\left[ f_{n-1}(s, y,x^{n-2},...,x^1)\right],$$
where $s = (t_1 - t_n,...,t_{n-2} - t_n, t_{n-1} - t_{n})$. By induction hypothesis the above becomes
$$e^{\iota t_n|x^n|}\int_{x^n_{n-1}}^{x^n_{n}} \cdots \int_{x^{n}_{1}}^{x^n_2}dy_{n-1}\cdots dy_1\prod_{1\leq i < j \leq {n-1}}\frac{1}{\iota( t_j - t_i)} \times \sum_{\sigma \in S_{n-1}} sign(\sigma)\exp\left( \iota \sum_{i = 1}^{n-1} s_iy_{\sigma(i)}\right) = $$
$$\prod_{1\leq i < j \leq {n-1}}\frac{1}{\iota( t_j - t_i)}  \sum_{\sigma \in S_{n-1}} sign(\sigma) e^{it_n|x^n|} \prod_{i = 1}^{n-1} \frac{\exp\left( \iota  s_ix^n_{\sigma(i)}\right) - \exp\left( \iota s_ix^n_{\sigma(i) + 1}\right) }{\iota (t_{n} - t_i)}  = $$
$$e^{\iota t_n|x^n|} \prod_{1\leq i < j \leq {n}}\frac{1}{\iota( t_j- t_i)}  \sum_{\sigma \in S_{n-1}} sign(\sigma) \prod_{i = 1}^{n-1}\left( \exp\left( \iota  s_{\sigma(i)}x^n_{i}\right) - \exp\left( \iota s_{\sigma(i)}x^n_{i + 1}\right) \right),$$
where in the last equality we used that $sign (\sigma) = sign(\sigma^{-1})$. 

The above equality reduces the induction step to showing 
\begin{equation}\label{chsym1}
 \sum_{\sigma \in S_{n-1}} sign(\sigma) \prod_{i = 1}^{n-1}\left( \exp\left( \iota  s_{\sigma(i)}x^n_{i}\right) - \exp\left( \iota s_{\sigma(i)}x^n_{i + 1}\right) \right) = \sum_{\sigma \in S_n} sign(\sigma)\exp\left( \iota \sum_{i = 1}^{n} s_{\sigma(i)}x^n_{i}  \right),
\end{equation}
where $s_n = 0$.

Put $A_{i, \sigma} = \exp\left( \iota  s_{\sigma(i)}x^n_{i}\right)$ and $B_{i, \sigma} = - \exp\left( \iota s_{\sigma(i)}x^n_{i + 1}\right)$. We open the brackets on the LHS of (\ref{chsym1}) and obtain a sum of words $sign(\sigma) C_{1,\sigma}\cdots C_{n-1, \sigma}$, where $C = A$ or $B$. We consider the words that have $B$ followed by an $A$ at positions $r,r+1$ and set $\tau$ to be the transposition $(r,r+1)$. Observe that 
$$sign(\sigma)B_{r,\sigma}A_{r+1,\sigma} + sign(\tau \sigma)B_{r,\tau \sigma}A_{r+1,\tau \sigma} = 0, \mbox{ and hence }$$
$$\sum_{\sigma \in S_{n-1}}sign(\sigma)C_{1,\sigma}\cdots C_{r-1,\sigma}B_{r,\sigma}A_{r+1,\sigma}C_{r+2,\sigma}\cdots C_{n-1, \sigma} = 0.$$
The latter implies that the only words that contribute to the LHS of (\ref{chsym1}) are $k$ $A$'s followed by $n- k -1$ $B$'s for $k = 0,..., n-1$. We conclude that the LHS of (\ref{chsym1}) equals
\begin{equation}\label{S5chsym}
\sum_{k = 0}^{n-1}(-1)^{n-1-k}\sum_{\sigma \in S_{n-1}} sign(\sigma)\prod_{i = 1}^{k}\exp\left( \iota  s_{\sigma(i)}x^n_{i}\right)\prod_{i = k+1}^{n-1}\exp\left( \iota  s_{\sigma(i)}x^n_{i+1}\right) 
\end{equation}
and the latter now clearly equals the RHS of (\ref{chsym1}) by inspecting the signs of the summands $\exp\left( \iota \sum_{i = 1}^n s_{\sigma(i)}x^n_{i}\right)$ on both sides for $\sigma \in S_n$.
\end{proof}

\begin{proof}(Lemma \ref{lSym2})
We proceed by induction on $n$. When $n = 2$ we have that 
$$d_n(x^2)\mathbb{E}^{x^2}\left[ g_n(t, x^2,x^1)\right] = e^{\iota t_2x^2_2} \int_{x_1^2}^{x^2_2} e^{\iota t_1x}dx = e^{\iota t_2x^2_2} \frac{e^{\iota t_1x_2^2} - e^{\iota t_1x_1^2}}{\iota t_1}.$$
Consequently, we have 
$$d_n(x^2) \cdot \sum_{\sigma \in S_2}sign(\sigma) ( \iota t_{\sigma(1)}) \mathbb{E}^{x^2}\left[ g_2(t_\sigma, x^2,...,x^1)\right] = e^{ \iota t_1x_2^2 + {\iota t_2x^2_1 } }- e^{ \iota t_2x_2^2 + \iota t_1x^2_1 },$$
from which we conclude the base case.

Suppose we know the result for $n-1 \geq 2$ and we wish to prove it for $n$. We have
$$d_n(x^n) \cdot \sum_{\sigma \in S_n} sign(\sigma)  \mathbb{E}^{x^n}\left[ g_n(t_\sigma, x^n,...,x^1)\right] \prod_{i = 1}^n( \iota t_{\sigma(i)})^{n-i}  = $$
$$ \sum_{\sigma \in S_n} sign(\sigma) e^{\iota t_{\sigma(n)} x^n_n} \prod_{i = 1}^n( \iota t_{\sigma(i)})^{n-i}  \cdot \int_{x^n_{n-1}}^{x^n_{n}} \cdots \int_{x^{n}_{1}}^{x^n_2}dy_{n-1}\cdots dy_1 d_{n-1}(y) \cdot \mathbb{E}^{y}\left[ g_{n-1}(s_{\sigma}, y,x^{n-2},...,x^1)\right],$$
where $s_\sigma = (t_{\sigma(1)}, ..., t_{\sigma(n-1)})$. Splitting the above sum over permutations of $t_{\sigma(1)},...,t_{\sigma(n-1)}$ and applying the induction hypothesis we see that the above equals
$$ (-1)^{ \frac{(n-1)(n-2)}{2}}\sum_{k = 1}^n (-1)^{n-k} e^{\iota t_{k} x^n_n} \prod_{r \neq k}( \iota t_{r})\int_{x^n_{n-1}}^{x^n_{n}} \cdots \int_{x^{n}_{1}}^{x^n_2}dy_{n-1}\cdots dy_1 \sum_{\tau \in S_{n-1}} sign(\tau)\exp\left( \iota \sum_{i = 1}^{n-1} s^k_{\tau(i)}y^n_{i}\right)  $$
$$(-1)^{ \frac{(n-1)(n-2)}{2}}\sum_{k = 1}^n(-1)^{n-k} e^{\iota t_{k} x^n_n} \sum_{\tau \in S_{n-1}} sign(\tau)\prod_{i = 1}^{n-1}\left( \exp\left( \iota  s^k_{\tau(i)}x^n_{i+1}\right) - \exp\left( \iota s^k_{\tau (i)}x^n_{i }\right)\right)$$
where $s^k = (t_1,...,t_{k-1}, t_{k+1},...,t_n)$.

Using equation (\ref{S5chsym}), we may rewrite the above as
\begin{equation*}\label{chsym2}
(-1)^{ \frac{(n-1)(n-2)}{2}}\sum_{k = 1}^n(-1)^{n-k} e^{\iota t_{k} x^n_1} \sum_{l = 0}^{n-1}(-1)^{l}\sum_{\tau \in S_{n-1}} sign(\tau)\prod_{i = 1}^{l}\exp\left( \iota  s^k_{\tau(i)}x^n_{i}\right)\prod_{i = l+1}^{n-1}\exp\left( \iota  s^k_{\tau(i)}x^n_{i+1}\right).
\end{equation*}
If $l < n-1$ we have
$$\sum_{k = 1}^n(-1)^{n-k} e^{\iota t_{k} x^n_n} \sum_{\tau \in S_{n-1}} sign(\tau)\prod_{i = 1}^{l}\exp\left( \iota  s^k_{\tau(i)}x^n_{i}\right)\prod_{i = l+1}^{n-1}\exp\left( \iota  s^k_{\tau(i)}x^n_{i+1}\right) = $$
$$  \sum_{\sigma \in S_{n}} sign(\sigma)\exp\left( \iota  (t_{\sigma(l+1)} + t_{\sigma(n)})x^n_{n}\right)\prod_{i = 1}^{l}\exp\left( \iota  t_{\sigma(i)}x^n_{i}\right)\prod_{i = l+1}^{n-2}\exp\left( \iota t_{\sigma(i)}x^n_{i+1}\right)  = 0.$$
To see the last equality we may swap $l+1$ and $n$ in the above sum by a transposition and observe that we get the same sum but with a flipped sign due to the factors $sign(\sigma)$. Hence, the sum is invariant under change of sign and must be $0$. The last argument shows that only $l = n-1$ contributes in our earlier formula and so we conclude that 
$$d_n(x^n) \cdot \sum_{\sigma \in S_n} sign(\sigma)  \mathbb{E}^{x^n}\left[ g_n(t_\sigma, x^n,...,x^1)\right] \prod_{i = 1}^n( \iota t_{\sigma(i)})^{n-i} =$$
$$ =  (-1)^{ \frac{n(n-1)}{2}}\sum_{k = 1}^n(-1)^{n-k} e^{\iota t_{k} x^n_n} \sum_{\tau \in S_{n-1}} sign(\tau)\prod_{i =1}^{n-1}\exp\left( \iota  s^k_{\tau(i)}x^n_{i}\right).$$
The latter expression is clearly equal to $(-1)^{ \frac{n(n-1)}{2}}\sum_{\sigma \in S_n} sign(\sigma)\exp\left( \iota \sum_{i = 1}^n t_{\sigma(i)}x^n_{i}\right)$, which proves the case $n$. The general result now follows by induction.
\end{proof}

\section{Gibbs measures on Gelfand-Tsetlin patterns}\label{Section6}
The purpose of this section is to analyze probability measures on half-strict Gelfand-Tsetlin patterns $\mathsf{GT}_n^+$, which satisfy what we call the {\em six-vertex Gibbs property} (see Definition \ref{DGP}). An example of such a measure is given by the distribution function of $(Y_i^j)_{1 \leq i \leq j; 1\leq j \leq n}$  (see Section \ref{Section1.2}). The main result of this section is Proposition \ref{S5Prop}, which roughly states that under weak limits the six-vertex Gibbs property becomes the continuous Gibbs property (Definition \ref{definitionGT}).

\subsection{Gibbs measures on the six-vertex model}\label{Section6.1} In this section we define the Gibbs property for the six-vertex model on a domain $D$. We also explain how to symmetrize such a model when $D$ is finite and relate the weight choice in this paper to the ferroelectric phase of the six-vertex model. In what follows we will adopt some of the notation from Appendix A in \cite{Amol2}.\\

Suppose we have a finite domain $D \subset \mathbb{Z}^2$. For $\Lambda \subset \mathbb{Z}^2$, we let $\partial \Lambda$ denote the {\em boundary} of $\Lambda$, which consists of all vertices in $\mathbb{Z}^2 / \Lambda$, which are adjacent to some vertex in $\Lambda$. We consider the six-vertex model on $D$ with fixed boundary condition. This is a probability measure on up-right paths in $D$ with fixed endpoints and we explain its construction below.

We start by assigning certain arrow configurations to the vertices in $\partial D$ and consider all up-right path configurations in $D$, which match the arrow assignments in $\partial D$. Call the latter set $\mathcal{P}(D, \partial D)$. Paths are not allowed to share horizontal or vertical pieces and as in Section \ref{Section1.2} we encode the arrow configuration at a vertex through the four-tuple $(i_1,j_1; i_2, j_2)$, representing the number of incoming and outgoing vertical and horizontal arrows. For $(i,j) \in D$ and $\omega \in \mathcal{P}(D,\partial D)$ we let $\omega(i,j)$ denote the arrow configuration at the corresponding vertex. We have six possible arrow configurations and we define corresponding positive vertex weights as follows
\begin{equation}\label{S6VW1}
\begin{split}
w(0,0;0,0) = w_1, \hspace{2mm}   &w(1,1;1,1) = w_2, \hspace{2mm} w(1,0;1,0) = w_3, \\
w(0,1; 0,1)  = w_4, \hspace{2mm} &w(1,0; 0,1) = w_5, \hspace{2mm} w(0,1;1,0) =w_6.
\end{split}
\end{equation}
The weight of a path configuration $\omega$ is defined through $\mathcal{W} (\omega) := \prod_{(i,j) \in D} w(\omega(i,j))$, and we define the six-vertex model as the the probability measure $\mu$ on $ \mathcal{P}(D,\partial D)$ with probability proportional to $\mathcal{W} (\omega)$. As weights are positive and $D$ is finite this is well-defined.\\

For $\omega \in \mathcal{P}(D, \partial D)$, $\Lambda \subset D$ and an arrow configuration $(i_1, j_1; i_2, j_2)$ we let $N_{\omega; \Lambda} (i_1,j_1; i_2, j_2)$ denote the number of vertices $(x,y) \in \Lambda$ with arrow configuration $(i_1,j_1; i_2, j_2)$. We abbreviate $N_1= N_{\omega;\Lambda}(0, 0; 0, 0)$, $N_2 =N_{\omega;\Lambda}(1, 1; 1, 1)$,  $N_3 = N_{\omega;\Lambda}(1, 0; 1, 0)$, $N_4 = N_{\omega;\Lambda}(0, 1; 0, 1)$, $N_5 = N_{\omega;\Lambda}(1, 0; 0, 1)$, and $N_6 = N_{\omega;\Lambda}(0, 1; 1, 0)$. With this notation we make the following definition.
\begin{definition}\label{GibbsPty}
Fix $w_1,w_2,w_3,w_4,w_5,w_6 > 0$. A probability measure $\rho$ on $\mathcal{P}(D;\partial D)$ is said to satisfy the {\em Gibbs property} (for the six-vertex model on $D$ with weights $(w_1,w_2, w_3, w_4, w_5,w_6)$) if for any finite subset $\Lambda \subset D$ the conditional probability $\rho_\Lambda(\omega)$ of selecting $\omega \in \mathcal{P}(D, \partial D)$ conditioned on $\omega | _{D / \Lambda}$ is proportional to $w_1^{N_1}w_2^{N_2}w_3^{N_3}w_4^{N_4}w_5^{N_5}w_6^{N_6}$. 
\end{definition}
Notice that Definition \ref{GibbsPty} makes sense even if $D$ is not finite. It is easy to see that the measure $\mu$ we defined earlier satisfies the Gibbs property with weights $(w_1,w_2,w_3,w_4,w_5,w_6)$. Similarly, let us consider the measure $\mathbb{P}^{N,M}_{u,v}$ from Definition \ref{parameters2} conditioned on the top row $\lambda^N(\omega)$ being fixed. The latter satisfies the Gibbs property for the domain $D_{N} = \mathbb{Z}_{\geq 0} \times \{1,...,N\}$ with weights
\begin{equation}\label{S6eqWT}
(w_1,w_2,w_3,w_4,w_5,w_6) =  \left( 1 , \frac{u - s^{-1}}{us - 1}, \frac{us^{-1} - 1}{us - 1}, \frac{u - s}{us - 1}, \frac{u(s^2-1)}{us - 1}, \frac{1 - s^{-2}}{us - 1} \right).
\end{equation}
The change of sign above compared to (\ref{VertW}) is made so that the above weights are positive (recall $u > s > 1$ in our case).

If we have $w_1 = w_2 = a$, $w_3 = w_4 = b$ and $w_5 = w_6 = c$ we call the resulting model a {\em symmetric} six-vertex model. Otherwise, we call the model {\em asymmetric}. An important point we want to make is that a single measure $\rho$ on $\mathcal{P}(D, \partial D)$ can satisfy a Gibbs property for many different $6$-tuples of weights $(w_1,w_2, w_3, w_4, w_5,w_6)$. The latter is a consequence of certain {\em conservation laws} satisfied by the quantities $N_{\omega; \Lambda} (i_1,j_1; i_2, j_2)$. As discussed in Appendix A of \cite{Amol2} we have the following conservation laws (see also Section 3 in \cite{BB15}).
\begin{enumerate}
\item The quantity $N_1 + N_2 + N_3 + N_4 + N_5 + N_6 = |\Lambda|$ is constant.
\item Conditioned on $\omega|_{D / \Lambda}$, the quantity $N_2 + N_4 + N_5$ is constant.
\item Conditioned on $\omega|_{D / \Lambda}$, the quantity $N_2 + N_3 + N_6$ is constant.	
\item Conditioned on $\omega|_{D / \Lambda}$, the quantity $N_5 - N_6$ is constant.
\end{enumerate}
The latter imply that if a measure $\rho$ satisfies the Gibbs property with weights  $(w_1,w_2, w_3, w_4, w_5,w_6)$ then  $\rho$ also satisfies the Gibbs property with weights $(xw_1,xyzw_2, xzw_3, xyw_4, xytw_5,xzt^{-1}w_6)$ for any $x,y,z,t > 0$.
 
Let us fix $x = \frac{w_2}{\sqrt{w_1 w_2}}$, $y = \frac{\sqrt{w_1w_2w_3w_4}}{w_2w_4}$, $z = \frac{\sqrt{w_1w_2w_3w_4}}{w_3w_4}$ and $t = \frac{\sqrt{w_4w_6}}{\sqrt{w_3w_5}}$. Then one directly checks that 
$$(xw_1,xyzw_2, xzw_3, xyw_4, xytw_5,xzt^{-1}w_6) = (a,a,b,b,c,c),$$
where  $a = \sqrt{w_1w_2}$, $b = \sqrt{w_3w_4}$ and $c = \sqrt{w_5w_6}.$ The latter shows that any six-vertex model on a finite domain with prescribed boundary condition can be realized as a symmetric six vertex model.

The above arguments can be repeated for other (e.g. periodic) boundary conditions and the consequence is that when working in a finite domain, one can always assume that the six-vertex model is symmetric. This is how the model typically appears in the literature. An important parameter for the symmetric six-vertex model with weights $(a,a,b,b,c,c)$ is given by 
$$\Delta := \frac{a^2 + b^2 - c^2}{2ab}.$$
As discussed in Chapters 8 and 9 in \cite{Bax} (see also \cite{Res10}) the symmetric six-vertex model has several phases called {\em ferroelectric} ($\Delta > 1$), {\em disordered} ($|\Delta| < 1)$ and {\em antiferroelectric} ($\Delta < -1$). 

Based on our earlier discussion, we may extend the definition of $\Delta$ to any (not necessarily symmetric) six-vertex model by
$$\tilde \Delta := \frac{w_1w_2 + w_3w_4 - w_5w_6}{2\sqrt{w_1w_2w_3w_4}}.$$
Observe that the latter quantity is invariant under the transformation of $(w_1,w_2, w_3, w_4, w_5,w_6)$ into $ (xw_1,xyzw_2, xzw_3, xyw_4, xytw_5,xzt^{-1}w_6)$. This implies that the parameter $\tilde \Delta$ for any  six-vertex model on a finite domain agrees with the parameter $\Delta$ for its symmetric realization.

For the six-vertex model we defined in Section \ref{Section1.2} a crucial assumption is that $w_1 = 1$, since our configurations contain infinitely many vertices of type $(0,0;0,0)$. This restriction forbids us from freely rescaling our vertex weights and forces us to work with an asymmetric six-vertex model. However, the above extension of $\Delta$ allows us to investigate to which phase our parameter choice $u > s > 1$ corresponds. As remarked $\mathbb{P}^{N,M}_{u,v}$ satisfies the Gibbs property for the domain $D_{N} = \mathbb{Z}_{\geq 0} \times \{1,...,N\}$ with weights as in (\ref{S6eqWT}). For these weights we find that $\tilde \Delta = (s + s^{-1})/2.$ The latter expression covers $(1,\infty)$ when $s > 1$ and so we conclude that our parameter choice $u > s > 1$ corresponds to the ferroelectric phase of the six-vertex model. 

A natural question that arises from the above discussion is whether we can find different parameter choices for $u$ and $s$, which would land us in the disordered or antiferroelectric phase. If this is achieved one could potentially use the methods of this paper to study the macroscopic behavior of this new model. It would be very interesting to see if the limit shape in Figure \ref{S1_5} changes when we move to a different phase - like in the six-vertex model with periodic (or domain wall) boundary condition. We leave these questions outside of the scope of this paper. 

\subsection{The six-vertex Gibbs property}\label{Section6.2}
We define several important concepts, adopting some of the notation from \cite{Gor14}. Let $\mathsf{GT}_n$ denote the set of $n$-tuples of {\em distinct} integers
$$\mathsf{GT}_{n} = \{ \lambda \in \mathbb{Z}^n: \lambda_1 < \lambda_2 < \cdots < \lambda_n\}.$$
We let $\mathsf{GT}_n^+$ be the subset of $\mathsf{GT}_n$ with $\lambda_1 \geq 0$.  We say that $\lambda \in \mathsf{GT}_n $ and $\mu \in \mathsf{GT}_{n-1}$ {\em interlace} and write $\mu \preceq \lambda$ if 
$$\lambda_1 \leq \mu_1 \leq \lambda_2 \leq \cdots \leq \mu_{n-1} \leq \lambda_n.$$

Let $\mathsf{GT}^{n}$ denote the set of sequences 
$$\mu^1 \preceq \mu^2 \preceq \cdots \preceq \mu^n, \hspace{3mm} \mu^i \in \mathsf{GT}_i, \hspace{2mm} 1 \leq i \leq n.$$
We call elements of $\mathsf{GT}^{n}$ {\em half-strict} Gelfand-Tsetlin patterns (they are also known as monotonous triangles, cf. \cite{MRR}). We also let $\mathsf{GT}^{n+}$ be the subset of $\mathsf{GT}^n$ with $\mu^n \in \mathsf{GT}_n^+$. For $\lambda \in \mathsf{GT}_n$ we let $\mathsf{GT}_\lambda \subset \mathsf{GT}^{n}$ denote the set of half-strict Gelfand-Tsetlin patterns $\mu^1 \preceq \cdots \preceq \mu^n$ such that $\mu^n = \lambda$.\\

We turn back to the notation from Section \ref{Section1.2} and consider $\omega \in \mathcal{P}_n$. For $k = 1,...,n$ we have that $\mu^k_i(\omega) = \lambda^k_{k-i+1}(\omega)$ for $i = 1,..,k$ satisfy $\mu^n \in \mathsf{GT}_n^+$ and $\mu^{k+1} \succeq \mu^k$ for $k = 1,...,n-1$. Consequently, the sequence $\mu^1 ,...,\mu^n$ defines an element of $\mathsf{GT}^{n+}$. It is easy to see that the map $h : \mathcal{P}_n  \rightarrow \mathsf{GT}^{n+}$, given by $h(\omega) = \mu^1(\omega)  \preceq  \cdots \preceq \mu^n(\omega)$, is a bijection. For $\lambda \in \mathsf{GT}_n^+$ we let 
$$\mathcal{P}^{\lambda}_n = \{ \omega \in \mathcal{P}_n : \lambda^n_i(\omega) = \lambda_{n-i+1} \mbox{ for } i = 1,...,n\}.$$
One observes that the map $h$ by restriction is a bijection between $\mathsf{GT}_\lambda$ and $\mathcal{P}^{\lambda}_n$. With the above notation we make the following definition.

\begin{definition}\label{DGP}
Fix $w_1,w_2,w_3,w_4,w_5,w_6 > 0$. A probability distribution $\rho$ on $\mathsf{GT}^{n+}$ is said to satisfy the {\em six-vertex Gibbs property} (with weights $(w_1,w_2,w_3,w_4,w_5,w_6)$) if the following holds. For any $\lambda \in \mathsf{GT}_n^+$ such that $\rho(\mu^n(\omega) = \lambda) > 0$ we have that the measure $\nu$ on $\mathcal{P}^\lambda_n$ defined through
$$\nu( h^{-1}(\omega)) = \rho(\omega| \mu^n = \lambda)$$
satisfies the Gibbs property for the six-vertex model on $D_{n}$ with weights $(w_1,w_2, w_3, w_4, w_5,w_6)$. In the above $ \rho(\cdot | \mu^n = \lambda)$ stands for the measure $\rho$ conditioned on $\mu^n = \lambda$. 
\end{definition}
\begin{remark}
If $w_1 = \cdots = w_6 = 1$ and $\rho$ satisfies the six-vertex Gibbs property with these weights then the conditional distribution $\rho( \cdot | \mu^n = \lambda) $ becomes the uniform distribution on $\mathsf{GT}_\lambda$. In this case the six-vertex Gibbs property reduces to the {\em discrete Gibbs property} of \cite{Gor14}.
\end{remark}
For a probability distribution $\rho$ on $\mathsf{GT}^{n+}$ and an element $\lambda \in \mathsf{GT}_n^+$ such that $\rho(\mu^n(\omega) = \lambda) > 0$, we denote by $\rho_\lambda$ the distribution on $\mathsf{GT}_{n-1}^+$ given by $\rho(\mu^1,...,\mu^{n-1}|\mu^n = \lambda)$. We let $\rho^k$ denote the projection of $\rho$ onto $\mu^k$ for $k = 1,...,n$. Then the six-vertex Gibbs property is equivallent to the following statement. If $f: \mathbb{Z}^{n(n+1)/2} \rightarrow \mathbb{C}$ is a bounded function, then
\begin{equation}\label{SVGB}
\mathbb{E}^\rho\left[ f(\mu^n,...,\mu^1)\right] = \mathbb{E}^{\rho^n} \left[ \mathbb{E}^{\rho_{\mu^n}} \left[f(\mu^n,...,\mu^1) \right] \right].
\end{equation}

We record two lemmas whose proof is deferred to Section \ref{Section6.3}.

\begin{lemma}\label{LemmaBound}
Fix $w_1,w_2,w_3,w_4,w_5,w_6 > 0$. Let $n \in \mathbb{N}$ and $\rho$ be a measure on $\mathsf{GT}^{n+}$, which satisfies the six-vertex Gibbs property. Then we can find a positive constant $c \in (0,1)$ (depending on $n$ and $w_1,...,w_6$) such that for all $\lambda \in \mathsf{GT}_n^+$ with $\rho(\mu^n = \lambda) > 0$, and $(\alpha^1,...,\alpha^{n-1}, \lambda), (\beta^1,...,\beta^{n-1}, \lambda) \in \mathsf{GT}_\lambda$ we have
$$ c^{-1} \geq \frac{\rho_{\lambda}(\alpha^1,...,\alpha^{n-1} )}{\rho_{\lambda}(\beta^1,...,\beta^{n-1})} \geq c.$$
\end{lemma}

\begin{definition}\label{DefVG} We consider sequences $\lambda^k \in \mathsf{GT}_n$. We call the sequence {\em very good} if for $i = 1,...,n-1$ each sequence $\lambda^k_{i+1} - \lambda^k_i$ has a limit in $\mathbb{N} \cup \{\infty \}$ and for $i = 1,...,n-2$, each sequence $\lambda^k_{i+2} - \lambda^k_i$ goes to $\infty$. We call the sequence {\em good} if every subsequence of $\lambda^k$ has a further subsequence that is very good. 
\end{definition}

\begin{lemma}\label{LemmaWeak}
Fix $n \in \mathbb{N}$. Let $a(k)$ and $b(k)$ be sequences in $\mathbb{R}$ such that $b(k) \rightarrow \infty$ as $k \rightarrow \infty$. Suppose that $\lambda^k$ is a good sequence in $\mathsf{GT}_n^+$ and $f$ is a bounded uniformly continuous function on $\mathbb{R}^{n(n+1)/2}$. Put $g_k: \mathbb{R}^{n(n+1)/2} \rightarrow  \mathbb{R}^{n(n+1)/2}$ to be
$$g_k(x) =  \frac{1}{b(k)}\left( x - a(k) \cdot {\bf 1}_{\frac{n(n+1)}{2}} \right).$$
Then we have
\begin{equation}\label{goodSeq}
\lim_{k \rightarrow \infty} \mathbb{E}^{\rho_{\lambda^k}} \left[ f \circ g_k (\lambda^k, \mu^{n-1}, \mu^{n-2},...,\mu^1) \right] - \mathbb{E}^{\lambda^k} \left[ f \circ g_k (\lambda^k, x^{n-1}, x^{n-2},...,x^1) \right]  = 0,
\end{equation}
where $ \mathbb{E}^{\rho_{\lambda^k}}$ is defined above while $\mathbb{E}^{\lambda^k}$ is as in Section \ref{Section5.1}.
\end{lemma}

With the above lemma we can prove the main result of this section.

\begin{proposition}\label{S5Prop}
Fix $w_1,w_2,w_3,w_4,w_5,w_6 > 0$ and $n \in \mathbb{N}$. Let $\rho(k)$ be a sequence of probability measures on $\mathsf{GT}^{n+}$, satisfying the six-vertex Gibbs property with weights $(w_1,w_2,w_3,w_4,w_5,w_6)$. Let $a(k)$ and $b(k)$ be sequences in $\mathbb{R}$ such that $b(k) \rightarrow \infty$ as $k \rightarrow \infty$. Put $g_k: \mathbb{R}^{n(n+1)/2} \rightarrow  \mathbb{R}^{n(n+1)/2}$ to be
$$g_k(x) =  \frac{1}{b(k)}\left( x - a(k) \cdot {\bf 1}_{\frac{n(n+1)}{2}} \right),$$
and suppose that $\rho(k) \circ g_k^{-1}$ converges weakly to a probability distribution $\mu$ on $GT^n$ (Gelfand-Tsetlin cone), such that
\begin{equation}\label{techCond}
\mathbb{P}^{\mu}( y_i^n = y_{i+1}^n = y_{i+2}^n \mbox{ for some } i = 1,...,n-2) = 0.
\end{equation}
Then $\mu$ satisfies the continuous Gibbs property (Definition \ref{definitionGT}).
\end{proposition}
\begin{remark}
The statement of the proposition remains true if we remove the condition (\ref{techCond}) on $\mu$; however, its proof requires a stronger statement than Lemma \ref{LemmaWeak}. For the applications we have in mind Proposition \ref{S5Prop} is sufficient and we will not pursue the most general possible result here.
\end{remark}

\begin{proof}
 By Skorohod's theorem, we may find random vectors $Y(k)$ for $k\in \mathbb{N}$ and $X$, defined on the same probability space $(\Omega, \mathcal{F}, \mathbb{P})$, such that $Y(k)$ have distribution $\rho(k)$, $X$ has distribution $\mu$, and
$$ \mathbb{P}\left( \{ \omega \in \Omega| \lim_{k\rightarrow \infty} g_k(Y(k)(\omega)) = X(\omega) \}\right) = 1.$$

Let $f$ be a bounded continuous function on $\mathbb{R}^{n(n+1)/2}$. As usual we write $Y(k) = Y^1(k) \preceq \cdots \preceq Y^n(k)$ and $X = X^1 \preceq \cdots \preceq X^n$. We want to show that 
$$\mathbb{E} \left[ f(X) \right] = \mathbb{E}\left[ \mathbb{E}^{X^n}\left[f(X)\right]\right].$$
From the Bounded Convergence Theorem we know that
\begin{equation}\label{s51}
\mathbb{E} \left[ f(X) \right]  = \lim_{ k \rightarrow \infty} \mathbb{E}\left[ f(g_k(Y(k)))\right].
\end{equation}
We now let $A =\{ \omega \in \Omega | \lim_{k\rightarrow \infty} g_k(Y(k)(\omega)) = X(\omega) \mbox{ and } X^n_i(\omega) =  X^n_{i+1}(\omega)  = X^n_{i+2}(\omega) \mbox{ for no }i \}$. One readily observes that for $\omega \in A$, $Y(k)(\omega)$ is a good sequence and so by Lemma \ref{LemmaWeak}, we conclude that 
\begin{equation*}
\lim_{k \rightarrow \infty} \mathbb{E}^{Y^n(k)(\omega)}\left[f(g_k(Y(k)))\right] - \mathbb{E}^{\rho_{Y^n(k)}(\omega)}\left[f(g_k(Y(k)))\right] = 0.
\end{equation*}
Taking expectations on both sides above (which is justified by the Bounded convergence theorem) and using that $\rho(k)$ satisfy the six-vertex Gibbs property (see also (\ref{SVGB})), we conclude that
\begin{equation}\label{s52}
\lim_{k \rightarrow \infty} \mathbb{E} \left[\mathbb{E}^{Y^n(k)(\omega)}\left[f(g_k(Y(k)))\right]\right] - \mathbb{E}\left[f(g_k(Y(k)))\right] = 0.
\end{equation}
Finally, if $\omega \in A$ and $Z(k) = g_k(Y(k))$, we have by Lemma \ref{LemmaC3} that
$$\lim_{k \rightarrow \infty} \mathbb{E}^{Y^n(k)(\omega)}\left[f(g_k(Y(k)))\right] = \lim_{k \rightarrow \infty} \mathbb{E}^{Z^n(k)(\omega)}\left[f(Z(k))\right] = \mathbb{E}^{X^n(\omega)}\left[f(X)\right].$$
Taking expectations on both sides above (which is justified by the Bounded convergence theorem) we conclude that 
\begin{equation}\label{s53}
\lim_{k \rightarrow \infty} \mathbb{E} \left[\mathbb{E}^{Y^n(k)(\omega)}\left[f(g_k(Y(k)))\right]\right]= \mathbb{E}\left[\mathbb{E}^{X^n(\omega)}\left[f(X)\right]\right].
\end{equation}
Combining (\ref{s51}), (\ref{s52}) and (\ref{s53}) proves the proposition.

\end{proof}

\subsection{Proof of Lemmas \ref{LemmaBound} and \ref{LemmaWeak}} \label{Section6.3}
We adopt the same notation as in Section \ref{Section6.2}.

\begin{proof}(Lemma \ref{LemmaBound}) Introduce vertex weights as in (\ref{S6VW1}). For $\lambda \in \mathsf{Sign}_n^+$ we fix $\omega_\lambda \in \mathcal{P}_n$, such that $\lambda^{j}_i(\omega_\lambda) = \lambda_i$ for $i = 1,...,j$ and $j = 1,...,n$. We also define for $\omega \in \mathcal{P}_n$ the weight
$\mathcal{W}(\omega) := \prod_{i = 1}^{n} \prod_{j = 1}^{\lambda_1} w(\omega(i,j) )$.

Since $\rho$ satisfies the conditions of  Definition \ref{DGP}, it is enough to show that for each $\lambda \in \mathsf{Sign}_n^+$, and any collection of paths $\omega \in \mathcal{P}_n,$ with $\lambda^n_i(\omega) =  \lambda_i$ for $i = 1,...,n$, we have
$$ c^{-1} \geq \frac{\mathcal{W}(\omega) }{\mathcal{W}(\omega_\lambda) } \geq c,$$
for some $c \in (0,1)$, which depends on $n$ and $w_1,...,w_6$. The strategy is to apply elementary moves to the configuration $\omega$ that transform it to $\omega_\lambda$, and record how the weight changes at each step. We will see that the number of changes is at most $n(n-1)$ and each change is given by a multiplication by some factor, which can take finitely many values, depending on $w_1,...,w_6$. This will show that $\frac{\mathcal{W}(\omega)}{\mathcal{W}(\omega_\lambda)} $ belongs to a finite set of numbers, which then can be upper and lower bounded, proving the lemma.\\

Let $\mathcal{P}_n^\lambda$ denote the set of $ \omega \in \mathcal{P}_n$ such that $\lambda^n(\omega) = \lambda$. Starting from any $\omega \in \mathcal{P}^\lambda_n$ an elementary move consists of increasing one of $\lambda_i^j(\omega)$ by $1$ so that the resulting element still lies in $\mathcal{P}^\lambda_n$. If we apply an elementary move to $\omega$, increasing $m = \lambda_i^j(\omega)$ by $1$ and obtain $\omega^+ \in \mathcal{P}_n^\lambda$ as a result, we observe that 
$$ \frac{\mathcal{W}(\omega)}{\mathcal{W}(\omega^+)} = \frac{w(\omega(m,j))w(\omega(m,j+1))w(\omega(m+1,j))w(\omega(m+1,j+1))}{w(\omega^+(m,j))w(\omega^+(m,j+1))w(\omega^+(m+1,j))w(\omega^+(m+1,j+1))}.$$
Since we have only finitely many possible vertex weights we see that $ \frac{\mathcal{W}(\omega)}{\mathcal{W}(\omega^+)}$ can take finitely many values.

The way we transform $\omega$ to $\omega_\lambda$ is as follows. We consider the complete order on pairs $(x,y)$ given by $(x,y) < (x',y')$ if and only if $x < x'$ or $x = x'$ and $y > y'$. We traverse the pairs $(i,j)$: $i = 1,...,j$, $j = 1,..,n-1$ in increasing order, and for each $(i,j)$ we increase $\lambda^j_i(\omega)$ by $1$ until it reaches $\lambda^j_i(\omega_\lambda)$. One readily observes that each such move is elementary and the result of applying all these moves to $\omega$ is indeed $\omega_\lambda$. We continue to denote the result of applying an elementary move to $\omega$ by $\omega$ - this should cause no confusion.\\

An important situation occurs when prior to the application of the move $m =\lambda_i^j(\omega) \rightarrow m +1$ we have that there are no vertical arrows coming in $(m, j)$ and $(m + 1, j)$ and coming out of $(m,j+1)$ and $(m + 1, j+1)$. The latter situation determines the types of the four vertices:
$$\mbox{ $\omega(m, j) = (0,1;1,0)$, $\omega(m+1, j) = (0,0;0,0)$, $\omega(m, j + 1) = (1,0;0,1)$, $\omega(m + 1, j+ 1) = (0,1;0,1)$}.$$
After the application of the move they become
$$\mbox{ $\omega(m, j) = (0,1;0,1)$, $\omega(m+1, j)= (0,1;1,0)$, $\omega(m, j + 1) = (0,0;0,0)$, $\omega(m + 1, j+ 1)= (1,0;0,1)$}.$$
We thus see that the product of these weights stays the same and so $\mathcal{W}(\omega)$ remains unchanged. We call such a situation {\em good}. 

Suppose that in the string of elementary moves, transforming $\omega$ to $\omega_\lambda$, we have reached the pair $(i,j)$, and we are increasing $\lambda^j_i(\omega)$ to $\lambda^j_i(\omega_\lambda)$. Let us denote $A = \lambda^j_i(\omega)$ and $B = \lambda^j_i(\omega_\lambda)$. The condition that we can increase $\lambda^j_i(\omega)$ to $B$ via elementary moves, implies that there are no arrows from $(k, j)$ to $(k, j+1)$ or from $(k, j - 1)$ to $(k,j)$  for $k = A+1,...,B-1$. Consequently, in the process of increasing $\lambda^j_i(\omega)$ to $B$, we encounter at most two non-good situations (corresponding to the first and last move). As we have $n(n-1)/2$ pairs $(i,j)$, we see that in our string of elementary moves the situation is good in all but at most $n(n-1)$ moves. This proves our desired result.

\end{proof}

Before we go to the proof of Lemma \ref{LemmaWeak}, we introduce some notation and prove a couple of facts. \vspace{1mm}  Let $\overline{\mathsf{GT}}_n$ denote the set of $n$-tuples of integers
$$\overline{\mathsf{GT}}_{n} = \{ \lambda \in \mathbb{Z}^n: \lambda_1 \leq \lambda_2 \leq \cdots \leq \lambda_n\}.$$
Let $\overline{\mathsf{GT}}^{n}$ denote the set of sequences 
\vspace{1mm}
$$\mu^1 \preceq \mu^2 \preceq \cdots \preceq \mu^n, \hspace{3mm} \mu^i \in \overline{\mathsf{GT}}_i, \hspace{2mm} 1 \leq i \leq n.$$
We call elements of $\overline{\mathsf{GT}}^{n}$ Gelfand-Tsetlin patterns. For $\lambda \in \overline{\mathsf{GT}}_n$ we let $\overline{\mathsf{GT}}_\lambda \subset \overline{\mathsf{GT}}^{n}$ denote the set of Gelfand-Tsetlin patterns $\mu^1 \preceq \cdots \preceq \mu^n$ such that $\mu^n = \lambda$.

We say that $\lambda \in \mathsf{GT}_n $ and $\mu \in \mathsf{GT}_{n-1}$ {\em strictly interlace} and write $\mu \prec \lambda$ if 
$$\lambda_1 < \mu_1 < \lambda_2 < \cdots < \mu_{n-1} < \lambda_n.$$
Let $\widehat{\mathsf{GT}}^{n}$ denote the set of sequences 
$$\mu^1 \prec \mu^2 \prec \cdots \prec \mu^n, \hspace{3mm} \mu^i \in \mathsf{GT}_i, \hspace{2mm} 1 \leq i \leq n.$$
We call elements of $\widehat{\mathsf{GT}}^{n}$ {\em strict} Gelfand-Tsetlin patterns. For $\lambda \in \mathsf{GT}_n$ we let $\widehat{\mathsf{GT}}_\lambda \subset \widehat{\mathsf{GT}}^{n}$ denote the set of Gelfand-Tsetlin patterns $\mu^1 \prec \cdots \prec \mu^n$ such that $\mu^n = \lambda$.

For $\lambda \in \overline{\mathsf{GT}}_n$ we consider the size of $\overline{\mathsf{GT}}_\lambda$, which is equal to the dimension of the representation of the unitary griup $U(n)$ with the highest weight $\lambda$ and is given by the well-known formula
\vspace{1mm}
\begin{equation}\label{Dvol1}
\left|\overline{\mathsf{GT}}_\lambda \right|= \prod_{i < j} \left( \frac{\lambda_j - \lambda_i + j - i}{j-i}\right).
\end{equation}
For $\lambda \in \overline{\mathsf{GT}}_n$ we let $\lambda^* = (\lambda_1 + (n - 1), \lambda_2 + (n-3),..., \lambda_n + (-n + 1))$. It is easy to check that if $\lambda \in \overline{\mathsf{GT}}_n$ and $\mu \in \overline{\mathsf{GT}}_{n-1}$ then $\mu \preceq \lambda$ if and only if $\mu^* \prec \lambda^*$. Consequently, we have that the map \vspace{1mm}\\
$  f: \overline{\mathsf{GT}}_\lambda  \rightarrow \widehat{\mathsf{GT}}_{\lambda^*}$, given by $f(\mu^1, ..., \mu^{n-1}, \lambda) = ((\mu^{1})^*,...,(\mu^{n-1})^*, \lambda^*)$, is a bijection.

It follows from (\ref{Dvol1}) that for $\lambda \in \mathsf{GT}_n$, the size of $ \widehat{\mathsf{GT}}_{\lambda}$ is given by
\begin{equation}\label{Dvol2}
\left| \widehat{\mathsf{GT}}_{\lambda} \right|= \prod_{i < j} \left( \frac{\lambda_j - \lambda_i  - j + i}{j-i}\right).
\end{equation}

Let us recall from Definition \ref{DefVG} that a sequence in $\mathsf{GT}_n$ is very good if for $i = 1,...,n-1$ each sequence $\lambda^k_{i+1} - \lambda^k_i$ has a limit in $\mathbb{N} \cup \{\infty \}$, while  for $i = 1,...,n-2$, each sequence $\lambda^k_{i+2} - \lambda^k_i$ goes to $\infty$. For each very good sequence $\lambda^k$, we let $M \subset \{1,...,n-1\}$, be the set of indices $i$, such that $\lambda^k_{i+1} - \lambda^k_i$ is bounded, and for $i \in M$, denote by $m_i \in \mathbb{N}$ the limit of the sequence $\lambda^k_{i+1} - \lambda^k_i$, which exists by assumption.

Given a subset $M \subset \{1,...,n-1\}$ and $\lambda \in \mathsf{GT}_n$, we let 
$$\widehat{\mathsf{GT}}_\lambda(M): = \{ (\mu^1,..., \mu^n ) \in \mathsf{GT}_\lambda: \mu^n = \lambda, \mu^{n-1} \succ \mu^{n-2} \succ \cdots \succ \mu^1, \mu^{n-1}_i \in (\lambda_i, \lambda_{i+1}) \mbox{ for } i\not \in M\}.$$
For $ (\mu^1,..., \mu^n ) \in \widehat{\mathsf{GT}}_\lambda(M)$ and numbers $x_i$ for $i \in M$, we define the function
$$f^{\bf x}_M(\mu^1,...,\mu^n) = (\nu^1,...,\nu^n) \mbox{ with } \nu^{j}_i = \begin{cases} x_i &\mbox{ if $j = n-1$ and $i \in M,$} \\ \mu^j_i &\mbox{ else.} \end{cases}$$
We now define
$$\widehat{\mathsf{GT}}_\lambda(M; f_M):=  \{ (\mu^1,..., \mu^n ) \in \widehat{\mathsf{GT}}_\lambda(M): f^{\bf x}_M(\mu^1,..., \mu^n ) \in \widehat{\mathsf{GT}}_\lambda(M) \mbox{ for all } x_i \in [\lambda_{i}, \lambda_{i+1}], i \in M \}.$$

The first key result we need is the following.
\begin{lemma}\label{Lvol}
Let $\lambda^k \in \mathsf{GT}_n$ be a very good sequence and $M$ as above. As $k \rightarrow \infty$, we have
$$\left| \mathsf{GT}_{\lambda^k} \right| \sim \left| \overline{ \mathsf{GT}}_{\lambda^k} \right| \sim \left| \widehat{ \mathsf{GT}}_{\lambda^k}(M;f_M) \right| \sim \prod_{\substack{ 1 \leq i < j \leq n \\ i + 1 <   j} }  \left( \frac{\lambda^k_j - \lambda^k_i }{j-i}\right) \times \prod_{i \not \in M, i < n} (\lambda^k_{i+1} - \lambda^k_i ) \times \prod_{i \in M} (m_i + 1).$$
\end{lemma}
\begin{proof}
\vspace{1mm}

We observe that $\widehat{\mathsf{GT}}_{\lambda^k}(M;f_M) \subset \mathsf{GT}_{\lambda^k} \subset \overline{ \mathsf{GT}}_{\lambda^k} $, and so by (\ref{Dvol1}), it suffices to show that
\begin{equation}\label{enough}
\limsup_{k \rightarrow \infty} \left| \widehat{ \mathsf{GT}}_{\lambda^k}(M;f_M) \right|^{-1}  \prod_{\substack{ 1 \leq i < j \leq n \\ i + 1 <   j} }  \left( \frac{\lambda^k_j - \lambda^k_i }{j-i}\right) \times \prod_{i \not \in M, i < n} (\lambda^k_{i+1} - \lambda^k_i ) \times \prod_{i \in M} (m_i + 1) \leq 1. 
\end{equation}

For $i \in \{ 1,..., n \}$, we let $X_i = \{ (x,y):  i \leq x \leq y \leq n\}$, and $Y_i = \{ (x,y): 1\leq x \leq y \leq n \mbox{ and } y - x < n - i - 1 \}$. Let $(\mu^1,..., \mu^{n}) \in \mathsf{GT}^{n}$ and let $i \in \{ 0,..., n \}$ be given. Then it is easy to see that if we increase $\mu^y_x$ by $1$ for all $(x,y) \in X_i$ or, alternatively, for all $(x,y) \in Y_i$, we still get an element that belongs to $\mathsf{GT}^{n}$. See Figure \ref{S7_1}. 

\begin{figure}[h]
\centering
\scalebox{0.6}{\includegraphics{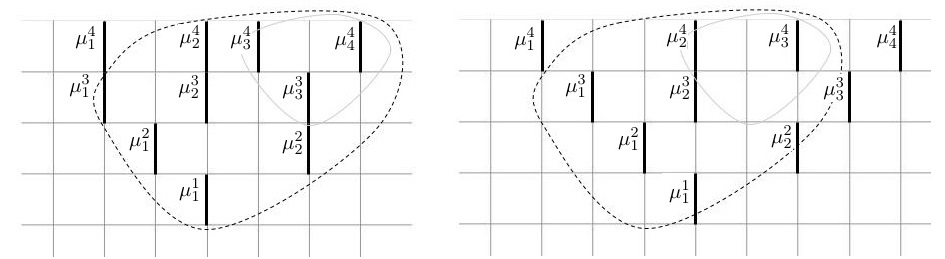}}
\caption{An element $\mu \in \mathsf{GT}^4$. In left picture the grey curve encloses $\mu^y_x$ for $(x,y) \in X_3$ and the dashed curve  for $(x,y) \in Y_0$. The right picture is the result of increasing $\mu^y_x$ by $1$ for $(x,y) \in X_3$ and then for $(x,y) \in Y_0$.  }
\label{S7_1}
\end{figure}

Let $M  \subset \{1,...,n-1\}$, be such that $M$ does not contain adjacent elements. Suppose that $\lambda \in \mathsf{GT}_n$ is such that $\lambda_{i+1} - \lambda_{i} = 2$ when $i \in M$. Suppose $m_i \in \mathbb{N}$ for $i \in M$ are given. Let $c_i \in \{0,...,m_i\}$ for $i \in M$. Starting from an element $\mu^1,...,\mu^n$ in $\widehat{\mathsf{GT}}_\lambda$, and given $c_i$ for $i \in M$ as above, we construct a new element in $\mathsf{GT}^n$ as follows:

\begin{enumerate}
\item We traverse the elements in $M$ in increasing order.
\item For each element $i \geq 2, i \in M$, we increase the values $\mu^y_x$ for each $(x,y) \in X_i$ by $1$.
\item Afterwards we increase $\mu^y_x$ by $m_i$ for each $(x,y) \in Y_i$.
\item Finally, we set $\mu^{n}_{i+1} = \mu^{n}_i + m_i$ and set $\mu^{n-1}_i$  to equal $\mu^{n}_i + c_i$.
\end{enumerate}
For a simple application of the above algorithm see Figure \ref{S7_2}.
\begin{figure}[h]
\centering
\scalebox{0.6}{\includegraphics{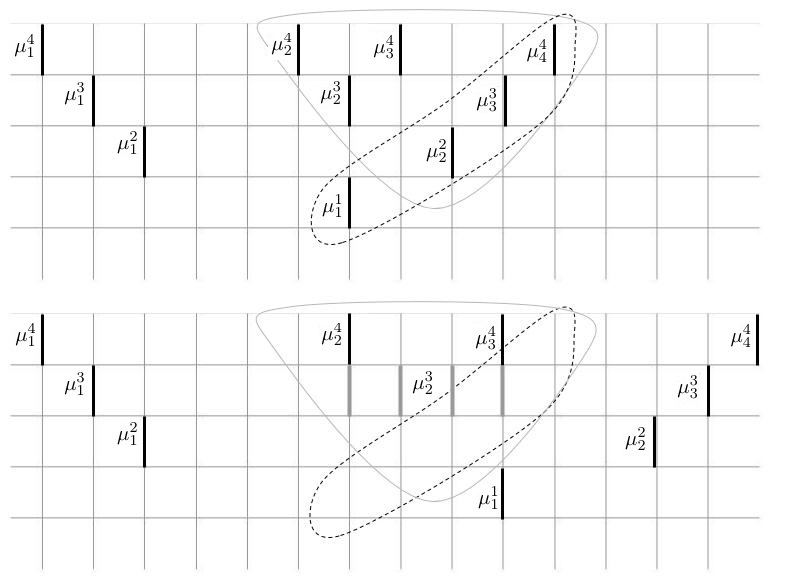}}
\caption{The top picture gives an element $\mu \in \widehat{\mathsf{GT}}^4$; the grey curve encloses $\mu^y_x$ for $(x,y) \in X_2$ and the dashed curve for $(x,y) \in Y_2$. If $M = \{2\}$ and $m_2 = 3$ the bottom picture gives the output of applying our algorithm to $\mu$. The position $\mu_2^3$ can be any element in $[ \mu^4_2 , \mu_3^4]$. }
\label{S7_2}
\end{figure}

One readily verifies that each element that was constructed with the above algorithm belongs to $\widehat{\mathsf{GT}}_\nu(M;f_M)$, where $\nu_1 = \lambda_1$ and for $i = 2,...,n$, we have
$$\nu_{i+1} - \nu_i = \begin{cases} m_i &\mbox{ if $i \in M$,} \\ \lambda_{i+1} - \lambda_i + 1 & \mbox{ if $i + 1 \in M$ and $i - 1 \not \in M$,} \\ \lambda_{i+1} - \lambda_i + 2 &\mbox{ if $i + 1 \not \in M$ and $i -1 \in M$,}   \\ \lambda_{i+1} - \lambda_i + 3 &\mbox{ if $i + 1 \in M$ and $i -1 \in M$} \\  \lambda_{i+1} - \lambda_i &\mbox{ if $i, i+1, i - 1 \not \in M$.}\end{cases} $$
We thus obtain a map from $\widehat{\mathsf{GT}}_\lambda \times \prod_{i \in M} \{0,...,m_i\}$ into $\widehat{\mathsf{GT}}_\nu(M;f_M)$, and it is easy to see that it is injective. The latter implies that 
$$ \left| \widehat{\mathsf{GT}}_\nu(M;f_M)\right| \geq \prod_{i \in M} (m_i + 1) \times \left| \widehat{\mathsf{GT}}_{\lambda} \right|.$$
Combining the above with (\ref{Dvol2}) we see that if $\nu \in \mathsf{GT}_n$ is such that $\nu_{i+1} - \nu_i = m_i$ for $i \in M$ and $\nu_{i+1} - \nu_{i} \geq 4$ for $i \not \in M$, then 
$$ \left| \widehat{\mathsf{GT}}_\nu(M;f_M)\right| \geq  \prod_{i \in M} (m_i + 1) \times  \prod_{\substack{ 1 \leq i < j \leq n \\ i + 1 <   j} }  \left( \frac{\nu_j - \nu_i - 4(j-i)}{j-i}\right) \times \prod_{i \not \in M, i < n} (\nu_{i+1} - \nu_i - 4 ).$$
This readily implies (\ref{enough}) and hence the lemma.
\end{proof}

An important property of $\widehat{\mathsf{GT}}_\lambda(M)$ is contained in the following lemma.
\begin{lemma}\label{LEq}
Fix $w_1,w_2,w_3,w_4,w_5,w_6 > 0$ and $n \in \mathbb{N}$. Suppose $\rho$ is a probability distribution on $\mathsf{GT}^{n+}$, which satisfies the six-vertex Gibbs property with weights $(w_1,w_2,w_3,w_4,w_5,w_6)$. Let $\lambda \in \mathsf{GT}_n$ be such that $\rho(\mu^n(\omega) = \lambda ) > 0$ and define $\rho_\lambda$ as in Section \ref{Section6.2}. Let $M  \subset \{1,...,n-1\}$ and suppose that $(\mu^1,...,\mu^n),$
$ (\nu^1,...,\nu^n) \in \widehat{\mathsf{GT}}_\lambda(M)$, are such that $\mu^{n-1}_i = \nu^{n-1}_i$ for $i \in M$. Then $\rho_\lambda(\mu^1,...,\mu^{n-1}) = \rho_\lambda(\nu^1,...,\nu^{n-1}) $.
\end{lemma}
\begin{proof}
We set $\omega_ 1 = h^{-1} ( (\mu^1,...,\mu^{n-1}, \lambda)$ and $\omega_2 = h^{-1} ( (\nu^1,...,\nu^{n-1}, \lambda)$ (the function $h$ was defined in Section \ref{Section6.2}). By definition we know that $\lambda^n_i(\omega_1) = \lambda^n_i(\omega_2) = \lambda_{n-i+1}$ for $i = 1,...,n$. As in the proof of Lemma \ref{LemmaBound} we introduce vertex weights as in (\ref{S6VW1}) and define for $\omega \in \mathcal{P}_n$ the weight $\mathcal{W}(\omega) := \prod_{i = 1}^{n} \prod_{j = 1}^{\lambda_n} w(\omega(i,j) )$. Since $\rho$ satisfies the six-vertex Gibbs property we see that to prove the lemma it suffices to show that $\mathcal{W}(\omega_1) = \mathcal{W}(\omega_2)$.

Recalling the proof of Lemma \ref{LemmaBound}, we see that it suffices to show that we can transform $\omega_1$ to $\omega_2$ via good elementary moves. I.e. we wish to show that any two elements in $h^{-1}(\widehat{\mathsf{GT}}_\lambda(M))$, that satisfy $\lambda^{n-1}_i(\omega_1) = \lambda^{n-1}_i(\omega_2)$ for $n - i  \in M$ are connected via good elementary moves. We prove the latter by induction on $|\omega_1 - \omega_2| := \sum_{j = 1}^{n-1}\sum_{i = 1}^j |\lambda^j_i(\omega_1) - \lambda^j_i(\omega_2)|$, the base case $|\omega_1 - \omega_2| = 0$ being obvious.\\

Suppose we know the result for $|\omega_1 - \omega_2| = k-1 \geq 0$, and we wish to show it for $k$. Since $|\omega_1 - \omega_2| = k \geq 1$, we know that there exist $(x,y)$ such that $\lambda^y_x(\omega_1) - \lambda^y_x(\omega_2) \neq 0$. Let $(x,y)$ be the smallest such index (in the order considered in the proof of Lemma \ref{LemmaBound}), and without loss of generality we assume that $\lambda^y_x(\omega_1) > \lambda^y_x(\omega_2)$. Notice that by assumption $(x,y) \neq (i,n-1)$ for any $n-i \in M$ and also $y \leq n-1$. 

We want to increase $\lambda^y_x(\omega_2)$ by $1$ and show that this is a good elementary move. In order for this to be the case we must have that $\lambda^y_{x-1}(\omega_2)$,  $\lambda^{y-1}_{x-1}(\omega_2)$  (if $x > 1$) and $\lambda^{y+1}_x(\omega_2)$ are all strictly bigger than $ \lambda^y_x(\omega_2) + 1$. Observe that 
$$\lambda^{y+1}_x(\omega_2) = \lambda^{y+1}_x(\omega_1) \geq  \lambda^{y}_x(\omega_1)  + 1\geq \lambda^{y}_x(\omega_2)  + 2,$$
where in the first equality we used the minimality of $(x,y)$, in the second one we used that $(x,y) \neq (i,n-1)$ for any $n-i \in M$ and in the third that $\lambda^y_x(\omega_1) > \lambda^y_x(\omega_2)$. Similarly, we have for $x > 1$ that
$$\lambda^y_{x-1}(\omega_2)= \lambda^{y}_{x-1}(\omega_1) \geq  \lambda^{y}_x(\omega_1)  + 1\geq \lambda^{y}_x(\omega_2)  + 2,$$
and 
$$\lambda^{y-1}_{x-1}(\omega_2)= \lambda^{y-1}_{x-1}(\omega_1) \geq  \lambda^{y}_x(\omega_1)  + 1\geq \lambda^{y}_x(\omega_2)  + 2.$$
Thus increasing $\lambda^y_x(\omega_2)$ by $1$ is a good elementary move, and does not change $\mathcal{W}(\omega_2)$, while it reduces $|\omega_1 - \omega_2|$ by $1$. Applying the induction hypothesis proves the result for $k$, and the general result follows by induction.
\end{proof}

With the above two results, we now turn to the proof of Lemma \ref{LemmaWeak}.
\begin{proof}(Lemma \ref{LemmaWeak})
Clearly it suffices to prove the lemma when $\lambda^k$ is very good. As before we let $M \subset \{1,...,n-1\}$, be the set of indices $i$, such that $\lambda^k_{i+1} - \lambda^k_i$ is bounded, and for $i \in M$, denote by $m_i \in \mathbb{N}$ the limit of the sequence $\lambda^k_{i+1} - \lambda^k_i$. By ignoring finitely many elements of the sequence $\lambda^k$, we may assume that $\lambda^k_{i+1} - \lambda^k_i = m_i$ for all $k$. We denote by $M_k = \prod_{\substack{ 1 \leq i < j \leq n \\ i + 1 <   j} }  \left( \frac{\lambda^k_j - \lambda^k_i }{j-i}\right)\times \prod_{i \not \in M, i < n} (\lambda^k_{i+1} - \lambda^k_i )$.

For $y = y_i^j$, $i = 1,...,j$ and $j = 1,...,n$, we let $Q(y)$ be the cube in $\mathbb{R}^{n(n-1)/2}$, given by $\prod_{j = 1}^{n-1} \prod_{i = 1}^j (y_i^j, y_i^j + 1)$. For $\lambda \in \mathsf{GT}_n$, we define
$$ \mathsf{GT}^*_{\lambda} := \{ (\mu^1,...,\mu^n) \in \widehat{\mathsf{GT}}_\lambda(M;f_M) : \mu^{n-1}_i < \lambda_i \mbox{ for } i \in M\}.$$
It is easy to see that if $(\mu^1,...,\mu^n) \in \mathsf{GT}^*_{\lambda}$, and $y^j_i = \mu^j_i$ for $i = 1,...,j$ and $j = 1,...,n$, then 
$$x^1 \preceq x^2 \preceq \cdots \preceq x^{n-1} \preceq \lambda \mbox{, for any } (x^1,...,x^{n-1}) \in Q(y).$$

Fix $c_i \in \{0,...,m_i\}$ for $i \in M$, and set $\mathsf{GT}^{*}_{\lambda^k}({\bf c}) = \{ (\mu^1,...,\mu^n)  \in \mathsf{GT}^{*}_{\lambda}: \mu_i^{n-1} = \lambda^k_i + c_i \mbox{ for $i \in M$} \}$. Then from the proof of Lemma \ref{LemmaBound} we know that
\begin{equation}\label{fin1}
\left| \mathsf{GT}^{*}_{\lambda^k}({\bf c}) \right| \sim M_k \mbox{ and }\left| \mathsf{GT}^{*}_{\lambda^k}\right| \sim  M_k\prod_{i \in M} m_i \sim d_n(\lambda^k) = vol(GT_n(\lambda^k)),  \mbox{ as $k \rightarrow \infty$. }
\end{equation}

Let us denote $B(M) = \prod_{i \in M} \{0,...,m_i\}$ and $B^*(M) = \prod_{i \in M} \{0,...,m_i-1\}$. In view of (\ref{fin1}) and the boundedness of $f$, we know that 
\begin{equation*}
\lim_{k \rightarrow \infty} \mathbb{E}^{\lambda^k} \left[ f \circ g_k (\lambda^k, x^{n-1}, x^{n-2},...,x^1) \right] -  \frac{\sum_{{\bf c} \in B^*(M)}\sum_{ y \in \mathsf{GT}^{*}_{\lambda^k}({\bf c})}\int_{ Q(y)}  f \circ g_k(\lambda^k, x^{n-1},...,x^1) dx}{M_k \times \prod_{i \in M} m_i} = 0.
\end{equation*}
Moreover, using the uniform continuity of $f$ and the fact that $b(k) \rightarrow \infty$ as $k \rightarrow \infty$, we conclude
\begin{equation}\label{fin2}
\lim_{k \rightarrow \infty} \mathbb{E}^{\lambda^k} \left[ f \circ g_k (\lambda^k, x^{n-1}, x^{n-2},...,x^1) \right] -  \frac{\sum_{{\bf c} \in B^*(M)}\sum_{ y \in \mathsf{GT}^{*}_{\lambda^k}({\bf c})}f \circ g_k(y) }{M_k \times \prod_{i \in M} m_i} = 0.
\end{equation}

Given $c_i \in \{0,...,m_i\}$ for $i \in M$, we let $w({\bf c},k) = \rho_{\lambda^k}(\mu^1,...,\mu^{n-1})$, where $(\mu^1,...,\mu^{n-1},\mu^n) \in \mathsf{GT}^{*}_{\lambda^k}({\bf c})$.  By Lemma \ref{LEq}, we know that $w({\bf c},k)$ is well-defined. The boundedness of $f$, Lemma \ref{Lvol}, and Lemma \ref{LemmaBound} now imply that
\begin{equation}\label{fin3}
\lim_{k \rightarrow \infty}  \mathbb{E}^{\rho_{\lambda^k}} \left[ f \circ g_k (\lambda^k, \mu^{n-1}, \mu^{n-2},...,\mu^1) \right] - \frac{\sum_{{\bf c} \in B(M)}w({\bf c},k) \sum_{ y \in \mathsf{GT}^{*}_{\lambda^k}({\bf c})}f \circ g_k(y)}{M_k\times \sum_{{\bf c} \in B(M)} w({\bf c},k)} = 0.
\end{equation}

Let ${\bf c}^0$, be such that $c_i^0 = 0$ for $i \in M$. From the uniform continuity of $f$ and the fact that $b(k) \rightarrow \infty$, we note that for any ${\bf c}$ we have
$$\lim_{k \rightarrow \infty} \frac{\sum_{ y \in \mathsf{GT}^{*}_{\lambda^k}({\bf c}^0)} f \circ g_k(y) }{M_k}  - \frac{\sum_{ y \in \mathsf{GT}^{*}_{\lambda^k}({\bf c})} f \circ g_k(y) }{M_k} = 0.$$
The latter, together with the boundedness of $f$ and Lemma \ref{LemmaBound}, implies that
\begin{equation*}
\begin{split}
&\lim_{k \rightarrow \infty} \frac{\sum_{{\bf c} \in B^*(M)}\sum_{ y \in \mathsf{GT}^{*}_{\lambda^k}({\bf c})}f \circ g_k(y) }{M_k \times \prod_{i \in M} m_i} - \frac{1}{ \prod_{i \in M}  m_i} \sum_{{\bf c} \in B^*(M)}  \frac{\sum_{ y \in \mathsf{GT}^{*}_{\lambda^k}({\bf c}^0)} f \circ g_k(y) }{M_k} = 0 \mbox{, and } \\
&\lim_{k \rightarrow \infty} \frac{\sum_{{\bf c} \in B(M)}w({\bf c},k) \sum_{ y \in \mathsf{GT}^{*}_{\lambda^k}({\bf c})}f \circ g_k(y)}{M_k\times\sum\limits_{{\bf c} \in B(M)} w({\bf c},k)} -\frac{1}{\sum\limits_{{\bf c} \in B(M)} \hspace{-4mm} w({\bf c},k)} \hspace{-2mm}  \sum_{{\bf c} \in B(M)}w({\bf c},k)  \frac{\sum_{y \in \mathsf{GT}^{*}_{\lambda^k}({\bf c}^0)} f \circ g_k(y) }{M_k} = 0.
\end{split}
\end{equation*}
and so
\begin{equation}\label{fin4}
\lim_{k \rightarrow \infty} \frac{\sum_{{\bf c} \in B^*(M)}\sum_{ y \in \mathsf{GT}^{*}_{\lambda^k}({\bf c})}f \circ g_k(y) }{M_k \times \prod_{i \in M} m_i} - \frac{\sum_{{\bf c} \in B(M)}w({\bf c},k) \sum_{ y \in \mathsf{GT}^{*}_{\lambda^k}({\bf c})}f \circ g_k(y)}{M_k\times \sum_{{\bf c} \in B(M)} w({\bf c},k)} =  0.
\end{equation}
Combining (\ref{fin2}), (\ref{fin3}) and (\ref{fin4}) concludes the proof.
\end{proof}

\section{Proof of Theorem \ref{theorem2} }\label{Section7}
In this section we give the proof of Theorem \ref{theorem2}. We will split the proof into several steps and outline here the flow of the argument. We assume the same notation as in Section \ref{Section1.2} and define $g_M: \mathbb{R}^{k(k+1)/2} \rightarrow  \mathbb{R}^{k(k+1)/2}$ as
$$g_M(x) =  \frac{1}{c \sqrt{M}}\left( x - aM \cdot {\bf 1}_{\frac{k(k+1)}{2}} \right).$$
In addition, we replace $Y(N,M;k)$ with $Y(M)$ for brevity. The statement of Theorem \ref{theorem2} is that $g_M(Y(M))$ converge weakly to the GUE-corners process or rank $k$.\\

In the first step of the proof we show that we may replace the distribution of $Y(M)$ with the distribution $\nu_M$, given by the distribution of $Y(M)$ conditioned on $Y(M)_k^k \leq N$, without affecting the statement of the theorem. The latter is a consequence of Theorem \ref{theorem1}. The measures $\nu_M$ are probability measures on $\mathsf{GT}^{k+}$, which satisfy the six-vertex Gibbs property with certain weights. 

In the second step we check that the sequence of measures $\nu_M \circ g_M^{-1}$ on $\mathbb{R}^{k(k+1)/2}$ is tight. This is shown by using the six-vertex Gibbs property satisfied by $\nu_M$ and Lemma \ref{keyProbLemma}. The proof we present is similar to the proof of Proposition 7 in \cite{Gor14}.

In the third step we prove that $\nu_M \circ g_M^{-1}$ converge weakly to the GUE-corners process of rank $k$ by induction on $k$. The base case is proved via Lemma \ref{keyProbLemma}. When going from $k$ to $k+1$ we use the induction hypothesis and Proposition \ref{S5Prop} to show that any weak limit of $\nu_M \circ g_M^{-1}$ satisfies the continuous Gibbs property. The latter is combined with Proposition \ref{pGibbs} to prove the result for $k+1$.\\

{\raggedleft {\bf Step 1.}} Let $E_M$ be the event that $Y(M)_k^k \leq N$. It follows from Theorem \ref{theorem1} (see also (\ref{S4eqEv})) that $\mathbb{P}^{N,M}_{u,v}(E_M) \rightarrow 1$ as $M \rightarrow \infty$. Let $\nu_M$ be the distribution of $Y(M)$, conditioned on $E_M$. Since $\mathbb{P}^{N,M}_{u,v}(E_M) \rightarrow 1$ , we see that it suffices to prove that $\nu_M \circ g_M^{-1}$ converge weakly to the GUE-corners process of rank $k$.

We will show that $\nu_M$ is a probability distribution on $\mathsf{GT}^{k+}$, which satisfies the six-vertex Gibbs property with weights
\begin{equation}\label{SVW}
\begin{split}
w_1  = \frac{u - s^{-1}}{us-1}, \hspace{2mm}   w_2 = 1, \hspace{2mm} w_3=  \frac{u - s}{us-1},  \hspace{2mm} w_4 =\frac{us^{-1} - 1}{us-1} , \hspace{2mm} w_5 = \frac{(s^2-1)u}{us-1} , \hspace{2mm}w_6 =\frac{1 - s^{-2}}{us-1}.
\end{split}
\end{equation}
\vspace{-4mm}
\begin{figure}[h]
\centering
\scalebox{0.60}{\includegraphics{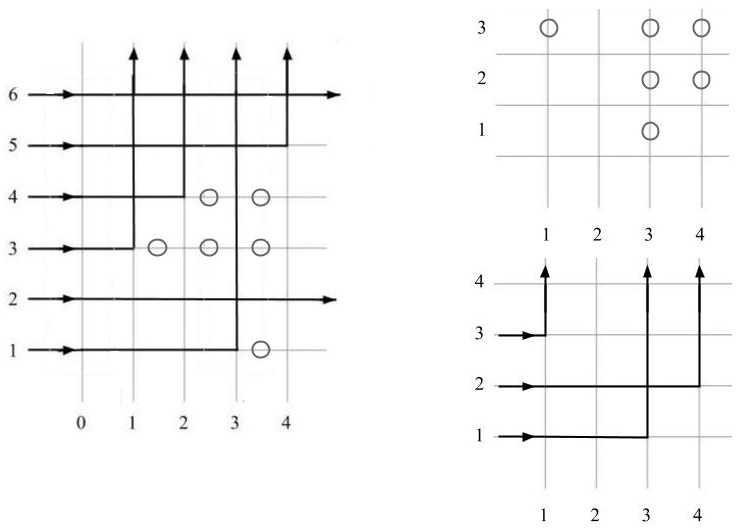}}
\caption{The left figure shows a path collection $\omega \in E_N$ with $N = 6$ and $k = 3$. Circles indicate the positions of the empty edges. The top right figure shows the array $(Y_i^j)_{1 \leq i \leq j \leq 3}$; $j$ varies vertically and position is measured horizontally. The bottom right figure shows the image of $(Y_i^j)_{1 \leq i \leq j \leq 3}$ under $h^{-1}$.}
\label{S8_1}
\end{figure}
Recall from Section \ref{Section1.2} that for $\omega \in E_N$, $(Y_i^j)_{1 \leq i \leq j \leq k}$ were the vertical positions of the empty horizontal edges in the first $k$ columns of $\omega$ (see the left part of Figure \ref{S8_1}). The condition $\omega \in E_M$ ensures that no $Y(M)_i^j$ are infinite, so that $\nu_M$ is a valid probability distribution on $\mathsf{GT}^{k+}$. 

The fact that $\nu_M$ satisfies a six-vertex Gibbs property is a consequence of the fact that $\mathbb{P}^{N,M}_{u,v}$ satisfies a Gibbs property for the six-vertex model on $D_N$ with weights  $(\tilde w_1, \tilde w_2, \tilde w_3, \tilde w_4, \tilde w_5, \tilde w_6)$ $= \left( 1 , \frac{u - s^{-1}}{us - 1}, \frac{us^{-1} - 1}{us - 1}, \frac{u - s}{us - 1}, \frac{u(s^2-1)}{us - 1}, \frac{1 - s^{-2}}{us - 1} \right)$ (see Section \ref{Section6.2}). We observe that there is a simple relationship between $\omega$ and $h^{-1}((Y_i^j)(\omega))$ (here $h$ is as in Section \ref{Section6.2}). Namely, $h^{-1}((Y_i^j)(\omega))$ is obtained by reflecting $\omega$ with respect to the line $x = y$ and then flipping filled and empty edges (see Figure \ref{S8_1}). This transformation has the followng effect on arrow configurations at a vertex
$$(0,0;0,0) \leftrightarrow (1,1;1,1) \mbox{ and } (1,0;1,0) \leftrightarrow (0,1;0,1),$$
while the vertices $(0,1;1,0)$ and $(1,0;0,1)$ are sent to themselves. This vertex transformation implies that the measure $h^{-1}((Y_i^j)(\omega))$ satisfies the Gibbs property for the six-vertex model on $D_k$ with weights $(\tilde w_2, \tilde w_1,\tilde w_4, \tilde w_3, \tilde w_5, \tilde w_6)$, which are the weights in (\ref{SVW}).\\

{\raggedleft {\bf Step 2.}} In this step we show that $\nu_M\circ g_M^{-1}$ is tight, which is equivalent to showing that $\eta(M)_i^j := \frac{Y(M)_i^j - aM}{c\sqrt{M}}$ is tight for each $i = 1,...,j$ and $j = 1,...,k$. We proceed by induction on $k$, with base case $k = 1$, true by Lemma \ref{keyProbLemma}.

Suppose the result is known for $k-1 \geq 1$ and we wish to show it for $k$. By induction hypothesis $\eta(M)_i^j$ is tight for each $i = 1,...,j$ and $j = 1,...,k-1$. In addition, by Lemma \ref{keyProbLemma} $\eta(M)_k^k$ is also tight. Using  the interlacing condition $Y(M)^{k-1}_{i-1} \leq Y(M)^k_i \leq Y(M)^{k-1}_{i}$ for $i = 2,...,k-1$, and the induction hypothesis we conclude that $\eta(M)_i^k$ is tight for $i = 2,...,k-2$. What remains to be seen is that $\eta(M)_1^k$ is tight.

We argue by contradiction and suppose that $\eta(M)_1^k$ is not tight as $M \rightarrow \infty$. Then we may find a positive number $p >0$, a subsequence $M_r$ and an increasing sequence $L_r$ going to $\infty$ such that
\begin{equation}\label{notTight}
\mathbb{P} \left( \left| \eta(M_r)_1^k \right|  > L_r \right)> p.
\end{equation}

Since $Y(M_r)_1^k \leq Y(M_r)_1^{k-1}$ and by induction hypothesis $\eta(M_r)_1^{k-1}$ is tight,  we see that if (\ref{notTight}) holds then we must have (by potentially passing to a further subsequence) that 
\begin{equation*}
\mathbb{P} \left( \eta(M_r)_1^k < - L_r \right)> p/2.
\end{equation*}

Let us denote by $B(M) = \min \left( Y(M)^{k}_2, Y(M)^{k-1}_{2} - 1, Y(M)^{k-2}_{1} \right)$ (if $k = 2$, $B(M) = Y(M)^{k}_2$). $B(M)$ is the rightmost positition that $Y^{k-1}_1(M)$ can take. From the tightness result established for $\eta(M)^{k}_2$, $\eta(M)^{k-1}_{2}$ and $\eta(M)^{k-2}_{1}$ we know that 
$$\lim_{r \rightarrow \infty} \mathbb{P} \left( \left| \frac{B(M_r) - aM_r}{c\sqrt{M_r}} \right|  < \sqrt{L_r}\right)   = 1.$$
Thus by further passing to a subsequence we know that 
\begin{equation}\label{notTight2}
\mathbb{P} \left(\eta(M_r)_1^k < - L_r;  \left| \frac{B(M_r) - aM_r}{c\sqrt{M_r}} \right|  < \sqrt{L_r} \right) > p/4.
\end{equation}

We know that $Y(M)_1^{k-1}$ is supported on $A(M), A(M)+1,...,B(M)$, where $A(M) =Y^{k}_1(M)$ and $B(M)$ is as above. Moreover, if 
$$p_i(M) = \mathbb{P}\left (Y^{k-1}_1(M) = i|Y^{k}_1(M),Y^{k}_2(M),  Y^{k-1}_2(M), Y^{k-2}_1(M) \right)\mbox{ for  $i = A(M),...,B(M)$, then we }$$
know by Lemma \ref{LemmaBound} that $c^{-1} >\frac{p_i(M)}{p_j(M)} > c$ for some constant $c \in (0,1)$ that depends on $k$. The latter implies that $p_i(M) \geq c^2/(B-A + 1)$ and so 
$$\mathbb{P}\left (Y^{k-1}_1(M)  \leq  \frac{A(M)+B(M)}{2}|Y^{k}_1(M),Y^{k}_2(M),  Y^{k-1}_2(M), Y^{k-2}_1(M) \right)\geq \frac{c^2}{2}.$$
 This together with (\ref{notTight2}) implies 
\begin{equation}\label{notTight3}
\mathbb{P} \left(\eta(M_r)_{1}^{k-1} < -  L_r/3\right) > (p/4)(c^2/2),
\end{equation}
which clearly contradicts the tightness of $\eta(M_r)_{1}^{k-1}$. The contradiction arose from our assumption that $\eta(M)_1^k$  is not tight as $M \rightarrow \infty$. This proves the induction step. By induction we conclude that $\nu_M\circ g_M^{-1}$ is tight for any $k \in \mathbb{N}$.\\

{\raggedleft {\bf Step 3.}} In this step we prove that $g_M(Y(M))$ converge to $\Lambda = \lambda_i^j$ $i = 1,...,j$, $j = 1,...,k$ as $M \rightarrow \infty$, where $\Lambda$ is the GUE-corners process of rank $k$. We proceed by induction on $k$, with base case $k = 1$ true by Lemma \ref{keyProbLemma}.

Suppose the result is known for $k-1 \geq 1$ and we wish to show it for $k$. From our earlier work we know that $\nu_M\circ g_M^{-1}$ is tight. Let $\mu$ be any subsequential limit and $\nu_{M_r}\circ g_{M_r}^{-1}$ converge weakly to $\mu$ for some sequence $M_r \rightarrow \infty$ as $r\rightarrow \infty$.

We observe that $\mu$ is a probability measure on $GT^{k} = \{ y \in \mathbb{R}^{k(k+1)/2}: y_{i }^{j+1} \leq y_{i}^j \leq y_{i+1}^{j+1}, \hspace{2mm} 1 \leq i \leq j \leq k-1\}$. We have by induction hypothesis that the restriction to $GT^{k-1}$ of $\mu$ is the GUE-corners process of rank $k-1$. In particular, we have
$$\mu ( y \in GT^{k}: y_{i}^{k-1} = y_{i+1}^{k-1} \mbox{ for some } i =1,...,k-2 ) = 0.$$
The above together with the interlacing property of elements in $GT^k$, shows that $\mu$ satisfies 
$$\mathbb{P}^{\mu}( y_i^n = y_{i+1}^n = y_{i+2}^n \mbox{ for some } i = 1,...,n-2) = 0.$$
From Step 1. in this proof we know that $\nu_{M_r}$ satisfy the six-vertex Gibbs property with weights as in (\ref{SVW}). We may thus apply Proposition \ref{S5Prop} to conclude that $\mu$ satisfies the continuous Gibbs property.

From Lemma \ref{keyProbLemma}, we know that under $\mu$ the distribution of $(y_1^1,...,y^k_k)$  is the same as $(\lambda_1^1,...,\lambda_k^k)$. This together with Proposition \ref{pGibbs} shows that $\mu$ is the  GUE-corners process of rank $k$. The above work shows that any subsequential limit of $\nu_M\circ g_M^{-1}$ has the same law as $\Lambda$. As $\nu_M\circ g_M^{-1}$ is tight we conclude it (and hence $g_M(Y(M))$) weakly converge to the GUE-corners process of rank $k$. The general result now follows by induction.

\newpage
\section{Exact sampling algorithm for $\mathbb{P}_{{\bf u}, {\bf v}} $}\label{Section8}

In this section, we describe an exact sampling algorithm for $\mathbb{P}_{{\bf u}, {\bf v}}$ (see Definition \ref{parameters}), which is based on discrete time dynamics on $\mathcal{P}_N$. We provide details on how this algorithm can be implemented efficiently and give some examples of typical path collections sampled from $\mathbb{P}_{{\bf u}, {\bf v}}$.

\subsection{Markov kernels and sequential update}
We start by recalling some notation from Section 6.2 in \cite{BP}. 
For any $n$ we define
\begin{equation}\label{L}
\Lambda^-_{u|{\bf u}}(\nu \rightarrow \mu):= \frac{\mathsf{F}_{\mu}(u_1,...,u_n)}{\mathsf{F}_{\nu}(u_1,...,u_n,u)} \mathsf{F}_{\nu / \mu}(u),
\end{equation}
where ${\bf u} = (u_1,...,u_n)$, and $\nu \in \mathsf{Sign}_{n+1}^+$, $\mu \in \mathsf{Sign}^+_n$. Let us also define 
\begin{equation}\label{Q}
Q^{\circ}_{{\bf u};v}(\mu \rightarrow \nu):= \left( \prod_{i = 1}^m \frac{1 - u_i v}{1 - qu_i v} \right) \frac{\mathsf{F}_{\nu}(u_1,...,u_n)}{\mathsf{F}_{\mu}(u_1,...,u_n)} \mathsf{G}^c_{\nu / \mu}(v),
\end{equation}
where $u_i$ and $v$ are admissible with respect to $s = q^{-1/2}$ for all $i$ and $\mu,\nu \in \mathsf{Sign}^+_n$. It follows from Propositions \ref{Branching} and \ref{PPieri} that 
$\Lambda^-_{u|{\bf u}}: \mathsf{Sign}_{m+1}^+ \dashrightarrow \mathsf{Sign}^+_m$ and $Q^{\circ}_{{\bf u};v}: \mathsf{Sign}^+_m \dashrightarrow \mathsf{Sign}^+_m$ define Markov kernels.\footnote{We use the notation ``$\dashrightarrow$'' to indicate that $\Lambda^-_{u|{\bf u}}$ and $Q^{\circ}_{{\bf u};v}$ are Markov kernels, i.e., they are functions in the first variable (belonging to the space on the left of ``$\dashrightarrow$')' and probability distributions in the second variable (belonging to the space on the right of ``$\dashrightarrow$'').} 

For $\omega \in \mathcal{P}_N$, we let $\lambda^n = \lambda^n(\omega)$ for $n = 1,...,N$ be as in Section \ref{Section1.2}, we also let $\lambda^0(\omega) = \varnothing$. Let $\mathbb{P}_{{\bf u}, {\bf v}}^n$ be the projection of $\mathbb{P}_{{\bf u}, {\bf v}}$ on $\lambda^n$. As direct consquences of Proposition \ref{PPieri}, we have 
\begin{equation}\label{S61}
\mathbb{P}^n_{{\bf u} \cup u, {\bf v}} \Lambda^-_{u| {\bf u}} = \mathbb{P}^n_{{\bf u}, {\bf v}}, \hspace{5mm} \mathbb{P}^n_{{\bf u}, {\bf v}} Q^\circ_{{\bf u};v} = \mathbb{P}^n_{{\bf u}, {\bf v} \cup v} \mbox{, and    }  \hspace{3mm}  Q^\circ_{{\bf u} \cup u ;v} \Lambda^-_{u| {\bf u}}  =  \Lambda^-_{u| {\bf u}}    Q^\circ_{{\bf u} ;v},
\end{equation}
where for a variable set ${\bf w} = (w_1,...,w_k)$ we write ${\bf w} \cup w = (w_1,...,w_k, w)$. \\

Our next goal is to define a stochastic dynamics on $\mathcal{P}_N$. The construction we use is parallel to those of \cite{BF} (see also \cite{Bor11, BorCor,BorGor1}) and it is based on an idea going back to \cite{DF}, which allows to couple the dynamics on signatures of different sizes. 

Suppose that $\omega$ is distributed accoriding to $\mathbb{P}_{{\bf u}, {\bf v}}$ as in Definition \ref{parameters}, and let $v$ be such that $0 < v$ and $u_i v < 1$ for all $i$. We consider a random $\mu^1 \preceq \mu^2 \preceq \cdots \preceq \mu^N$, with $\mu^i \in \mathsf{Sign}_i^+$, whose distribution depends on $\lambda^n(\omega)$ for $n = 1,..., N$ and the parameters ${\bf u}$, $v$, and is defined through the following {\em sequential update} rule. 

We start with $\mu^1$ and sample it according to the distribution
\begin{equation}\label{seqUp0}
\mathbb{P}( \mu^1 = \nu | \lambda^1 = \lambda) = \frac{\mathsf{F}_{\nu}(u_1) \mathsf{G}^c_{\nu / \lambda}(v)} { \sum_{\kappa \in \mathsf{Sign}^+_1 }\mathsf{F}_{\kappa}(u_1) \mathsf{G}^c_{\kappa / \lambda}(v)}.
\end{equation}
If $\mu^1,...,\mu^{k-1}$ are sampled, we sample $\mu^k$ for $k \geq 2$ according to
\begin{equation}\label{seqUp1}
\mathbb{P}( \mu^k = \nu | \lambda^k = \lambda, \mu^{k-1} = \mu) =   \frac{\mathsf{F}_{\nu/ \mu}(u_k) \mathsf{G}^c_{\nu / \lambda}(v)} { \sum_{\kappa \in \mathsf{Sign}^+_k }\mathsf{F}_{\kappa/ \mu}(u_k) \mathsf{G}^c_{\kappa / \lambda}(v)}.
\end{equation}
We now let $\omega'$ be the resulting element in $\mathcal{P}_N$, i.e., $\lambda^n(\omega') = \mu^n$ for $n = 1,...,N$. The key observation is that if $\omega$ is distributed accoridng to $\mathbb{P}_{{\bf u}, {\bf v}}$, then $\omega'$ is distributed according to $\mathbb{P}_{{\bf u}, {\bf v}\cup v}$. The latter is a consequence of (\ref{S61}) and a Gibbs property satisfied by $\mathbb{P}_{{\bf u}, {\bf v}}$, which states that conditioned on $\lambda^k$, the distribution of $\lambda^1,...,\lambda^{k-1}$ is independent of ${\bf v}$ and is given by
\begin{equation}
\Lambda^-_{u_k|(u_1,...,u_{k-1})}(\lambda^k \rightarrow \lambda^{k-1}) \cdots \Lambda^-_{u_3|(u_1,u_2)}(\lambda^3 \rightarrow \lambda^{2})\Lambda^-_{u_2|(u_1)}(\lambda^2 \rightarrow \lambda^{1}).
\end{equation}
For a more detailed description of the above procedure in analogous contexts we refer the reader to Section 2 of \cite{BF} and Section 2 of \cite{BorCor}.\\

For $m = 0,...,M$ we let $\mathbb{P}_{{\bf u}, {\bf v}_m}$, denote the probability distribution as in Definition \ref{parameters} with ${\bf v}_m = (v_1,...,v_m)$. Equations (\ref{seqUp0}) and (\ref{seqUp1}) provide a mechanism for sampling $\omega'$ distributed according to $\mathbb{P}_{{\bf u}, {\bf v}_{m+1}}$, given $\omega$ distributed as $\mathbb{P}_{{\bf u}, {\bf v}_{m}}$. Our strategy to sample $ \mathbb{P}_{{\bf u}, {\bf v}} = \mathbb{P}_{{\bf u}, {\bf v}_M}$ is to first sample $\mathbb{P}_{{\bf u}, {\bf v}_0}$ and then use the above mechanism to sequentially sample $\mathbb{P}_{{\bf u}, {\bf v}_{m+1}}$ for $m = 0,...,M-1$. 

We now turn to an algorithmic description of the above strategy. We assume we have the following samplers, which will be described in the following section. For $N \geq 1$, $q \in (0,1)$ and ${\bf u} = (u_1,...,u_N)$ such that $u_i > q^{-1/2}$ for $i = 1,...,N$, we let {\tt ZeroSampler}$(N, q, {\bf u})$ produce a random element $\omega \in \mathcal{P}_N$, distributed according to $\mathbb{P}_{{\bf u}, {\bf v}_0}$. For $k \in \{1,...,N\}$, $v > 0$ such that $u_i v < 1$ for all $i$, $\lambda \in \mathsf{Sign}^+_k$ and $\mu \in \mathsf{Sign}_{k-1}^+$, we let {\tt RowSampler}$ (k,q, {\bf u}, v, \lambda, \mu)$ produce a random signature $\mu^k \in \mathsf{Sign}_k^+$, distributed according to (\ref{seqUp1}). With this notation we have the following exact sampler for $\mathbb{P}_{{\bf u}, {\bf v}}$. 

\vspace{3mm}
\begin{tabular}{l}
\hline
{\bf Algorithm} {\tt SixVerexSampler}$(N, M,q, {\bf u}, {\bf v})$ \\
\hline
 {\bf Input:} $q \in (0,1)$, ${\bf u} = (u_1,...,u_N)$ and ${\bf v} = (v_1,...,v_M)$ - parameters of the distribution.\\
\hspace{5mm}\\
$\omega$ := {\tt ZeroSampler}$(N, q, {\bf u})$;\\
{\tt initialize} $\mu^i$ for $i = 0,...,N$;\\
$\mu^ 0 = \varnothing$; \\
{\bf for} ($i = 1$, $i \leq M$, $i = i + 1$) {\bf do}\\
\hspace{5mm} {\bf for} ($k = 1$, $k \leq N$, $k = k + 1$) {\bf do}\\
\hspace{12mm}  $\mu^{k}$ = {\tt RowSampler}$ (k,q, u_k, v_i, \lambda^k(\omega), \mu^{k-1})$; \\
\hspace{5mm} {\bf end}\\
\hspace{5mm}$\omega$ = $(\mu^1 \preceq \mu^2 \preceq \cdots \preceq \mu^N)$;\\
{\bf end}\\
{\bf Output:} $\omega$.\\
\hline
\end{tabular}

\vspace{5mm}

\subsection{ The algorithms  {\tt ZeroSampler} and {\tt RowSampler}}
From the definition of $\mathbb{P}_{{\bf u}, {\bf v}_0}$, we know that it agrees with the distribution of the vertically inhomogeneous stochastic six vertex model of Section 6.5 in \cite{BP}, except that all columns are shifted by $1$ to the right so that all vertices in the $0$-th column are of the form $(0,1;0,1)$. The vertically inhomogeneous six vertex model has a known sampling procedure, which we now describe - see Section 6.5 \cite{BP} and \cite{BCG14} for details. 

For $u > q^{-1/2}$ and $q \in (0,1)$, we let 
$$b_1(u) = \frac{1 - u q^{1/2} }{1 - u q^{-1/2}} \mbox{ and }b_2(u) = \frac{ -uq^{-1/2}+ q^{-1}}{ 1 - uq^{-1/2}}.$$
Notice that $b_1(u), b_2(u) \in (0,1)$. We construct a random element $\omega \in \mathcal{P}_N$ by choosing the types of vertices sequentially: we start from the corner vertex at $(1,1)$, then proceed to $(1,2)$ and $(2,1),...$, then proceed to all vertices $(x,y)$ with $x + y = k$, then with $x + y = k + 1$ and so forth. The combinatorics of the model implies that when we choose the type of the vertex $(x,y)$, either it is uniquely determined by the types of its previously chosen neighbors, or we need to choose between vertices of type $(1,0;1,0)$ and $(1,0;0,1)$, or we need to choose between vertices of type $(0,1; 0,1)$ and $(0,1; 1,0)$. We do all choiced independently and choose type $(1,0;1,0)$ with probability $b_1(u_y)$ and type $(1,0;0,1)$ with probability $1- b_1(u_y)$. Similarly, we choose type $(0,1; 0,1)$ with probability $b_2(u_y)$ and $(0,1; 1,0)$ with probability $1 - b_2(u_y)$. We denote by {\tt Bernoulli}$(p)$ a Bernoulli random variable sampler with parameter $p \in (0,1)$. For a vertex $\alpha = (i_1, j_1; i_2, j_2)$, we let {\tt I2}$(\alpha)$ = $i_2$ and {\tt J2}$(\alpha)$ = $j_2$. With this notation we have the following algorithm for {\tt ZeroSampler}.

\vspace{3mm}
\begin{tabular}{l}
\hline
{\bf Algorithm} {\tt ZeroSampler}$(N, q, {\bf u})$ \\
\hline
 {\bf Input:} $q \in (0,1)$, ${\bf u} = (u_1,...,u_N)$ - parameters of the distribution.\\
\hspace{5mm}\\
{\tt initialize} $\omega$;\\
$c$ := $0$;\\
$k$ := $2$;\\
{\bf while} ($c < N$) {\bf do} \\
\hspace{5mm} {\bf for} ($x$ = $1$, $x< k$, $x$ = $x+1$) {\bf do }\\
\hspace{12mm} $y$ = $k - x$;\\
\hspace{12mm}  {\bf  if } ($y > N$) do nothing \\ 
\hspace{12mm}  {\bf else if } ($x$ == $1$ and $y$ == $1$) \\ 
\hspace{19mm}  {\bf if } ({\tt Bernoulli}$(b_2(u_y))$ == $1$) $\omega(x,y)$ = $(0,1;0,1)$;\\
\hspace{19mm}  {\bf else} $\omega(x,y)$ = $(0,1;1,0)$;\\
\hspace{19mm}  {\bf end}\\
\hspace{12mm} {\bf else if } ($x$ == $1$) \\
\hspace{19mm}  {\bf if } ({\tt I2}$(\omega(x,y-1)) == 1$) $\omega(x,y)$ = $(1,1;1,1)$;\\
\hspace{19mm}  {\bf else if} ({\tt Bernoulli}$(b_2(u_y))$ == $1$) $\omega(x,y)$ = $(0,1;0,1)$;\\
\hspace{19mm}  {\bf else} $\omega(x,y)$ = $(0,1;1,0)$;\\
\hspace{19mm}  {\bf end}\\
\hspace{12mm} {\bf else if} ($y$ == $1$) \\
\hspace{19mm}  {\bf if } ({\tt J2}$(\omega(x-1,y))$ == $0$) $\omega(x,y)$ = $(0,0;0,0)$;\\
\hspace{19mm}  {\bf else if} ({\tt Bernoulli}$(b_2(u_y))$ == $1$) $\omega(x,y)$ = $(0,1;0,1)$;\\
\hspace{19mm}  {\bf else} $\omega(x,y)$ = $(0,1;1,0)$;\\
\hspace{19mm}  {\bf end}\\
\hspace{12mm} {\bf else } \\
\hspace{19mm}  {\bf if } ({\tt I2}$(\omega(x,y-1)) == 0$ and {\tt J2}$(\omega(x-1,y))$ == $0$)  $\omega(x,y)$ = $(0,0;0,0)$;\\
\hspace{19mm}  {\bf else if } ({\tt I2}$(\omega(x,y-1)) == 1$ and {\tt J2}$(\omega(x-1,y))$ == $1$)  $\omega(x,y)$ = $(1,1;1,1)$;\\
\hspace{19mm}  {\bf else if } ({\tt I2}$(\omega(x,y-1)) == 0$ and {\tt J2}$(\omega(x-1,y))$ == $1$) \\
\hspace{26mm}  {\bf if } ({\tt Bernoulli}$(b_2(u_y))$ == $1$) $\omega(x,y)$ = $(0,1;0,1)$;\\
\hspace{26mm}  {\bf else} $\omega(x,y)$ = $(0,1;1,0)$;\\
\hspace{26mm}  {\bf end}\\
\hspace{19mm}  {\bf else  } \\
\hspace{26mm}  {\bf if} ({\tt Bernoulli}$(b_1(u_y))$ == $1$) $\omega(x,y)$ = $(1,0;1,0)$;\\
\hspace{26mm}  {\bf else} $\omega(x,y)$ = $(1,0;0,1)$;\\
\hspace{26mm}  {\bf end}\\
\hspace{19mm}  {\bf end}\\
\hspace{12mm} {\bf end } \\
\hspace{12mm} {\bf if } ($y$ == $N$ and {\tt I2}($\omega(x,y)$) == $1$) $c$ = $c + 1$;\\
\hspace{5mm} {\bf end } \\
\hspace{5mm}$k$ = $k + 1$;\\
{\bf end}\\
{\tt initialize} $\mu^i$ for $i = 1,...,N$;\\
{\bf for} ($i = 1$, $i \leq M$, $i = i + 1$) {\bf do}\\
\hspace{5mm} $\mu^i$ = $\lambda^i(\omega)$ + $1^i$;\\
{\bf end}\\
{\bf Output:} $(\mu^1 \preceq \mu^2 \preceq \cdots \preceq \mu^N)$.\\
\hline
\end{tabular}

\vspace{5mm}

Let us fix $k \geq 1$, parameters $q,u,v$ such that $q \in (0,1)$, $u > q^{-1/2}$, $v > 0$ and $uv < 1$. We also fix $\mu \in \mathsf{Sign}^+_{k-1}$ and $\lambda \in \mathsf{Sign}^+_k$. We now discuss how to sample the distribution from (\ref{seqUp1}) with the above parameters, which we denote by $\mathbb{P}$ for brevity. Let us define the numbers $a_i  = \max(\mu_{i}, \lambda_i, 0)$ and $b_i = \min( \mu_{i-1}, \lambda_{i-1})$, where we agree that $\mu_k = -\infty$, $\lambda_0 = \mu_0 = \infty$. We also set $A({\bf a}, {\bf b}) = \cup_{i = 1}^k [a_i, b_i]$. The definition of $\mathbb{P}$ implies that $\mathbb{P} ( \{ \nu \in \mathsf{Sign}_k^+ : \nu_i \in [a_i,b_i] \mbox{ for } i = 1,...,k \}) = 1$. Moreover, if $\nu \in \mathsf{Sign}_k^+ $ is such that $ \nu_i \in [a_i,b_i]$ for $i = 1,...,k$ then $\mathcal{P}^c_{\nu/\lambda}$ and $\mathcal{P}_{\nu/\mu}$ (see Definitions \ref{defG} and \ref{defF}) consist of single elements $\omega$ and $\omega^c$, which implies that 
$$\mathbb{P}( \{\nu \}) \propto  \prod_{j \in A({\bf a}, {\bf b}) \cap [0,\nu_1]} w_u(\omega(j,1)) w^c_v(\omega^c(j,1)), \mbox{ where $w_u$ and $w_u^c$ are as in (\ref{weights1}) and (\ref{weights2}).}$$

Sampling $\mathbb{P}$ is rather hard because there are infinite possible signautes $\nu$ that are allowed. Even if we consider only signatures, whose parts are bounded by some large constant $L$, their number is still exponentially large in $L$ and we cannot hope to efficiently enumerate possible cases and calculate their weight.

The key observation that allows one to sample this distribution is that if $l \in \{1,...,k\}$, then conditional on $\nu_l$, the distributions of $\nu_1,...,\nu_{l-1}$ and $\nu_{l+1},...,\nu_k$ are independent and similar to the one of $\nu_1,...,\nu_k$. Let us make the last statement more precise. Fix an integer $l \in \{1,...,k\}$, suppose we have fixed $\nu_{l} = x \in [a_{l}, b_{l}]$ and that there is at least one possible signature $\nu_1,...,\nu_k$ with $\nu_l = x$. We modify $a_i$ and $b_i$ as follows
$$b_{i}^x= \begin{cases} b_{i} &\mbox{ if } b_{i} \neq x \\ b_{i} - 1 &\mbox{ else,} \end{cases} \hspace{10mm}  a_{i}^x = \begin{cases} a_{i} &\mbox{ if } a_{i} \neq x \\ a_{i} + 1 &\mbox{ else.} \end{cases}$$
Let us fix $y_i \in [a^x_i, b^x_i]$ for $i \neq l$, put $y_l = x$ and denote $A_R =\cup^{l-1}_{i =1}[a_i^x, b_i^x]$ and $A_L = \cup^{k}_{i =l+1}[a_i^x, b_i^x]$. Then we have 
$$\mathbb{P}( \{ \nu  \in \mathsf{Sign}_k^+ : \nu_i = y_i \mbox{ for } i \neq l | \nu_l = x )  = \mathbb{P}( \{ \nu  \in \mathsf{Sign}_{k}^+ : \nu_i = y_i \mbox{ for } i = 1,...,l-1 | \nu_l = x)\times  $$
$$\mathbb{P}( \{ \nu  \in \mathsf{Sign}_{k-l}^+ : \nu_i = y_{i} \mbox{ for } i = l+1,...,k | \nu_l = x).$$
Moreover, we have
$$ \mathbb{P}( \{ \nu  \in \mathsf{Sign}_{k}^+ : \nu_i = y_i \mbox{ for } i = 1,...,l-1 \}| \nu_l = x )\propto \prod_{j \in A_R \cap [0,y_1]} w_u(\omega(j,1)) w^c_v(\omega^c(j,1)),$$
$$\mathbb{P}( \{ \nu  \in \mathsf{Sign}_{k-l}^+ : \nu_i = y_{i} \mbox{ for } i = l+1,...,k \} | \nu_l = x) \propto \prod_{j \in A_L \cap [0,y_1]} w_u(\omega(j,1)) w^c_v(\omega^c(j,1)).$$

The above arguments imply that we can sample $\mathbb{P}$, by first sampling $\nu_{\lfloor k/2 \rfloor}$, conditioning on its value and recursively sampling $\nu_1,...,\nu_{\lfloor k/2 \rfloor - 1}$ and $\nu_{\lfloor k/2 \rfloor +1},..., \nu_k$. The recursion reduces runtime from exponentially large to polynomial in $N$.\\

We begin by explaining how to sample $\nu_l$ for $l \in \{1,...,k\}$. Suppose that $I = \{i_1,...,i_r\} \subset \{1,...,k\}$, $[c_i,d_i] \subset [a_i, b_i]$ for $i \in I$ and $x_i \in [c_i,d_i]$. We define
$$\mathcal{W}({\bf c}, {\bf d }, I,  {\bf x}):= \prod_{i \in I} \prod_{j  = c_i}^{x_i}  w_u(\omega(j,1)) w^c_v(\omega^c(j,1)),$$
where $\omega^c$ and $\omega$ are the single elements of $\mathcal{P}^c_{\nu/\lambda}$ and $\mathcal{P}_{\nu/\mu}$, where $\nu_i = x_i$ for $i \in I$ and $\nu_{j} < c_i$ if $j > i$, $i \in I$ and $j \not \in I$ (if no such $\nu \in \mathsf{Sign}_k^+$ exists $\mathcal{W}({\bf c}, {\bf d }, I,  {\bf x}) = 0$). We also define 
$$\mathcal{W}({\bf c}, {\bf d};I) := \sum_{x_{i_1} = c_{i_1}}^{d_{i_1}} \cdots   \sum_{x_{i_r} = c_{i_r}}^{d_{i_r}} \mathcal{W}({\bf c}, {\bf d }, I,  {\bf x}).$$

Let us denote by $p = w_u((0,1;0,1)) w^c_v((0,1;0,1))$, and suppose $[c_l,d_l] \subset [a_l,b_l]$. We wish to sample $\nu_l$ conditioned on it belonging to $[c_l,d_l]$ according to $\mathbb{P}$. We define
\begin{equation}
\begin{split}
wL1 &= \mathcal{W}({\bf c}, {\bf d};\{1,...,l-1\}) \mbox{, where $c_i = a_i$ and $d_i = b_i$;}\\
wL2 &= \mathcal{W}({\bf c}, {\bf d};\{1,...,l-1\}) \mbox{, where $c_i = a_i$ and $d_i = b_i$, except $c_{l-1} = a_{l-1} + 1$};\\
wR1 &= \mathcal{W}({\bf c}, {\bf d};\{l+1,...,k\}) \mbox{, where $c_i = a_i$ and $d_i = b_i$;}\\
wR2 &= \mathcal{W}({\bf c}, {\bf d};\{l+1,...,k\}), \mbox{ where $c_i = a_i$ and $d_i = b_i$, except $d_{l+1} = d_{l+1} - 1$}.\\
\end{split}
\end{equation}
\begin{equation}\label{s6bs}
\begin{split}
bLL &= 1 \mbox{ if } c_l = \lambda_l \mbox{ and } bLL = 0 \mbox{ otherwise};\\
bRL &= 1 \mbox{ if } d_l = \lambda_{l-1} \mbox{ and } bRL = 0 \mbox{ otherwise};\\
bLM &= 1 \mbox{ if } c_l = \mu_l \mbox{ and } bLM = 0 \mbox{ otherwise};\\
bRM &= 1 \mbox{ if } d_l = \mu_{l-1} \mbox{ and } bRM = 0 \mbox{ otherwise};\\
\end{split}
\end{equation}
The conditional distribution of $\nu_l$ depends on $c_l, d_l$ and the above four variables $bLL, bRL, bLM$ and $bRM$. 

If $c_l = d_l$ then we have $\nu_l = c_l$ with probability $1$. If $d_l = c_l + 1$, then we have sixteen possible cases for the variables $bLL, bRL, bLM$ and $bRM$, which lead to different probability distributions. To give one example, if $bLL = bRL = bLM = bRM = 0$, then
\begin{equation*}
\mathbb{P}( \nu_l = c_l) \propto wL2\cdot wR1, \hspace{5mm} \mathbb{P}( \nu_l = d_l) \propto    p\cdot wL1 \cdot wR2.
\end{equation*}
If $d_l - c_l = n \geq 2$, then we have sixteen possible cases for the variables  $bLL, bRL, bLM$ and $bRM$, which lead to different probability distributions. To give one example, if $bLL = bRL = bLM = bRM = 0$, then
\begin{equation*}
\begin{split}
&\mathbb{P}( \nu_l = c_l) \propto w_u(0,1;1,0) w_u(0,1;1,0) \cdot wL2\cdot wR1, \\
& \mathbb{P}( \nu_l = d_l) \propto w_u(0,1;1,0) w_v^c(0,1;1,0) w_u(0,1;0,1) w_v^c(0,1;0,1) p^{n-1} \cdot wL1 \cdot wR2 ,\\
& \mathbb{P}( \nu_l = c_l + i) \propto w_u(0,1;1,0) w_v^c(0,1;1,0) w_u(0,1;0,1) w_v^c(0,1;0,1) p^{i-1} \cdot wL1 \cdot wR1 \mbox{, $1 \leq i \leq k-1$}.\\
\end{split}
\end{equation*}
 There are altogether thirty-three cases (sixteen corresponding to $d_l = c_l + 1$, sixteen for $d_l > c_l + 1$ and the trivial case of $c_l = d_l$) and we will not write them out explicitly. The important point is that the conditional distribution of $\nu_l$, given $wL1, wL2, wR1$ and $wR2$ is explicit and can be sampled. We let ${\tt ArrowSampler}(u,v,c,d, bLL,bRL, bLM, bRM, wL1,wL2,wR1,wR2)$ denote an algorithm that samples the above probability distribution when $c_l = c$, $d_l = d$. 

Suppose that we have an algorithm ${\tt Weight}(u,v,{\bf c}, {\bf d}, x,y, \lambda, \mu) := \mathcal{W}({\bf c}, {\bf d};\{x,x + 1...,y\})$, then we have the following algorithm for ${\tt RowSampler}$. 

\vspace{3mm}
\begin{tabular}{l}
\hline
{\bf Algorithm} {\tt RowSampler}$(u,v, {\bf c}, {\bf d}, x,y, \lambda, \mu)$ \\
\hline
 {\bf Input:} $u,v$ - parameters, ${\bf c} = (c_{x}, c_{x+1}, ..., c_{y})$, ${\bf d} = (d_x,d_{x+1},...,d_y)$, $\lambda \in \mathsf{Sign_k}^+$, $\mu \in \mathsf{Sign}_{k-1}^+$.\\
\hspace{5mm}\\
{\bf if }($x$ == $y$) \\
\hspace{5mm} $bLL$ := $0$; {\bf if } ($c_x$ == $\lambda_x$) $bLL$ = $1$; {\bf end}\\
\hspace{5mm} $bRL$ := $0$; {\bf if } ($d_x$ == $\lambda_{x-1}$) $bRL$ = $1$; {\bf end}\\
\hspace{5mm} $bLM$ := $0$; {\bf if } ($c_x$ == $\mu_x$) $bLM$ = $1$; {\bf end}\\
\hspace{5mm} $bRM$ := $0$; {\bf if } ($d_x$ == $\mu_{x-1}$) $bRM$ = $1$; {\bf end }\\
\hspace{5mm} $\nu_x$ = {\tt ArrowSampler}($u,v,c_x,d_x, bLL, bRL, bLM, bRM,1,1,1,1$);\\
{\bf else }  \\
\end{tabular}
\newpage

\begin{tabular}{l}
\hspace{5mm} $s$ := $\lfloor (x+y) /2 \rfloor$; \\
\hspace{5mm} ${\bf c}'$ := $\{c_x,...,c_{s-1}\}$;\hspace{2mm} ${\bf c}''$ := $\{ c_{s+1},..., c_y\}$;\hspace{2mm} ${\bf d}'$ := $\{d_x,...,d_{s-1}\}$;\hspace{2mm} ${\bf d}''$ := $\{ d_{s+1},..., d_y\}$;\\
\hspace{5mm} $wR1$ := {\tt Weight}($u,v,{\bf c}', {\bf d}', x, s-1, \lambda, \mu$);\hspace{2mm} $wR2$ := $wR1$;\\
\hspace{5mm} {\bf if } ($d_{s}$ == $c_{s-1}$) \\
\hspace{12mm} ${\bf c}'$ = $\{ c_x,...,c_{s-2}, c_{s-1} +1\}$; \hspace{2mm} $wR2$ = {\tt Weight}($u,v,{\bf c}', {\bf d}', x, s-1, \lambda ,\mu$);\\
\hspace{5mm} {\bf end}\\
\hspace{5mm} $wL1$ := {\tt Weight}($u,v,{\bf c}'', {\bf d}'', s+1, y$);\hspace{2mm} $wL2$ := $wL1$;\\
\hspace{5mm} {\bf if } ($c_{s}$ == $d_{s+1}$) \\
\hspace{12mm} ${\bf d}'' = \{ d_{s+1}- 1,d_{s+2},...,d_y\}$; \hspace{2mm} $wL2$ = {\tt Weight}($u,v,{\bf c}'', {\bf d}'', s+1, y, \lambda, \mu$);\\
\hspace{5mm} {\bf end}\\
\hspace{5mm} $bLL$ := $0$; {\bf if } ($c_s$ == $\lambda_s$) $bLL$ = $1$; {\bf end}\\
\hspace{5mm} $bRL$ := $0$; {\bf if } ($d_s$ == $\lambda_{s-1}$) $bRL$ = 1; {\bf end}\\
\hspace{5mm} $bLM$ := $0$; {\bf if } ($c_s$ == $\mu_s$) $bLM$ = $1$; {\bf end}\\
\hspace{5mm} $bRM$ := $0$; {\bf if } ($d_s$ == $\mu_{s-1}$) $bRM$ = $1$; {\bf end}\\
\hspace{5mm} $\nu_s$ = {\tt ArrowSampler}($u,v,c_s,d_s, bLL, bRL, bLM, bRM,wL1,wL2,wR1,wR2$);\\
\hspace{5mm} ${\bf c}'$ = $\{c_x,...,c_{s-1}\}$;\hspace{2mm} ${\bf c}''$ = $\{ c_{s+1},..., c_y\}$;\hspace{2mm} ${\bf d}'$ = $\{d_x,...,d_{s-1}\}$;\hspace{2mm} ${\bf d}''$ = $\{ d_{s+1},..., d_y\}$;\\
\hspace{5mm} {\bf if } ($\nu_s$ == $d_{s+1}$) ${\bf d}'' = \{ d_{s+1}- 1,d_{s+2},...,d_y\}$; {\bf end}\\
\hspace{5mm} {\bf if } ($\nu_s$ == $c_{s-1}$) ${\bf c}'$ = $\{ c_x,...,c_{s-2}, c_{s-1} +1\}$;  {\bf end}\\
\hspace{5mm}  {\tt RowSampler}$(u,v, {\bf c}', {\bf d}', x,s-1, \lambda, \mu)$; \\ 
\hspace{5mm}  {\tt RowSampler}$(u,v, {\bf c}'', {\bf d}'', s+1,y, \lambda, \mu)$; \\ 
{\bf Output:} $\nu$ \\
\hline
\end{tabular}

\vspace{3mm}

In the algorithm ${\tt RowSampler}$ $\nu = \nu_1 \geq \nu_2 \geq \cdots \geq \nu_k$ is a global variable that we are updating through the recursive calls to the same algorithm. Going back to the notation of {\tt SixVerexSampler}, we have that  {\tt RowSampler}$ (k,q, u, v, \lambda, \mu)$ = {\tt RowSampler}$(u, v, {\bf a}, {\bf b}, 1,k, \lambda, \mu)$. Thus what remains is to show how to calculate  {\tt Weight}($u,v,{\bf c}, {\bf d}, x, y$).\\

The function ${\tt Weight}$ can be calculated recursively by again conditioning on the middle arrow and summing over the weights corresponding to its possible positions. We first discuss the base case of having a single interval $[c_l,d_l]$. We will consider a reweighted version of $\mathcal{W}({\bf c}, {\bf d};\{l\})$, where we have additional four weights $wL1,wL2, wR1$ and $wR2$, which are fixed. By definition $\mathcal{W}({\bf c}, {\bf d};\{l\})$ is the sum of weights over the possible positions of the arrow in $[c_l,d_l]$. Our reweighed version will be the same sum, however we will multiply each term by $wL1 \cdot wR1$, $wL1\cdot wR2$, $wL2 \cdot wR1$ or $wL2 \cdot wR2$ according to the following rules.

We multiply the weight by $wL1$ unless the arrow is in location $c_l$, in which case we multiply it by $wL2$, we then multiply the weight by $wR1$ unless the arrow is in location $d_l$, in which case we multiply it by $wR2$.
One observes that the weight of the interval $[c_l,d_l]$, depends on $d_l - c_l$. We have three cases for $d_l - c_l$ - when it is $0$, $1$ and $\geq 2$, and the weights are as follows.
\begin{itemize}
\item $ d_l - c_l = 0: \mathcal{W}({\bf c}, {\bf d};\{l\}) = w_u(1,1;1,1)\cdot w_v^c(1,0;1,0) \cdot wL2 \cdot wR2; $
\item $d_l - c_l = 1: \mathcal{W}({\bf c}, {\bf d};\{l\}) = w_u(0,1;1,0)\cdot w_u(1,0;0,1)\cdot w_v^c(1,0;1,0) \cdot wL2 \cdot wR1+ 
w_u(0,1;0,1)\cdot w_u(1,1;1,1) \cdot w_v^c(1,0;0,1) \cdot w_v^c(0,1;1,0)  \cdot wL1 \cdot wR2 ;$
\item $d_l - c_l = n \geq 2: \mathcal{W}({\bf c}, {\bf d};\{l\}) =  w_u(0,1;1,0)\cdot w_u(1,0;0,1)\cdot w_v^c(1,0;1,0)  \cdot wL2 \cdot wR1 + w_u(0,1;0,1)\cdot w_u(1,1;1,1) \cdot w_v^c(1,0;0,1) \cdot w_v^c(0,1;1,0) \cdot p^{n-1} \cdot wL1 \cdot wR2  + w_u(0,1;1,0) \cdot w_u(1,0;0,1) \cdot w_u(0,1;0,1) \cdot w_v^c(1,0;0,1)  \cdot w_v^c(0,1;1,0)\cdot \frac{1 - p^{n-1}}{1 - p} \cdot wL1 \cdot wR1. $
\end{itemize}
We let ${\tt BaseWeight}(u,v,c,d, wL1, wL2,wR1,wR2)$ denote the above single interval weight function and with it we define  {\tt Weight}($u,v,{\bf c}, {\bf d}, x, y, \lambda, \mu$) as follows.

\vspace{3mm}
\begin{tabular}{l}
\hline
{\bf Algorithm} {\tt Weight}($u,v,{\bf c}, {\bf d}, x, y, \lambda, \mu$) \\
\hline
 {\bf Input:} $u,v$ - parameters, ${\bf c} = (c_{x}, c_{x+1}, ..., c_{y})$, ${\bf d} = (d_x,d_{x+1},...,d_y)$, $\lambda \in \mathsf{Sign_k}^+$, $\mu \in \mathsf{Sign}_{k-1}^+$.\\
\hspace{5mm}\\
{\tt initialize} w;\\
{\bf if }($x$ == $y$) \\
\hspace{5mm} $w$ =  ${\tt BaseWeight}(u,v,c_x,d_x,1,1,1,1)$;\\
{\bf else }  \\
\hspace{5mm} $s$ := $\lfloor (x+y) /2 \rfloor$; \\
\hspace{5mm} ${\bf c}'$ := $\{c_x,...,c_{s-1}\}$;\hspace{2mm} ${\bf c}''$ := $\{ c_{s+1},..., c_y\}$;\hspace{2mm} ${\bf d}'$ := $\{d_x,...,d_{s-1}\}$;\hspace{2mm} ${\bf d}''$ := $\{ d_{s+1},..., d_y\}$;\\
\hspace{5mm} $wR1$ := {\tt Weight}($u,v,{\bf c}', {\bf d}', x, s-1, \lambda, \mu$);\hspace{2mm} $wR2$ := $wR1$;\\
\hspace{5mm} {\bf if } ($d_{s}$ == $c_{s-1}$) \\
\hspace{12mm} ${\bf c}'$ = $\{ c_x,...,c_{s-2}, c_{s-1} +1\}$; \hspace{2mm} $wR2$ = {\tt Weight}($u,v,{\bf c}', {\bf d}', x, s-1, \lambda ,\mu$);\\
\hspace{5mm} {\bf end}\\
\hspace{5mm} $wL1$ := {\tt Weight}($u,v,{\bf c}'', {\bf d}'', s+1, y$);\hspace{2mm} $wL2$ := $wL1$;\\
\hspace{5mm} {\bf if } ($c_{s}$ == $d_{s+1}$) \\
\hspace{12mm} ${\bf d}'' = \{ d_{s+1}- 1,d_{s+2},...,d_y\}$; \hspace{2mm} $wL2$ = {\tt Weight}($u,v,{\bf c}'', {\bf d}'', s+1, y, \lambda, \mu$);\\
\hspace{5mm} {\bf end}\\
\hspace{5mm} $w$ = {\tt BaseWeight}($u,v,c_s,d_s, bLL, bRL, bLM, bRM,wL1,wL2,wR1,wR2$);\\
{\bf end}\\
{\bf Output:} $w$ \\
\hline
\end{tabular}

\vspace{3mm}

\subsection{Simulations}
In this section we use the sampling algorithm developed above to produce some simulations. We will be interested in demonstrating that there is a limit shape for the six-vertex model that we have considered. In addition, we will provide some empirical evidence supporting the validity of Theorem \ref{theorem2}.\\

For the simulations we fix $N = 100$ and consider different choices for $q,u$ and $v$. From Theorem \ref{theorem2} we know that $Y_1^1$ asymptotically looks like $a\cdot M$, with $a$ as in Theorem \ref{theorem1}. We pick the parameter $M$ in our simulations so that $a\cdot M$ is roughly $N/2$. The results are summarized in Figures \ref{S6_1} and \ref{S6_2} (see also Figure \ref{S1_5} in Section \ref{Section1}). As discussed in Section \ref{Section1}, there is a macroscopic frozen region, made of $(0,1;0,1)$ vertices in the bottom left corner and another one, made of $(1,1;1,1)$ vertices in the top left corner. The two regions are separated by a disordered region containing all six types of vertices. It would be interesting to see if the methods of this paper can be utilized to rigorously confirm the existence of a limit shape, and to find suitable parametrizations for it.\\

\begin{figure}[h]
\centering
\scalebox{0.96}{\includegraphics{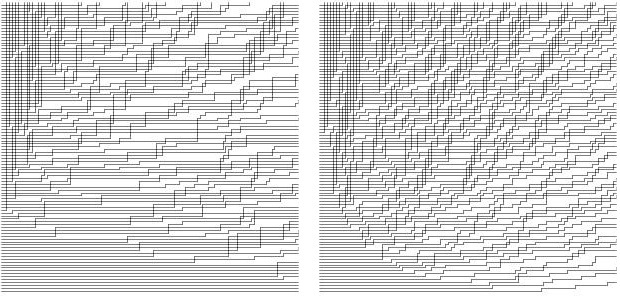}}
\caption{Random paths in $\mathcal{P}_N$, sampled according to $\mathbb{P}_{u,v}^{N,M}$ with $N = 100$. For the left picture $s^{-2} = q = 0.5$, $u = 5$, $v = 0.1$ and $M = 100$; for the right $s^{-2} = q = 0.5$, $u = 2$, $v = 0.1$ and $M = 1000$.  }
\label{S6_1}
\end{figure}

\begin{figure}[h]
\centering
\scalebox{0.96}{\includegraphics{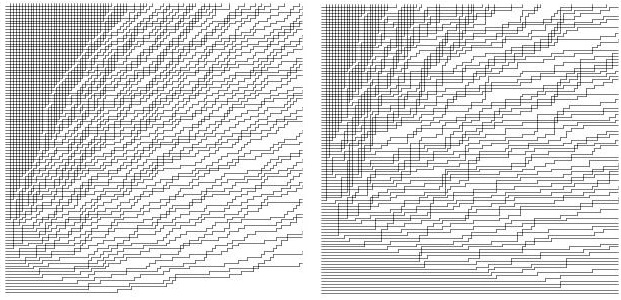}}
\caption{Random paths in $\mathcal{P}_N$, sampled according to $\mathbb{P}_{u,v}^{N,M}$ with $N = M = 100$. For the left picture $s^{-2} = q = 0.25$, $u = 2.5$ and $v = 0.25$; for the right $s^{-2} = q = 0.25$, $u = 5$, and $v = 0.1$ . }
\label{S6_2}
\end{figure}

A particular implication of Theorem \ref{theorem2} is that $ \frac{1}{c\sqrt{M}}\left( Y_1^1- aM \right)$ converges to the standard Gaussian distribution as $M \rightarrow \infty$. In what follows, we provide some numerical simulations supporting this fact. We took $1000$ samples from $\mathbb{P}_{u,v}^{N,M}$ with $N = M = 200$ and different values for $q,u,v$, and calculated $\frac{1}{c\sqrt{M}}\left( Y_1^1- aM \right)$. The empirical distribution of the samples is compared with the standard normal cdf, and the results are given in Figure \ref{S6_3}. As can be seen, the distributions appear to be quite close, as is expected.

\begin{figure}[h]
\centering
\scalebox{0.58}{\includegraphics{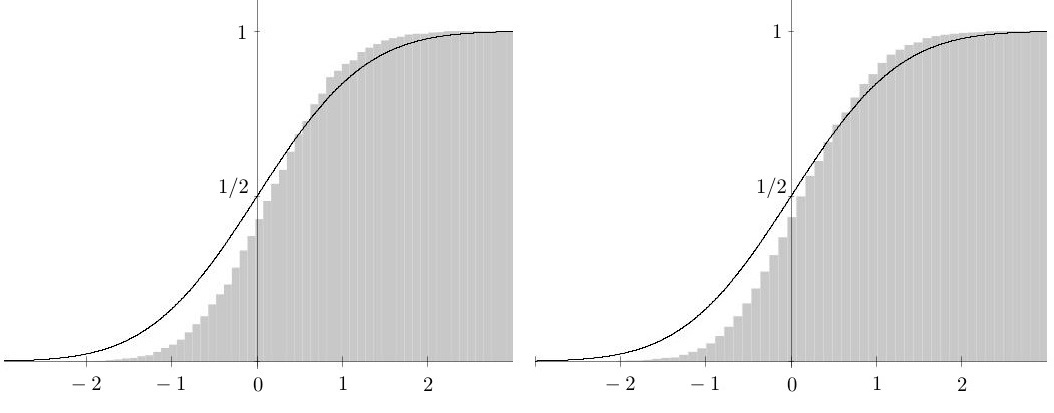}}
\caption{Empirical distribution of $1000$ samples of $\frac{1}{c\sqrt{M}}\left( Y_1^1(\omega)- aM \right)$ with $\omega$ distributed as $\mathbb{P}_{u,v}^{N,M}$ with $N = M = 200$. For the left picture $s^{-2} = q = 0.8$, $u = 1.2$ and $v = 0.8$; for the right $s^{-2} = q = 0.5$, $u = 1.5$, $v = 0.6$. }
\label{S6_3}
\end{figure}

\clearpage

\bibliographystyle{amsplain}
\bibliography{PD}

\end{document}